\documentclass[a4paper]{amsart}

\usepackage{url}
\usepackage{amssymb}
\usepackage{latexsym}
\usepackage{amsfonts}
\usepackage{amsmath}
\usepackage{eucal}
\usepackage{bm}
\usepackage{bbm}
\usepackage{graphicx}
\usepackage[english]{varioref}
\usepackage[nice]{nicefrac}
\usepackage[all]{xy}
\usepackage{amsthm}
\usepackage[colorlinks,linkcolor=blue,urlcolor=cyan,citecolor=red]{hyperref}

\newtheorem{thm}{Theorem}[section]
\newtheorem{prop}[thm]{Proposition}
\newtheorem{lem}[thm]{Lemma}
\newtheorem{defn}[thm]{Definition}

\newtheorem{rmk}[thm]{Remark}
\newtheorem{conj}[thm]{Conjecture}

\title[On ADLV for $Sp_4(L)$]{On affine Deligne-Lusztig varieties for $Sp_4(L)$}
\author[Zhongwei Yang]{Zhongwei Yang}
\address{Southwest Jiaotong University\\
Chengdu, Sichuan, China}
\email{merciyang@home.swjtu.edu.cn}

\begin{document}

\begin{abstract}
  In this paper, we study the emptiness/nonemptiness and the dimension formulas of affine Deligne-Lusztig varieties for $Sp_4(L)$. We mainly calculate the degree of class polynomials for the Iwahori-Hecke algebra of type $\widetilde{C}_2$. Then, give an explicit description on the emptiness/nonemptiness and dimension formulas of affine Deligne-Lusztig varieties for the group $Sp_4(L)$.
\end{abstract}

\maketitle
\tableofcontents

\section{Introduction}
The notion of an affine Deligne-Lusztig variety first introduced by Rapoport in \cite{Ra05}, which is an analogue of a (classical) Deligne-Lusztig variety. An affine Deligne-Lusztig variety $X_{\tilde{w}}(b)$ is ingredient in the study of arithmetic geometry. More precisely, it plays a key role in the study of the reduction of Shimura varieties with Iwahori level structure. Namely, it is closely related to the intersection of the Newton stratum and the Kottwitz-Rapoport stratum. There are two important stratifications on the special fiber of a Shimura variety: The Newton stratification indexed by certain $\sigma$-conjugacy classes $[b]\subset G(L)$; The Kottwitz-Rapoport stratification indexed by some $\tilde{w}\in\widetilde{W}$. The intersection of the Newton stratum associated with $[b]$ and the Kottwitz-Rapoport stratum associated with $\tilde{w}$ is related to the affine Deligne-Lusztig variety $X_{\tilde{w}}(b)$, see \cite{HR17}. Inspired by the study of Shimura varieties, one might be interested in investigating the emptiness/nonemptiness and dimension formulas of the affine Deligne-Lusztig varieties. Although these fundamental questions have been studied by many people (see \cite{GHKR10}, \cite{GH10}, \cite{GHN15}, \cite{He14}, \cite{He21}, \cite{MST19}, \cite{R02}, \cite{R04}, \cite{Y16} for example), we still don't have a complete answer for these questions on $X_{\tilde{w}}(b)$ with arbitrary $b$ and $\tilde{w}$. Recently, X.~He established the nonemptiness and the dimension formula of $X_{\tilde{w}}(b)$ for any $b\in G(L)$ and most $\tilde{w}\in\widetilde{W}$ (see \cite{He21}).

We are inspired by a remarkable result, the ``Dimension $=$ Degree" Theorem (Theorem \ref{DimDeg}), discovered by He in \cite{He14}, which points out that the emptiness/nonemptiness and dimension formulas of affine Deligne-Lusztig varieties can be obtained from class polynomials for the corresponding affine Hecke algebras.

In this paper, we calculate class polynomials for the Iwahori-Hecke algebra of type $\widetilde{C}_2$. And then, we use results on class polynomials to give an explicit description on the emptiness/nonemptiness and dimension formulas of the affine Deligne-Lusztig varieties for $Sp_4(L)$.

Let's recall the definition of an affine Deligne-Lusztig variety. Let $\mathbb{F}_q$ be a finite field with $q$ element, $\textbf{k}$ be its algebraic closure and $L=\textbf{k}((\epsilon))$ be the field of the Laurent series. Let $\sigma$ be the standard Frobenius automorphism on $L$, $G$ be a connected reductive group over $\mathbb{F}_q$. The automorphism $\sigma$ induces an automorphism on the loop group $G(L)$ which will also be denoted by $\sigma$. Let $I$ be a $\sigma$-stable Iwahori subgroup of $G(L)$ and $\widetilde{W}$ be the corresponding Iwahori-Weyl group. We have the Iwahori-Burhat decomposition $G(L)=\bigsqcup_{\tilde{w}\in\widetilde{W}}I\dot{\tilde{w}}I$, where $\dot{\tilde{w}}\in G(L)$ is a representative of $\tilde{w}\in\widetilde{W}$. For any $\tilde{w}\in\widetilde{W}$ and $b\in G(L)$, the \emph{\textbf{affine Deligne-Lusztig variety}} $X_{\tilde{w}}(b)$ is defined as $$X_{\tilde{w}}(b)=\{gI\in G(L)/I | g^{-1}b\sigma(g)\in I\dot{\tilde{w}}I\}.$$

\section{Preliminary}\label{Sect2}
\subsection{The Iwahori-Weyl group}\label{IWgp}
We keep the notations: $\mathbb{F}_q,\ \textbf{k},\ L$ and $\sigma$ as before. Assume that $F=\mathbb{F}_q((\epsilon))$ is the field of Laurent series. Let $G$ be a connected reductive reductive group over $F$ and splits over a tamely ramified extension of $L$. Let $S\subset G$ be a maximal $L$-split torus over $F$, $T=Z_G(S)$ be its centralizer and $N$ be the normalizer of $T$.
\begin{defn}
The algebraic loop group $LG$ associated with $G$ is the ind-group scheme over $\textbf{\emph{k}}$ that represents the functor $$R\longmapsto LG(R)=G(R((\epsilon)))$$on the category of $\textbf{\emph{k}}$-algebras.
\end{defn}
Let $\mathbb{A}$ be the \emph{\textbf{apartment}} of the loop group $G(L)$ corresponding to $S$ and $\mathfrak{a}_C$ be a $\sigma$-invariant alcove in $\mathbb{A}$. Let $I\subset G(L)$ be the Iwahori subgroup corresponding to $\mathfrak{a}_C$ over $L$ and $\widetilde{\mathbb{S}}$ be the set of simple reflections at the walls of $\mathfrak{a}_C$.
\begin{defn}
The \textbf{finite Weyl group} $W$ associated with $S$ is $W=N(L)/T(L)$, and the \textbf{Iwahori-Weyl group} $\widetilde{W}$ associated with $S$ is $\widetilde{W}=N(L)/T(L)_1$, here $T(L)_1$ is the unique parahoric subgroup of $T(L)$.
\end{defn}
The Weyl group $W$ is a Coxeter group (see \cite{Bo02} for definition). The Iwahori-Weyl group $\widetilde{W}$ is in bijection with $I\backslash G(L)/I$, and we have the \emph{Iwahori-Bruhat decomposition} $G(L)=\bigsqcup_{\tilde{w}\in\widetilde{W}}I\dot{\tilde{w}}I$, here $\dot{\tilde{w}}\in G(L)$ is any representative of $\tilde{w}\in\widetilde{W}$.

Let $\Gamma=\text{Gal}(\bar{L}/L)$ and $P$ be the $\Gamma$-coinvariants of $X_*(T)$. Following \cite{HR08}, by choosing a special vertex in $\mathbb{A}$ we identify $T(L)/T(L)_1$ with $P$. And we have a split short exact sequence $1\longrightarrow P\longrightarrow\widetilde{W}\longrightarrow W\longrightarrow 1$ and a semi-direct product $\widetilde{W}=P\rtimes W$. The automorphism $\sigma$ on $G(L)$ induces an automorphism on $\widetilde{W}$, which will be denoted by $\delta$. $\delta$ gives a bijection on $\widetilde{\mathbb{S}}$. We choose a special vertex in $\mathbb{A}$ such that the previous split short exact sequence is preserved by $\delta$. So $\delta$ induces an automorphism on $W$ and we denote it by the same symbol.

Let $\Phi$ be the set of roots of $(G,\ S)$ over $L$, $\Phi^+$ be the set of positive roots and $\Phi_a$ be the set of affine roots. Let $\mathbb{S}$ be the set of simple roots in $\Phi$. We identify $\mathbb{S}$ with the set of simple reflections in $W$, and thus $\mathbb{S}$ is a $\delta$-stable proper subset of $\widetilde{\mathbb{S}}$. Let $G_1$ be the subgroup of $G(L)$ generated by all parahoric subgroups and let $N_1=N(L)\cap G_1$. By \cite{BT84}, the quadruple $(G_1,\ I,\ N_1,\ \widetilde{\mathbb{S}})$ is a double Tits system with affine Weyl group $W_a=N_1/(N(L)\cap I)$. We identify $W_a$ with the Iwahori-Weyl group of the simply connected cover $G_{\text{sc}}$ of the derived group $G_{\text{der}}$ of $G$. Let $T_{\text{sc}}$ be the maximal torus of $G_{\text{sc}}$ given by $T$. Thus we have $W_a=X_*(T_{\text{sc}})_{\Gamma}\rtimes W$. This shows that a reduced root system $\Delta$ exists such that $W_a=Q^\vee(\Delta)\rtimes W(\Delta)$, where $Q^\vee(\Delta)$ is the coroot lattice of $\Delta$. In the following, we will write $Q$ for $Q^\vee(\Delta)$ and identify $Q$ with $X_*(T_{\text{sc}})_{\Gamma}$ and $W$ with $W(\Delta)$.

\subsection{Class polynomials}\label{Cp}
In general, the Iwahori-Weyl group $\widetilde{W}$ is not a Coxeter group, it is a quasi-Coxeter group in the sense that $\widetilde{W}=W_a\rtimes\Omega$, where $\Omega$ is the subgroup of $\widetilde{W}$ consisting length 0 elements and the length of $\tilde{w}\in\widetilde{W}$ (denoted by $\ell(\tilde{w})$) is defined to be the number of ``affine root hyperplanes" between $\tilde{w}(\mathfrak{a}_C)$ and $\mathfrak{a}_C$ in $\mathbb{A}$.

Now, we recall the Hecke algebra associated to a quasi-Coxeter group by analogy the definition of a Hecke algebra associated with a Coxeter system.
\begin{defn}
Let $\widetilde{H}$ be the Hecke algebra associated with $\widetilde{W}$, i.e., $\widetilde{H}$ is the associative $A=\mathbb{Z}[v,\ v^{-1}]$-algebra with basis $T_{\tilde{w}}$ for $\tilde{w}\in\widetilde{W}$ and the multiplication is given by$$T_{\tilde{x}}T_{\tilde{y}}=T_{\tilde{x}\tilde{y}},\ \ \ \ \emph{if}\ \ \ell(\tilde{x})+\ell(\tilde{y})=\ell(\tilde{x}\tilde{y});$$$$(T_s-v)(T_s+v^{-1})=0,\ \ \ \ \emph{for}\ \ \ s\in\widetilde{\mathbb{S}}.$$
\end{defn}
Note that the map $T_{\tilde{w}}\mapsto T_{\delta(\tilde{w})}$ defines an $A$-algebra automorphism of $\widetilde{H}$, and we still denote it as $\delta$.

For any $\tilde{w},\ \tilde{w}'\in\widetilde{W}$ are said to be $\delta$-conjugate if $\tilde{w}'=\tilde{x}\tilde{w}\delta(\tilde{x})^{-1}$ for some $\tilde{x}\in\widetilde{W}$. For $\tilde{w},\ \tilde{w}'\in \widetilde{W}$ and $s_i\in\widetilde{\mathbb{S}}$, we write $\tilde{w}\xrightarrow[]{s_i}_{\delta}\tilde{w}'$ if $\tilde{w}'=s_i\tilde{w}s_{\delta(i)}$ and $\ell(\tilde{w})'\leqslant\ell(\tilde{w})$. We write $\tilde{w}\xrightarrow[]{}_{\delta}\tilde{w}'$ if there is a sequence $\tilde{w}=\tilde{w}_0,\ \tilde{w}_1,\ \cdots,\ \tilde{w}_n=\tilde{w}'$ of elements in $\widetilde{W}$ such that for any $k$, $\tilde{w}_{k-1}\xrightarrow[]{s_i}_{\delta}\tilde{w}_k$ for some $s_i\in\tilde{\mathbb{S}}$. We write $\tilde{w}\approx_{\delta}\tilde{w}'$ if $\tilde{w}\xrightarrow[]{}_{\delta}\tilde{w}'$ and $\tilde{w}'\xrightarrow[]{}_{\delta}\tilde{w}$. Write $\tilde{w}\tilde{\approx}_{\delta}\tilde{w}'$ if $\tilde{w}\approx_{\delta}\tau\tilde{w}'\delta(\tau)^{-1}$ for some $\tau\in\Omega$.

We say that $\tilde{w},\ \tilde{w}'\in \widetilde{W}$ are elementarily strongly $\delta$-conjugate if $\ell(\tilde{w})=\ell(\tilde{w}')$ and there exists $\tilde{x}\in\widetilde{W}$ such that $\tilde{w}'=\tilde{x}\tilde{w}\delta(\tilde{x})^{-1}$ and $\ell(\tilde{x}\tilde{w})=\ell(\tilde{x})+\ell(\tilde{w})$ or $\ell(\widetilde{w}\delta(\widetilde{x})^{-1})=\ell(\tilde{x})+\ell(\tilde{w})$. And we call that $\tilde{w},\ \tilde{w}'$ are strongly $\delta$-conjugate if there is a sequence $\tilde{w}=\tilde{w}_0,\ \tilde{w}_1,\ \cdots,\ \tilde{w}_n=\tilde{w}'$ of elements in $\widetilde{W}$ such that for any $i$, $\tilde{w}_{i-1}$ is elementarily strongly $\delta$-conjugate to $\tilde{w}_i$. We write $\tilde{w}\tilde{\thicksim}_{\delta}\tilde{w}'$ if $\tilde{w}$ and $\tilde{w}'$ are strongly $\delta$-conjugate.

In \cite{HN14}, He and Nie proved that minimal length elements $w_{\mathbb{O}}$ of any $\delta$-conjugacy class $\mathbb{O}$ of $\widetilde{W}$ satisfy some special properties, generalizing the results of Geck and Pfeiffer \cite{GP93} on finite Weyl groups. These properties play a key role in the study of affine Deligne-Lusztig varieties and affine Hecke algebras. In \cite{He14} and \cite{HN14}, it is showed that for any $\delta$-conjugacy class $\mathbb{O}$, we can fix a minimal length representative $w_{\mathbb{O}}$ and the image of $T_{w_{\mathbb{O}}}$ in $\tilde{H}/[\tilde{H},\ \tilde{H}]_{\delta}$ where $[\tilde{H},\ \tilde{H}]_{\delta}$ is the $A$-submodule (we regard $\tilde{H}$ as a left $A$-module) of $\tilde{H}$ generated by all $\delta$-commutators (i.e. for any $h,\ h'\in\tilde{H}$, $[h,\ h']_{\delta}=hh'-h'\delta{(h)}$). Moreover, $T_{w_{\mathbb{O}}}$ is independent of the choice of $w_{\mathbb{O}}$ and forms a basis of  $\tilde{H}/[\tilde{H},\ \tilde{H}]_{\delta}$. It is proved in \cite[Theorem 6.7]{HN14} that

\begin{thm}[He-Nie]
The elements $T_{\tilde{w}_{\mathbb{O}}}$ form an $A$-basis of $\widetilde{H}/[\widetilde{H},\ \widetilde{H}]_{\delta}$, here $\mathbb{O}$ runs over all the $\delta$-conjugacy classes of $\widetilde{W}$.
\end{thm}

From now on, we denote $T_{\tilde{w}_{\mathbb{O}}}$ as $T_{\mathbb{O}}$ for simplicity. For any $\tilde{w}\in\widetilde{W}$ and a $\delta$-conjugacy class $\mathbb{O}$, there exists a unique $f_{\tilde{w},\mathbb{O}}\in A$ such that\[T_{\tilde{w}}\equiv\sum_{\mathbb{O}}f_{\tilde{w},\mathbb{O}}T_{\mathbb{O}}\ \ \mod\ [\widetilde{H},\ \widetilde{H}]_{\delta}.\]

$f_{\tilde{w},\mathbb{O}}$ is a polynomial in $\mathbb{Z}[v-v^{-1}]$ with nonnegative coefficient. This is called the \textbf{class polynomial} attached to $\tilde{w}$ and $\mathbb{O}$, and it can be constructed inductively as follows:

If $\tilde{w}$ is a minimal length element in a $\delta$-conjugacy class of $\widetilde{W}$, then we set \[f_{\tilde{w},\mathbb{O}}=\left\{
                                                                                                                                                 \begin{array}{ll}
                                                                                                                                                   1, & \text{if}\ \tilde{w}\in\mathbb{O},\\
                                                                                                                                                   0, & \text{if}\ \tilde{w}\notin\mathbb{O}.
                                                                                                                                                 \end{array}
                                                                                                                                               \right.
\]
If $\tilde{w}$ is not a minimal length element in its $\delta$-conjugacy class and for any $\tilde{w}'\in\tilde{W}$ with $\ell(\tilde{w}')<\ell(\tilde{w})$, $f_{\tilde{w}',\mathbb{O}}$ is constructed. By \cite{HN14}, there exists $\tilde{w}_1\approx_{\delta}\tilde{w}$ and $s_i\in\widetilde{\mathbb{S}}$ such that $\ell(s_i\tilde{w}_1s_{\delta(i)})<\ell(\tilde{w}_1)=\ell(\tilde{w}).$ In this case, $\ell(s_i\tilde{w})<\ell(\tilde{w})$ and we define $f_{\tilde{w}, \mathbb{O}}$ as $$f_{\tilde{w},\mathbb{O}}=(v-v^{-1})f_{s_i\tilde{w}_1,\mathbb{O}}+f_{s_i\tilde{w}_1s_{\delta(i)},\mathbb{O}}.$$

\subsection{The ``Dimension \texorpdfstring{$=$}. Degree" Theorem}\label{DimDegtm}
We call an element $\tilde{w}\in\widetilde{W}$ a $\delta$-\emph{straight element} if and only if for any $n\in\mathbb{N}$ we have $\ell(\tilde{w}\delta(\tilde{w})\cdots\delta^{n-1}(\tilde{w}))=n\ell(\tilde{w})$. We call a $\delta$-conjugacy class in $\widetilde{W}$ \emph{straight} if it contains some straight element. We denote by $\cdot_{\sigma}$ the $\sigma$-conjugation action on $G(L)$ and it is defined by that for any $g,\ g'\in G(L)$, $g\cdot_{\sigma}g'=gg'\sigma(g)^{-1}$. We have the following Kottwitz's classification of $\sigma$-conjugacy classes $B(G)$ on $G(L)$.

\begin{thm}[Kottwitz, He]
For any straight $\delta$-conjugacy class $\mathbb{O}$ of $\widetilde{W}$, we fix a minimal length representative $\tilde{w}_{\mathbb{O}}$. Then\[G(L)=\bigsqcup_{\mathbb{O}}G(L)\cdot_{\sigma}\tilde{w}_{\mathbb{O}},\]here $\mathbb{O}$ runs over all the straight $\delta$-conjugacy classes of $\widetilde{W}$.
\end{thm}

Let $(P/Q)_{\delta}$ be the $\delta$-coinvariants on $P/Q$, let \[\kappa:\widetilde{W}\longrightarrow\widetilde{W}/W_a\cong P/Q\longrightarrow(P/Q)_{\delta}\] be the natural projection. We call $\kappa$ the \emph{Kottwitz map}.

Let $P_{\mathbb{Q}}=P\otimes_{\mathbb{Z}}\mathbb{Q}$ and $P_{\mathbb{Q}}/W$ be the quotient of $P_{\mathbb{Q}}$ by the natural action of $W$. We can identify $P_{\mathbb{Q}}/W$ with $P_{\mathbb{Q},+}=\{\chi\in P_{\mathbb{Q}}\mid \alpha(\chi)\geqslant0,\ for\ all\ \alpha\in \Phi^+\}.$ Let $P^{\delta}_{\mathbb{Q},+}$ be the set of $\delta$-invariant points in $P_{\mathbb{Q},+}$.

Since the image of $W\rtimes\langle\delta\rangle$ in $Aut(W)$ is a finite group, for each $\tilde{w}=t^{\chi}w\in\widetilde{W}$, there exists $n\in\mathbb{N}$ such that $\delta^n=1$ and $w\delta(w)\delta^2(w)\cdots\delta^{n-1}(w)=1$. Then $\tilde{w}\delta(\tilde{w})\delta^2(\tilde{w})\cdots\delta^{n-1}(\tilde{w})=t^{\lambda}$ for some $\lambda\in P$. Let $\nu_{\tilde{w}}=\lambda/n\in P_{\mathbb{Q}}$ and $\bar{\nu}_{\tilde{w}}$ be the corresponding element in $P_{\mathbb{Q},+}$. Note that $\nu_{\tilde{w}}$ is independent of the choice of $n$. Let $\bar{\nu}_{\tilde{w}}$ be the unique element in $P_{\mathbb{Q},+}$ that lies in the $W$-orbit of $\nu_{\tilde{w}}$. Since $t^{\lambda}=\tilde{w}t^{\delta(\lambda)}\tilde{w}^{-1}=t^{w\delta(\lambda)}$, $\bar{\nu}_{\tilde{w}}\in P^{\delta}_{\mathbb{Q},+}$. We call the map $\widetilde{W}\longrightarrow P^{\delta}_{\mathbb{Q},+}$ with $\tilde{w}\mapsto\bar{\nu}_{\tilde{w}}$ the \emph{Newton map}. We define \[f:\widetilde{W}\longrightarrow P^{\delta}_{\mathbb{Q},+}\times(P/Q)_{\delta}\ \ \ \ \ by\ \ \ \  \tilde{w}\longmapsto(\bar{\nu}_{\tilde{w}},\ \kappa(\tilde{w})).\]Follows \cite{Ko97}, $f$ 
is constant on each $\delta$-conjugacy class of $\widetilde{W}$. We denote the image of the map by $B(\widetilde{W},\ \delta)$.

\begin{rmk}
(1) There is a bijection between the set of $\sigma$-conjugacy classes $B(G)$ of $G(L)$ and the set of straight conjugacy classes of $\widetilde{W}$.

(2) Any $\sigma$-conjugacy class of $G(L)$ contains a representative in $\widetilde{W}$. Moreover, the map $f:\widetilde{W}\longrightarrow P^{\delta}_{\mathbb{Q},+}\times(P/Q)_{\delta}$ is in fact the restriction of a map defined on $G(L)$ as $b\mapsto (\bar{\nu}_{b},\ \kappa(b))$ and we call $\bar{\nu}_{b}$ the \emph{Newton vector} of $b$.
\end{rmk}

The ``Dimension $=$ Degree'' Theorem is a main result in \cite{He14} which is

\begin{thm}[He]\label{DimDeg}
Let $b\in G(L)$ and $\tilde{w}\in\widetilde{W}$. Then\[\dim(X_{\tilde{w}}(b))=\max_\mathbb{O}\frac{1}{2}(\ell(\tilde{w})+\ell(\mathbb{O})+\deg(f_{\tilde{w},\mathbb{O}}))-\langle\bar{\nu}_b,2\rho\rangle,\]
here $\mathbb{O}$ runs over $\delta$-conjugacy classes of $\widetilde{W}$ with $f(\mathbb{O})=f(b)$ and $\ell(\mathbb{O})$ is the length of any minimal length element in $\mathbb{O}$, and $\rho$ is the half sum of all the positive roots of $\Phi^+$.
\end{thm}

\begin{rmk}
We use the convention that the dimension of an empty variety and the degree of a zero polynomial are both $-\infty$.
\end{rmk}

Theorem \ref{DimDeg} shows that the dimension and emptiness/nonemptiness pattern of affine Deligne-Lusztig varieties $X_{\tilde{w}}(b)$ only depend on the data $(\widetilde{W}, \delta, \tilde{w}, f(b))$ and thus independent of the choice of $G$. And this theorem inspired our study of the class polynomials of Iwahori-Hecke algebras.

\section{Class polynomials for the Iwahori-Hecke algebra of type \texorpdfstring{$\widetilde{C}_2$}.}\label{Classcomput}
\subsection{Conjugacy classes}
Let $\widetilde{W}$ be the Iwahori-Weyl group of $Sp_4(L)$. Then $\widetilde{W}=P\rtimes W=W_a\rtimes\Omega$ where $P$ is the coweight lattice, $W$ is the Weyl group and $W_a$ is the affine Weyl group generated by $\mathbb{S}=\{s_1,\ s_2\}$ and $\tilde{\mathbb{S}}=\{s_0,\ s_1,\ s_2\}$ respectively. Here $s_1,\ s_2$ are the simple reflections correspond to the simple roots $\alpha_1$ and $\alpha_2$ respectively, and $s_0$ is the simple reflection corresponds to $\alpha_0=2\alpha_1+\alpha_2$.

Set $\epsilon_1=\alpha_1+\frac{1}{2}\alpha_2,\ \epsilon_2=\frac{1}{2}(\alpha_1+\alpha_2)$, then the coweight lattice $P=\mathbb{Z}\epsilon_1\oplus\mathbb{Z}\epsilon_2=\{m\epsilon_1+n\epsilon_2|m,\ n\in\mathbb{Z}\}$. Let $\Lambda_1=\epsilon_1=\alpha_1+\frac{1}{2}\alpha_2,\ \Lambda_2= 2\epsilon_2=\alpha_1+\alpha_2$, and then the coroot lattice $Q=\mathbb{Z}\Lambda_1\oplus\mathbb{Z}\Lambda_2=\{m\Lambda_1+n\Lambda_2|m,\ n\in\mathbb{Z}\}$. We have $W_a=Q\rtimes W$.

It will be convenience for us to give a partition of $Q$. Let $\Xi_1=\{-2m\Lambda_1+(m+n+1)\Lambda_2,\ -(2m+1)\Lambda_1+(m+n+1)\Lambda_2| m,\ n\in\mathbb{N}\}$, $\Xi_2=\{(m+1)\Lambda_1+n\Lambda_2|m,\ n\in\mathbb{N}\}$, $\Xi_3=\{(m+n+2)\Lambda_1-(n+1)\Lambda_2|m,\ n\in\mathbb{N}\}$, $\Xi_4=\{(n-m+1)\Lambda_1-(n+1)\Lambda_2|m,\ n\in\mathbb{N}\}$, and $\Xi_5=\{-(m+2n)\Lambda_1+n\Lambda_2|m,\ n\in\mathbb{N}\}$. Then $Q=\bigsqcup_{i=1}^{5}\Xi_i$.

The subgroup $\Omega$ is generated by one element $\tau$ such that $\tau^2=1$ (1 is the identity in $\widetilde{W}$) and $\tau s_0 \tau^{-1}=s_2,\ \tau s_1 \tau^{-1}=s_1,\ \tau s_2 \tau^{-1}=s_0$. The element $\tau$ will be written by $t^{\epsilon_2}s_2s_1s_2$, while the simple reflection $s_0$ will be written as $t^{\Lambda_1}s_1s_2s_1$.

For any $\tilde{w}\in\widetilde{W}$, the conjugacy class of $\tilde{w}$ is $\widetilde{W}\cdot\tilde{w}=\{\tilde{w}'\tilde{w}\tilde{w}'^{-1}|\ \forall\ \tilde{w}'\in\widetilde{W}\}$. We will write an element in $\widetilde{W}$ as $t^\lambda w\tau^r$, where $\lambda\in Q$, $w\in W$ and $r\in\{0,\ 1\}$. For the loop group $Sp_4(L)$, the Iwahori-Weyl group $\widetilde{W}$ is the disjoint union of the affine Weyl group $W_a$ and $W_a\tau=\{t^\lambda w\tau|\lambda\in Q,\ \text{and}\ w\in W\}$. It is easy to see that if $\tilde{w}\in W_a$, then $\widetilde{W}\cdot\tilde{w}\subset W_a$. And if $\tilde{w}\in W_a\tau$, then $\widetilde{W}\cdot\tilde{w}\subset W_a\tau$. Now, we give all the conjugacy classes in $W_a$ and in $W_a\tau$.

\begin{lem}\label{s12}
Let $\mathbb{O}_2=\{t^{\lambda}s_1s_2,\ t^{\lambda}s_2s_1|\lambda\in Q\}$. Then $\widetilde{W}\cdot(s_1s_2)=\mathbb{O}_2$.
\end{lem}
\begin{proof}
Let $\tilde{w}=s_1s_2\in W_a$. Then \[\tilde{w}'\tilde{w}\tilde{w}'^{-1}=\left\{
                                                             \begin{array}{ll}
                                                               t^{\lambda-ws_1s_2w^{-1}(\lambda)}ws_1s_2w^{-1}, & \text{if}\ \tilde{w}'=t^\lambda w\\
                                                               t^{\lambda+w(\Lambda_2-\Lambda_1)-ws_2s_1w^{-1}(\lambda)}ws_2s_1w^{-1}, & \text{if}\ \tilde{w}'=t^\lambda w\tau .
                                                             \end{array}
                                                           \right.\]
This shows that $\widetilde{W}\cdot(s_1s_2)\subseteq\mathbb{O}_2$. Now, for any $\tilde{w}\in\mathbb{O}_2$, by induction on $\ell(\tilde{w})$, we show that $\tilde{w}$ is conjugate to $s_1s_2$. If $\tilde{w}\in\mathbb{O}_2$ with $\ell(\tilde{w})=2$, $\tilde{w}$ is obviously conjugate to $s_1s_2$. We assume that for any $\tilde{w}_1\in\mathbb{O}_2$ and $\ell(\tilde{w}_1)<\ell(\tilde{w})$, $\tilde{w}_1$ is conjugate to $s_1s_2$. For any $\tilde{w}\in\mathbb{O}_2$, then $\tilde{w}$ is of form $t^\lambda s_1s_2$ or $t^\lambda s_2s_1$ with $\lambda\in Q$. Since $\tau t^\lambda s_2s_1\tau^{-1}=t^{\Lambda_1+s_2s_1s_2(\lambda)}s_1s_2$ and $\ell(t^\lambda s_2s_1)=\ell(\tau t^\lambda s_2s_1\tau^{-1})$, it is sufficient to prove that $\tilde{w}=t^\lambda s_1s_2$ is conjugate to $s_1s_2$.

If $\lambda\in\Xi_1$ and one can take $\tilde{w}'=s_1$, then $\tilde{w}_1=\tilde{w}'t^\lambda s_1s_2\tilde{w}'^{-1}=t^{s_1(\lambda)}s_2s_1$ and $\ell(t^{s_1(\lambda)}s_2s_1)=\ell(t^\lambda s_1s_2)-2$. If $\lambda\in\Xi_2$ and $\lambda\neq(m+1)\Lambda_1$ ($m\in\mathbb{N}$), let $\tilde{w}'=s_0s_1=t^{\Lambda_1}s_1s_2$, then $\tilde{w}_1=\tilde{w}'t^\lambda s_1s_2\tilde{w}'^{-1}=t^{2\Lambda_1-\Lambda_2+s_1s_2(\lambda)}s_1s_2$ and $\ell(t^{2\Lambda_1-\Lambda_2+s_1s_2(\lambda)}s_1s_2)=\ell(t^\lambda s_1s_2)-2$. If $\lambda=(m+1)\Lambda_1$ ($m\in\mathbb{N}_+$), let $\tilde{w}'=s_1s_0=t^{\Lambda_2-\Lambda_1}s_2s_1$, then $\tilde{w}_1=\tilde{w}'t^\lambda s_1s_2\tilde{w}'^{-1}=t^{\Lambda_2+s_2s_1(\lambda)}s_1s_2$ and $\ell(t^{\Lambda_2+s_2s_1(\lambda)}s_1s_2)=\ell(t^\lambda s_1s_2)-2$. If $\lambda=\Lambda_1$, $\tilde{w}=t^\lambda s_1s_2=t^{\Lambda_1}s_1s_2=s_0s_1$, it is conjugate to $s_1s_2$. If $\lambda\in\Xi_3$, let $\tilde{w}'=s_0=t^{\Lambda_1}s_1s_2s_1$, then $\tilde{w}_1=\tilde{w}'t^\lambda s_1s_2\tilde{w}'^{-1}=t^{2\Lambda_1-\Lambda_2+s_1s_2s_1(\lambda)}s_2s_1$ and $\ell(t^{2\Lambda_1-\Lambda_2+s_1s_2s_1(\lambda)}s_2s_1)=\ell(t^\lambda s_1s_2)-2$. If $\lambda\in\Xi_4$, let $\tilde{w}'=s_2$, then $\tilde{w}_1=\tilde{w}'t^\lambda s_1s_2\tilde{w}'^{-1}=t^{s_2(\lambda)}s_2s_1$ and $\ell(t^{s_2(\lambda)}s_2s_1)=\ell(t^\lambda s_1s_2)-2$. If $\lambda\in\Xi_5$ and $\lambda\neq-m\Lambda_1$ ($m\in\mathbb{N}$), let $\tilde{w}'=s_2s_1$, then $\tilde{w}_1=\tilde{w}'t^\lambda s_1s_2\tilde{w}'^{-1}=t^{s_2s_1(\lambda)}s_1s_2$ and $\ell(t^{s_2s_1(\lambda)}s_1s_2)=\ell(t^\lambda s_1s_2)-2$. If $\lambda=-m\Lambda_1$ ($m\in\mathbb{N}_+$), let $\tilde{w}'=s_1s_2$, then $\tilde{w}_1=\tilde{w}'t^\lambda s_1s_2\tilde{w}'^{-1}=t^{s_1s_2(\lambda)}s_1s_2$ and $\ell(t^{s_1s_2(\lambda)}s_1s_2)=\ell(t^\lambda s_1s_2)-2$. Thus, by induction $\tilde{w}_1=\tilde{w}'t^\lambda s_1s_2\tilde{w}'^{-1}$ is conjugate to $s_1s_2$. Thus, $\tilde{w}=t^\lambda s_1s_2$ is conjugate to $s_1s_2$. Hence, $\widetilde{W}\cdot(s_1s_2)=\mathbb{O}_2$.
\end{proof}

\begin{lem}\label{s02}
Let $\mathbb{O}'_2=\{t^{(2m+1)\Lambda_1+n\Lambda_2}s_1s_2s_1s_2| m,\ n\in\mathbb{Z}\}$. Then $\widetilde{W}\cdot(s_0s_2)=\mathbb{O}'_2$.
\end{lem}
\begin{proof}
Since $s_0s_2=s_2s_0=\tau s_0s_2\tau^{-1}=t^{\Lambda_1}s_1s_2s_1s_2$, and for any $\tilde{w}'=t^\lambda w\ \text{or}\ t^\lambda w\tau$ we have $\tilde{w}'s_0s_2\tilde{w}'^{-1}=t^{w(\Lambda_1)+2\lambda}s_1s_2s_1s_2$, so $\widetilde{W}\cdot(s_0s_2)\subseteq\mathbb{O}'_2$.

For any $\tilde{w}\in\mathbb{O}'_2$, we use induction on $\ell(\tilde{w})$ to prove that $\tilde{w}$ is conjugate to $s_0s_2$. If $\tilde{w}=t^{\Lambda_1}s_1s_2s_1s_2$, then $\tilde{w}=s_0s_2$ and $\ell(\tilde{w})=2$, $\tilde{w}$ is obviously conjugate to $s_0s_2$. For arbitrary $\tilde{w}\in\mathbb{O}'_2$, we assume that if $\tilde{w}_1\in\mathbb{O}'_2$ with $\ell(\tilde{w}_1)<\ell(\tilde{w})$, then $\tilde{w}_1$ is conjugate to $s_0s_2$.

For any $\tilde{w}\in\mathbb{O}'_2$, it can be written as $t^{(2m+1)\Lambda_1+(\pm n-m)\Lambda_2}s_1s_2s_1s_2$, where $m\in\mathbb{Z}$ and $n\in\mathbb{N}$. Since $\tau t^{(2m+1)\Lambda_1+(-n-m)\Lambda_2}s_1s_2s_1s_2\tau^{-1}=t^{(2m+1)\Lambda_1+(n-m)\Lambda_2}s_1s_2s_1s_2$ and $\ell(t^{(2m+1)\Lambda_1+(-n-m)\Lambda_2}s_1s_2s_1s_2)=\ell(t^{(2m+1)\Lambda_1+(n-m)\Lambda_2}s_1s_2s_1s_2)$, it is sufficient to consider $\tilde{w}=t^{(2m+1)\Lambda_1+(n-m)\Lambda_2}s_1s_2s_1s_2$ with $m\in\mathbb{Z}$ and $n\in\mathbb{N}$. If $m\leqslant-1$, let $\tilde{w}'=s_1$, then $\tilde{w}_1=\tilde{w}'\tilde{w}\tilde{w}'^{-1}=t^{-(2m+1)\Lambda_1+(m+n+1)\Lambda_2}s_1s_2s_1s_2$ and $\ell(\tilde{w}_1)=\ell(\tilde{w})-2$. If $m=n=0$, then $\tilde{w}=t^{\Lambda_1}s_1s_2s_1s_2=s_0s_2$. If $m>0,\ n\geqslant0$ or $m=0,\ n>0$, let $\tilde{w}'=s_0$, then $\tilde{w}_1=\tilde{w}'\tilde{w}\tilde{w}'^{-1}=t^{-(2n-1)\Lambda_1+(n-m)\Lambda_2}s_1s_2s_1s_2$ and $\ell(\tilde{w}_1)=\ell(\tilde{w})-2$. By induction, $\tilde{w}_1$ is conjugate to $s_0s_2$. Hence, $\tilde{w}$ is conjugate to $s_0s_2$.
\end{proof}

\begin{lem}\label{s1212}
Let $\mathbb{O}_{s_{1212}}=\{t^{2m\Lambda_1+n\Lambda_2}s_1s_2s_1s_2| m,\ n\in\mathbb{Z}\}$. Then $\widetilde{W}\cdot(s_1s_2s_1s_2)=\mathbb{O}_{s_{1212}}$.
\end{lem}
\begin{proof}
The proof is quite similar to that of Lemma \ref{s02}, and we omit it.
\end{proof}

\begin{lem}\label{lambda}
Let $Q_+=\{\lambda\in Q|\alpha(\lambda)\geqslant0,\ \alpha\in\Phi^+\}=\{m\Lambda_1+n\Lambda_2|m,\ n\in\mathbb{N}\}$. For any $\lambda\in Q_+$, let $\mathbb{O}_\lambda=\{t^{w(\lambda)}|w\in W\}$. Then $\widetilde{W}\cdot t^\lambda=\mathbb{O}_\lambda$.
\end{lem}
\begin{proof}
For any $\lambda\in Q_+$ and $w\in W$, $wt^\lambda w^{-1}=t^{w(\lambda)}$. Moreover, $\ell(t^{w(\lambda)})=\ell(t^\lambda)$. The lemma follows from \[\tilde{w}'t^\lambda\tilde{w}'^{-1}=\left\{
                                                             \begin{array}{ll}
                                                               t^{w(\lambda)}, & \text{if}\ \tilde{w}'=t^\chi w\\
                                                               t^{ws_2s_1s_2(\lambda)}, & \text{if}\ \tilde{w}'=t^\chi w\tau .
                                                             \end{array}
                                                           \right.\]
\end{proof}

In the following, let $k$ be any nonnegative integer.

\begin{lem}\label{4k+1}
Let $\mathbb{O}_{4k+1}=\{t^{\pm k(\Lambda_2-\Lambda_1)+m'\Lambda_1}s_1s_2s_1,\ t^{\pm k\Lambda_1+m'(\Lambda_2-\Lambda_1)}s_2|m'\in\mathbb{Z}\}$. Then $\widetilde{W}\cdot (t^{k\Lambda_1}s_2)=\mathbb{O}_{4k+1}$.
\end{lem}
\begin{proof}
Let $\tilde{w}'=t^{m\Lambda_1+n\Lambda_2}w$. Then\[\tilde{w}'t^{k\Lambda_1}s_2\tilde{w}'^{-1}=\left\{
                                                             \begin{array}{ll}
                                                             t^{k\Lambda_1+2n(\Lambda_2-\Lambda_1)}s_2, & \text{if}\ w\in\{1,\ s_2\}\\
                                                             t^{-k\Lambda_1+2n(\Lambda_2-\Lambda_1)}s_2, & \text{if}\ w\in\{s_1s_2s_1,\ s_1s_2s_1s_2\}\\
                                                             t^{k(\Lambda_2-\Lambda_1)+2(m+n)\Lambda_1}s_1s_2s_1, & \text{if}\ w\in\{s_1,\ s_1s_2\}\\
                                                             t^{-k(\Lambda_2-\Lambda_1)+2(m+n)\Lambda_1}s_1s_2s_1, & \text{if}\ w\in\{s_2s_1,\ s_2s_1s_2\}.
                                                             \end{array}
                                                           \right.\]
Let $\tilde{w}'=t^{m\Lambda_1+n\Lambda_2}w\tau$. Then \[\tilde{w}'t^{k\Lambda_1}s_2\tilde{w}'^{-1}=\left\{
                                                             \begin{array}{ll}
                                                             t^{-k(\Lambda_2-\Lambda_1)+(2m+2n+1)\Lambda_1}s_1s_2s_1, &\ \ \ \ \ \ \ \ \text{if}\ w=1\\
                                                             t^{-k\Lambda_1+(2n+1)(\Lambda_2-\Lambda_1)}s_2, &\ \ \ \ \ \ \ \ \text{if}\ w=s_1\\
                                                             t^{k(\Lambda_2-\Lambda_1)+(2m+2n+1)\Lambda_1}s_1s_2s_1, &\ \ \ \ \ \ \ \  \text{if}\  w=s_2\\
                                                             t^{k\Lambda_1+(2n+1)(\Lambda_2-\Lambda_1)}s_2, &\ \ \ \ \ \ \ \ \text{if}\  w=s_1s_2\\
                                                             t^{-k\Lambda_1+(2n-1)(\Lambda_2-\Lambda_1)}s_2, &\ \ \ \ \ \ \ \ \text{if}\ w=s_2s_1\\
                                                             t^{-k(\Lambda_2-\Lambda_1)+(2m+2n-1)\Lambda_1}s_1s_2s_1, &\ \ \ \ \ \ \ \ \text{if}\  w=s_1s_2s_1\\
                                                             t^{k\Lambda_1+(2n-1)(\Lambda_2-\Lambda_1)}s_2, &\ \ \ \ \ \ \ \ \text{if}\ w=s_2s_1s_2\\
                                                             t^{k(\Lambda_2-\Lambda_1)+(2m+2n-1)\Lambda_1}s_1s_2s_1, &\ \ \ \ \ \ \ \ \text{if}\  w=s_1s_2s_1s_2.
                                                             \end{array}
                                                           \right.\]
Thus, the lemma is proved.
\end{proof}

\begin{lem}\label{6k+1}
Let $\mathbb{O}'_{6k+1}=\{t^{\pm k\Lambda_2+m'(2\Lambda_1-\Lambda_2)}s_1,\ t^{\pm k(2\Lambda_1-\Lambda_2)+m'\Lambda_2}s_2s_1s_2|m'\in\mathbb{Z}\}$. Then $\widetilde{W}\cdot (t^{k\Lambda_2}s_1)=\mathbb{O}'_{6k+1}$.
\end{lem}
\begin{proof}
Let $\tilde{w}'=t^{m\Lambda_1+n\Lambda_2}w$. Then\[\tilde{w}'t^{k\Lambda_2}s_1\tilde{w}'^{-1}=\left\{
                                                             \begin{array}{ll}
                                                             t^{k\Lambda_2+m(2\Lambda_1-\Lambda_2)}s_1, & \text{if}\ w\in\{1,\ s_1\}\\
                                                             t^{-k\Lambda_2+m(2\Lambda_1-\Lambda_2)}s_1, & \text{if}\ w\in\{s_2s_1s_2,\ s_1s_2s_1s_2\}\\
                                                             t^{k(2\Lambda_1-\Lambda_2)+(m+2n)\Lambda_2}s_2s_1s_2, & \text{if}\  w\in\{s_2,\ s_2s_1\}\\
                                                             t^{-k(2\Lambda_1-\Lambda_2)+(m+2n)\Lambda_2}s_2s_1s_2, & \text{if}\  w\in\{s_1s_2,\ s_1s_2s_1\}.
                                                             \end{array}
                                                           \right.\]
Let $\tilde{w}'=t^{m\Lambda_1+n\Lambda_2}w\tau$. Then \[\tilde{w}'t^{k\Lambda_2}s_1\tilde{w}'^{-1}=\left\{
                                                             \begin{array}{ll}
                                                             t^{-k\Lambda_2+m(2\Lambda_1-\Lambda_2)}s_1, & \text{if}\ w\in\{1,\ s_1\}\\
                                                             t^{k\Lambda_2+m(2\Lambda_1-\Lambda_2)}s_1, & \text{if}\  w\in\{s_2s_1s_2,\ s_1s_2s_1s_2\}\\
                                                             t^{-k(2\Lambda_1-\Lambda_2)+(m+2n)\Lambda_2}s_2s_1s_2, & \text{if}\  w\in\{s_2,\ s_2s_1\}\\
                                                             t^{k(2\Lambda_1-\Lambda_2)+(m+2n)\Lambda_2}s_2s_1s_2, & \text{if}\  w\in\{s_1s_2,\ s_1s_2s_1\}.
                                                             \end{array}
                                                           \right.\]
This lemma is proved by the above relations.
\end{proof}

\begin{lem}\label{6k+3}
Let $$\mathbb{O}'_{6k+3}=\{t^{\pm (k\Lambda_2+\Lambda_1)+m'(2\Lambda_1-\Lambda_2)}s_1,\ t^{\pm (k(2\Lambda_1-\Lambda_2)+\Lambda_1)+m'\Lambda_2}s_2s_1s_2|m'\in\mathbb{Z}\}.$$ Then $\widetilde{W}\cdot (t^{k(2\Lambda_1-\Lambda_2)+\Lambda_1}s_2s_1s_2)=\mathbb{O}'_{6k+3}$.
\end{lem}
\begin{proof}
The proof is similar to that of Lemma \ref{6k+1} and we omit it.
\end{proof}
By Lemma \ref{s12}, \ref{s02}, \ref{s1212}, \ref{lambda}, \ref{4k+1}, \ref{6k+1}, \ref{6k+3}, and $$W_a=\mathbb{O}_2\sqcup\mathbb{O}'_2\sqcup\mathbb{O}_{s_{1212}}\sqcup(\sqcup_{\lambda\in Q_+}\mathbb{O}_\lambda)\sqcup_{k\in\mathbb{N}}(\mathbb{O}_{4k+1}\sqcup\mathbb{O}'_{6k+1}\sqcup\mathbb{O}'_{6k+3}).$$ We have
\begin{prop}\label{conjWa}
The congugacy classes in $W_a$ are: $\mathbb{O}_2,\ \mathbb{O}'_2,\ \mathbb{O}_{s_{1212}},\ \mathbb{O}_\lambda\ (\lambda\in Q_+)$ and $\mathbb{O}_{4k+1},\ \mathbb{O}'_{6k+1},\ \mathbb{O}'_{6k+3}\ (k\in\mathbb{N})$.
\end{prop}

\begin{lem}\label{s2tau}
Let $\mathbb{O}_{1,\tau}=\{t^{\lambda}s_2\tau,\ t^{\lambda}s_1s_2s_1\tau|\lambda\in Q\}$. Then $\widetilde{W}\cdot (s_2\tau)=\mathbb{O}_{1,\tau}$.
\end{lem}
\begin{proof}
Let $\tilde{w}'=t^{m\Lambda_1+n\Lambda_2}w$. Then\[\tilde{w}'s_2\tau\tilde{w}'^{-1}=\left\{
                                                             \begin{array}{ll}
                                                             t^{2(m+n)\Lambda_1-m\Lambda_2}s_2\tau, & \text{if}\ w=1\\
                                                             t^{(2m+2n-1)\Lambda_1-(m-1)\Lambda_2}s_2\tau, & \text{if}\  w=s_1\\
                                                             t^{-(2n-1)\Lambda_1+(m+2n)\Lambda_2}s_1s_2s_1\tau, & \text{if}\ w=s_2\\
                                                             t^{(2m+2n-1)\Lambda_1-(m-1)\Lambda_2}s_2\tau, & \text{if}\ w=s_1s_2\\
                                                             t^{(2m+2n-1)\Lambda_1-m\Lambda_2}s_2\tau, & \text{if}\ w=s_2s_1\\
                                                             t^{-(2n-1)\Lambda_1+(m+2n-1)\Lambda_2}s_1s_2s_1\tau, & \text{if}\ w=s_1s_2s_1\\
                                                             t^{-2(n-1)\Lambda_1+(m+2n-1)\Lambda_2}s_1s_2s_1\tau, & \text{if}\ w=s_2s_1s_2\\
                                                             t^{2(m+n-1)\Lambda_1-(m-1)\Lambda_2}s_2\tau, & \text{if}\  w=s_1s_2s_1s_2.
                                                             \end{array}
                                                           \right.\]
Let $\tilde{w}'=t^{m\Lambda_1+n\Lambda_2}w\tau$. Then\[\tilde{w}'s_2\tau\tilde{w}'^{-1}=\left\{
                                                             \begin{array}{ll}
                                                             t^{-(2n-1)\Lambda_1+(m+2n)\Lambda_2}s_1s_2s_1\tau, & \text{if}\ w=1\\
                                                             t^{(2m+2n-1)\Lambda_1-(m-1)\Lambda_2}s_2\tau, & \text{if}\ w=s_1\\
                                                             t^{2(m+n)\Lambda_1-m\Lambda_2}s_2\tau, & \text{if}\ w=s_2\\
                                                             t^{-2n\Lambda_1+(m+2n)\Lambda_2}s_1s_2s_1\tau, & \text{if}\ w=s_1s_2\\
                                                             t^{-2(n-1)\Lambda_1=(m+2n-1)\Lambda_2}s_1s_2s_1\tau, & \text{if}\ w=s_2s_1\\
                                                             t^{2(m+n-1)\Lambda_1-(m-1)\Lambda_2}s_2\tau, & \text{if}\ w=s_1s_2s_1\\
                                                             t^{(2m+2n-1)\Lambda_1-m\Lambda_2}s_2\tau, & \text{if}\ w=s_2s_1s_2\\
                                                             t^{-(2n-1)\Lambda_1+(m+2n-1)\Lambda_2}s_1s_2s_1\tau, & \text{if}\ w=s_1s_2s_1s_2.
                                                             \end{array}
                                                           \right.\]
Thus, the lemma is proved.
\end{proof}

\begin{lem}\label{s1tau}
Let $\mathbb{O}'_{1,\tau}=\{t^{\lambda}s_1\tau| \lambda\in Q\}$. Then $\widetilde{W}\cdot(s_1\tau)=\mathbb{O}'_{1,\tau}$.
\end{lem}
\begin{proof}
For any $\tilde{w}'=t^{m\Lambda_1+n\Lambda_2}w$ or $t^{m\Lambda_1+n\Lambda_2}w\tau$, we have \[\tilde{w}'s_1\tau\tilde{w}'^{-1}=\left\{
                                                             \begin{array}{ll}
                                                             t^{2m\Lambda_1+2n\Lambda_2}s_1\tau, & \text{if}\ w\in\{1,\ s_1\}\\
                                                             t^{(2m+1)\Lambda_1+(2n-1)\Lambda_2}s_1\tau, & \text{if}\  w\in\{s_2,\ s_2s_1\}\\
                                                             t^{(2m-1)\Lambda_1+2n\Lambda_2}s_1\tau, & \text{if}\  w\in\{s_1s_2,\ s_1s_2s_1\}\\
                                                             t^{2m\Lambda_1+(2n-1)\Lambda_2}s_1\tau, & \text{if}\  \in\{s_2s_1s_2,\ s_1s_2s_1s_2\}.
                                                             \end{array}
                                                           \right.\]
This lemma is proved.
\end{proof}

\begin{lem}\label{lambdatau}
Let $Q'_+=\{\lambda\in Q_+|\lambda\neq m'\Lambda_1,\ m'\in\mathbb{N}\}$. For any $\lambda\in Q'_+$, set $\mathbb{O}_{\lambda,\tau}=\widetilde{W}\cdot(t^{\lambda}s_2s_1s_2\tau)$. Then $$\mathbb{O}_{\lambda,\tau}=\{t^{\lambda'}s_2s_1s_2\tau\ |\ \lambda'\ \text{is one of the following elements}\},$$

\[\lambda'=\left\{
                                                             \begin{array}{ll}
                                                             w(\lambda), & \emph{if}\ w\in\{1,\ s_1\}\\
                                                             w(\lambda)-\Lambda_1+\Lambda_2, & \emph{if}\ w\in\{s_2,\ s_2s_1\}\\
                                                             w(\lambda)+\Lambda_1, & \emph{if}\ w\in\{s_1s_2,\ s_1s_2s_1\}\\
                                                             w(\lambda)+\Lambda_2, & \emph{if}\ w\in\{s_2s_1s_2,\ s_1s_2s_1s_2\}.
                                                             \end{array}
                                                           \right.\]
\end{lem}
\begin{proof}
It is proved by direct computation.
\end{proof}

\begin{lem}\label{4k+2tau}
Let $\mathfrak{o}(k)\in\{k+1,\ -k\}$, and $$\mathbb{O}_{4k+2,\tau}=\{t^{\mathfrak{o}(k)\Lambda_1+m'(\Lambda_2-\Lambda_1)}s_1s_2\tau,\ t^{\mathfrak{o}(k)(\Lambda_2-\Lambda_1)+m'\Lambda_1}s_2s_1\tau|\ m'\in\mathbb{Z}\}.$$ 
Then $\widetilde{W}\cdot(t^{(k+1)\Lambda_1}s_1s_2\tau)=\mathbb{O}_{4k+2,\tau}$.
\end{lem}
\begin{proof}
Let $\tilde{w}'=t^{m\Lambda_1+n\Lambda_2}w$. Then\[\tilde{w}'t^{(k+1)\Lambda_1}s_1s_2\tau\tilde{w}'^{-1}=\left\{
                                                             \begin{array}{ll}
                                                             t^{(k+1)\Lambda_1+2n(\Lambda_2-\Lambda_1)}s_1s_2\tau, & \text{if}\  w=1\\
                                                             t^{(k+1)\Lambda_1+(2n-1)(\Lambda_2-\Lambda_1)}s_1s_2\tau, & \text{if}\  w=s_2\\
                                                             t^{(k+1)(\Lambda_2-\Lambda_1)+(2m+2n)\Lambda_1}s_2s_1\tau, & \text{if}\  w=s_1\\
                                                             t^{(k+1)(\Lambda_2-\Lambda_1)+(2m+2n-1)\Lambda_1}s_2s_1\tau, & \text{if}\  w=s_1s_2\\
                                                             t^{-k(\Lambda_2-\Lambda_1)+(2m+2n)\Lambda_1}s_2s_1\tau, & \text{if}\  w=s_2s_1\\
                                                             t^{-k(\Lambda_2-\Lambda_1)+(2m+2n-1)\Lambda_1}s_2s_1\tau, & \text{if}\  w=s_2s_1s_2\\
                                                             t^{-k\Lambda_1+2n(\Lambda_2-\Lambda_1)}s_1s_2\tau, & \text{if}\  w=s_1s_2s_1\\
                                                             t^{-k\Lambda_1+(2n-1)(\Lambda_2-\Lambda_1)}s_1s_2\tau, & \text{if}\  w=s_1s_2s_1s_2.
                                                             \end{array}
                                                           \right.\]
A similar computation for $\tilde{w}'=t^{m\Lambda_1+n\Lambda_2}w\tau$. Thus, the lemma is proved.
\end{proof}

\begin{lem}\label{6ktau}
Let $$\mathbb{O}'_{6k,\tau}=\{t^{\pm k(2\Lambda_1-\Lambda_2)+m'\Lambda_2}\tau,\ t^{\pm k\Lambda_2+\Lambda_1+m'(2\Lambda_1-\Lambda_2)}s_1s_2s_1s_2\tau|\ m'\in\mathbb{Z}\}.$$ Then $\widetilde{W}\cdot(t^{k\Lambda_2+\Lambda_1}s_1s_2s_1s_2\tau)=\mathbb{O}'_{6k,\tau}$.
\end{lem}
The proofs of this lemma and the next one are quite similar to the proof of Lemma \ref{4k+2tau}, we omit it.
\begin{lem}\label{6k+4tau}
Let $\mathfrak{o}(k)$ be as in \emph{Lemma \ref{4k+2tau}} and $$\mathbb{O}'_{6k+4,\tau}=\{t^{\mathfrak{o}(k)\Lambda_2+m'(2\Lambda_1-\Lambda_2)}s_1s_2s_1s_2\tau,\ t^{\mathfrak{o}(k)(2\Lambda_1-\Lambda_2)-\Lambda_1+m'\Lambda_2}\tau|\ m'\in\mathbb{Z}\}.$$
Then $\widetilde{W}\cdot(t^{(k+1)\Lambda_2}s_1s_2s_1s_2\tau)=\mathbb{O}'_{6k+4,\tau}$.
\end{lem}

By Lemma \ref{s2tau}, \ref{s1tau}, \ref{lambdatau}, \ref{4k+2tau}, \ref{6ktau}, \ref{6k+4tau}, and $$W_a\tau=\mathbb{O}_{1,\tau}\sqcup\mathbb{O}'_{1,\tau}\sqcup(\sqcup_{\lambda\in Q'_+}\mathbb{O}_{\lambda,\tau})\sqcup_{k\in\mathbb{N}}(\mathbb{O}_{4k+2,\tau}\sqcup\mathbb{O}'_{6k,\tau}\sqcup\mathbb{O}'_{6k+4,\tau}).$$ We have
\begin{prop}\label{conjWatau}
The congugacy classes in $W_a\tau$ are: $\mathbb{O}_{1,\tau},\ \mathbb{O}'_{1,\tau},\ \mathbb{O}_{\lambda,\tau}\ (\lambda\in Q'_+)$ and $\mathbb{O}_{4k+2,\tau},\ \mathbb{O}'_{6k,\tau},\ \mathbb{O}'_{6k+4,\tau}\ (k\in\mathbb{N})$.
\end{prop}

\subsection{Class polynomials}
For simplicity, for any $\tilde{w},\ \tilde{w}'\in\widetilde{W}$, we will write $T_{\tilde{w}}\equiv T_{\tilde{w}'} \mod[\widetilde{H},\widetilde{H}]$  (or $[\widetilde{H},\widetilde{H}]_{\delta}$) by $T_{\tilde{w}}\equiv T_{\tilde{w}'}$, and write $v-v^{-1}$ by $[v]$. Let $Q_{++}:=Q_+-(\mathbb{N}\Lambda_1\cup\mathbb{N}\Lambda_2)$, and for any $\lambda=m\Lambda_1+n\Lambda_2\in Q_+$, let $\Lambda_{\lambda}:=Q_{++}\cap(\{\lambda-n'(2\Lambda_1-\Lambda_2)|n'\in\mathbb{N}_+\}\cup\{\lambda-n'(\Lambda_2-\Lambda_1)|n'\in\mathbb{N}_+\})$, $\Lambda^1_{\lambda}=Q_{++}\cap\{\lambda-n'(2\Lambda_1-\Lambda_2)|n'\in\mathbb{N}\}$, $\Lambda^2_{\lambda}=Q_{++}\cap\{\lambda-n'(\Lambda_2-\Lambda_1)|n'\in\mathbb{N}_+\}$, $\Lambda^{1,2}_{\lambda}=\Lambda^1_{\lambda}\cup\Lambda^2_{\lambda}$, $\Lambda^{<,1}_\lambda=\Lambda^1_\lambda-\{\lambda\},\ \Lambda^{\leqslant, 2}_{\lambda}=\Lambda^2_\lambda\cup\{\lambda\}$. Moreover, we set $\diamondsuit_\lambda$ to be the closed convex set bounded by the following vertices: $0,\ (\frac{m}{2}+n)\Lambda_2,\ (m+n)\Lambda_1,\ \lambda$. Let $\Lambda^{<}_{\lambda}=Q_{++}\cap\diamondsuit_\lambda\setminus\{\lambda\}$ and $\Lambda^{\leqslant}_{\lambda}=Q_{++}\cap\diamondsuit_\lambda$.

In the remaining of this section, we will calculate degrees of class polynomials for all $\tilde{w}$ and $\mathbb{O}$.

\subsubsection{$\tilde{w}\in W_a$}

\begin{prop}\label{deglambda}
If $\lambda\in Q_+$ and $\tilde{w}\in\mathbb{O}_{\lambda}$, then \[\deg(f_{\tilde{w},\mathbb{O}})=\left\{
                                                             \begin{array}{ll}
                                                               0, & \emph{if}\ \mathbb{O}=\mathbb{O}_{\lambda}\\
                                                               -\infty, & \emph{otherwise}.
                                                             \end{array}
                                                           \right.\]
\end{prop}
\begin{proof}
By Lemma \ref{lambda}, all elements in $\mathbb{O}_{\lambda}$ have the same length. Thus, for any $\tilde{w},\ \tilde{w}'\in\mathbb{O}_\lambda$, $T_{\tilde{w}}\equiv T_{\tilde{w}'}=T_{\mathbb{O}_\lambda}$.
\end{proof}
Let $\tilde{w}\in\mathbb{O}_2$,\ then we can write $\tilde{w}$ as $t^\lambda w$, where $w\in\{s_1s_2,\ s_2s_1\}$ and $\lambda=m\Lambda_1+n\Lambda_2\in Q$, with $m,\ n\in\mathbb{Z}$. We want to calculate $f_{\tilde{w},\ \mathbb{O}}$ for $\tilde{w}\in\mathbb{O}_2$. Since $$T_{t^{m\Lambda_1+n\Lambda_2}s_1s_2}\equiv T_{\tau\cdot(t^{m\Lambda_1+n\Lambda_2}s_1s_2)}= T_{t^{(m-1)\Lambda_1-(m+n-1)\Lambda_2}s_2s_1}$$ and $$T_{t^{m\Lambda_1+n\Lambda_2}s_2s_1}\equiv T_{\tau\cdot(t^{m\Lambda_1+n\Lambda_2}s_2s_1)}=T_{t^{(m+1)\Lambda_1-(m+n)\Lambda_2}s_1s_2},$$ it is sufficient to consider $\lambda=m(2\Lambda_1-\Lambda_2)+n\Lambda_1$ where $m\in\mathbb{Z}$, $n\in\mathbb{N}_+$.

Now, we assume $m,\ n\in\mathbb{N}_+$.
\begin{prop}\label{1degso2}
For any $\tilde{w}\in\mathbb{O}_2$, we set $\tilde{w}=\tilde{w}_i$ $(1\leqslant i\leqslant 8)$, where$$\tilde{w}_1=t^{n\Lambda_1}s_2s_1,\ \tilde{w}_2\in\{t^{n(\Lambda_2-\Lambda_1)}s_2s_1,\ t^{n\Lambda_1}s_1s_2,\ t^{(n-1)(\Lambda_2-\Lambda_1)}s_1s_2\},$$ $$\tilde{w}_3\in\{t^{ m(2\Lambda_1-\Lambda_2)+n\Lambda_1}s_1s_2,\ t^{-m(2\Lambda_1-\Lambda_2)+n(\Lambda_2-\Lambda_1)}s_2s_1\},$$ $$\tilde{w}_4=t^{m\Lambda_2+(n-1)(\Lambda_2-\Lambda_1)}s_1s_2,\ \tilde{w}_5=t^{m\Lambda_2+(n-1)\Lambda_1}s_2s_1,$$ $$\tilde{w}_6=t^{m(2\Lambda_1-\Lambda_2)+n\Lambda_1}s_2s_1,\ \tilde{w}_7\in\{t^{m\Lambda_2+n\Lambda_1}s_1s_2,$$ $$t^{m\Lambda_2+n(\Lambda_2-\Lambda_1)}s_2s_1\},\ \tilde{w}_8=t^{(1-m)(2\Lambda_1-\Lambda_2)+n(\Lambda_2-\Lambda_1)}s_1s_2.$$Then
\[\deg(f_{\tilde{w},\mathbb{O}})=\left\{
                                                             \begin{array}{ll}
                                                              2, & \emph{if}\ \mathbb{O}\in\{\mathbb{O}_{\lambda}|\ \lambda\in\lambda(\tilde{w})\}\\
                                                              1, & \emph{if}\ \mathbb{O}\in \mathbb{O}^<_{\tilde{w}}\\
                                                              0, & \emph{if}\ \mathbb{O}=\mathbb{O}_{2}\\
                                                              -\infty, & \emph{otherwise}.
                                                             \end{array}
                                                           \right.\]
In the above formula,\[\lambda(\tilde{w})=\left\{
                                                             \begin{array}{ll}
                                                             \Lambda^<_{n\Lambda_1}, & \emph{if}\ \tilde{w}=\tilde{w}_1\\
                                                             \Lambda^<_{(n-1)\Lambda_1}, & \emph{if}\ \tilde{w}=\tilde{w}_2\\
                                                             \Lambda^{\leqslant}_{m\Lambda_2+(n-1)\Lambda_1}, & \emph{if}\ \tilde{w}=\tilde{w}_3\ \emph{or}\ \tilde{w}_7\\
                                                             \Lambda^{\leqslant}_{(m-1)\Lambda_2+n\Lambda_1}, & \emph{if}\ \tilde{w}=\tilde{w}_4\ \emph{or}\ \tilde{w}_8\\
                                                             \Lambda^{\leqslant}_{m\Lambda_2+(n-2)\Lambda_1}, & \emph{if}\ \tilde{w}=\tilde{w}_5\\
                                                             \Lambda^{\leqslant}_{(m+1)\Lambda_2+(n-2)\Lambda_1}, & \emph{if}\ \tilde{w}=\tilde{w}_6.
                                                             \end{array}
                                                           \right.\]
$\mathbb{O}^<_{\tilde{w}}=\{\mathbb{O}_{4i+1}|1\leqslant i\leqslant n(\tilde{w})\}\cup\{\mathbb{O}'_{6i+3}|0\leqslant i\leqslant n'(\tilde{w})\} \cup\{\mathbb{O}'_{6i+1}|1\leqslant i\leqslant n''(\tilde{w})\}$, where \[n(\tilde{w})=\left\{
\begin{array}{ll}
n-1, & \emph{if}\ \tilde{w}=\tilde{w}_1\ \emph{or}\ \tilde{w}_2\\
m+n-1, & \emph{if}\ \tilde{w}\in\{\tilde{w}_3,\ \tilde{w}_4,\ \tilde{w}_6,\ \tilde{w}_7,\ \tilde{w}_8\}\\
m+n-2, & \emph{if}\ \tilde{w}=\tilde{w}_5,
\end{array}
\right.\]
\[n'(\tilde{w})=\left\{
\begin{array}{ll}
\lfloor\frac{n-1}{2}\rfloor, & \emph{if}\ \tilde{w}=\tilde{w}_1\\
\lfloor\frac{n}{2}\rfloor-1, & \emph{if}\ \tilde{w}=\tilde{w}_2\\
m+\lfloor\frac{n}{2}\rfloor-1, & \emph{if}\ \tilde{w}=\tilde{w}_3\ \emph{or}\ \tilde{w}_7\\
m+\lfloor\frac{n-1}{2}\rfloor-1, & \emph{if}\ \tilde{w}\in\{\tilde{w}_4,\ \tilde{w}_5,\ \tilde{w}_8\}\\
m+\lfloor\frac{n-1}{2}\rfloor, & \emph{if}\ \tilde{w}=\tilde{w}_6,
\end{array}
\right.\]
\[n''(\tilde{w})=\left\{
\begin{array}{ll}
\lfloor\frac{n}{2}\rfloor, & \emph{if}\ \tilde{w}=\tilde{w}_1\\
\lfloor\frac{n-1}{2}\rfloor, & \emph{if}\ \tilde{w}=\tilde{w}_2\\
m+\lfloor\frac{n-1}{2}\rfloor, & \emph{if}\ \tilde{w}=\tilde{w}_3\ \emph{or}\ \tilde{w}_7\\
m+\lfloor\frac{n}{2}\rfloor-1, & \emph{if}\ \tilde{w}\in\{\tilde{w}_4,\ \tilde{w}_5,\ \tilde{w}_8\}\\
m+\lfloor\frac{n}{2}\rfloor, & \emph{if}\ \tilde{w}=\tilde{w}_6.
\end{array}
\right.\]
\end{prop}
\begin{proof}
If $\tilde{w}=\tilde{w}_1$, then $$T_{\tilde{w}}\equiv [v]T_{s_2t^{n\Lambda_1}s_2s_1}+T_{s_2\cdot(t^{n\Lambda_1}s_2s_1)}=[v]T_{t^{n\Lambda_1}s_1}+T_{t^{n\Lambda_1}s_1s_2}$$

\ \ \ \ \ \ \ \ \ \ $\equiv[v]T_{t^{n\Lambda_1}s_1}+T_{(\tau s_0)\cdot(t^{n\Lambda_1}s_1s_2)}=[v]T_{t^{n\Lambda_1}s_1}+T_{t^{(n-1)(\Lambda_2-\Lambda_1)}s_1s_2}$\\

\ \ \ \ \ \ \ \ \ \ $\equiv[v]T_{t^{n\Lambda_1}s_1}+[v]T_{s_1t^{(n-1)(\Lambda_2-\Lambda_1)}s_1s_2}+T_{s_1\cdot(t^{(n-1)(\Lambda_2-\Lambda_1)}s_1s_2)}$\\

\ \ \ \ \ \ \ \ \ \ $=[v](T_{t^{n\Lambda_1}s_1}+T_{t^{(n-1)\Lambda_1}s_2})+T_{t^{(n-1)\Lambda_1}s_2s_1}$\\

\ \ \ \ \ \ \ \ \ \ $\equiv \cdots\cdots$
$$\equiv[v](\sum_{i=1}^{n-1}T_{\mathbb{O}_{4i+1}}+\sum_{i=1}^{n}T_{t^{i\Lambda_1}s_1})+T_{\mathbb{O}_2}.\ \ \ \ \ \ \ \ \ \ \ \ \ \ \ \ \ \ \ \ \ \ \ \ \ \ \ \ \ \ \ \ \ $$
And \[T_{t^{i\Lambda_1}s_1}\equiv\left\{
\begin{array}{ll}
\sum_{\lambda\in\Lambda_{i\Lambda_1}}[v]T_{\mathbb{O}_\lambda}+T_{\mathbb{O}'_{6\frac{i-1}{2}+3}}, & i\ \text{is odd}\\
\sum_{\lambda\in\Lambda_{i\Lambda_1}}[v]T_{\mathbb{O}_\lambda}+T_{\mathbb{O}'_{6\frac{i}{2}+1}}, & i\ \text{is even}.
\end{array}
\right.\]
From the above calculation, we deduce that $$f_{t^{n(\Lambda_2-\Lambda_1)}s_2s_1,\mathbb{O}}=f_{t^{n\Lambda_1}s_1s_2,\mathbb{O}}=f_{t^{(n-1)(\Lambda_2-\Lambda_1)}s_1s_2,\mathbb{O}}.$$Thus, the results on $\tilde{w}_2$ is obtained.

If $\tilde{w}=\tilde{w}_3=t^{m(2\Lambda_1-\Lambda_2)+n\Lambda_1}s_1s_2$, then
$$\tilde{w}\xrightarrow[]{s_0} t^{-n\Lambda_1-(m-1)\Lambda_2}s_2s_1\ \tilde{\thicksim}\ t^{m\Lambda_2+(n-1)(\Lambda_2-\Lambda_1)}s_1s_2$$ $$\xrightarrow[]{s_1} t^{m\Lambda_2+(n-1)\Lambda_1}s_2s_1\ \xrightarrow[]{s_0}\ t^{-(m-1)(2\Lambda_1-\Lambda_2)-(n-1)\Lambda_1}s_1s_2$$ $$\tilde{\thicksim} t^{-(m-1)(2\Lambda_1-\Lambda_2)+n(\Lambda_2-\Lambda_1)}s_2s_1\ \approx\\ t^{(m-1)(2\Lambda_1-\Lambda_2)+n\Lambda_1}s_1s_2.$$
Thus
$T_{\tilde{w}}\equiv [v] T_{t^{(1-n)\Lambda_1-m\Lambda_2}s_1}+T_{t^{-n\Lambda_1-(m-1)\Lambda_2}s_2s_1}$\\

$\equiv [v]T_{t^{m\Lambda_2+(n-1)(\Lambda_2-\Lambda_1)}s_1}+T_{t^{m\Lambda_2+(n-1)(\Lambda_2-\Lambda_1)}s_1s_2}$\\

$\equiv [v]T_{t^{m\Lambda_2+(n-1)(\Lambda_2-\Lambda_1)}s_1}+[v]T_{t^{m\Lambda_2+(n-1)\Lambda_1}s_2}+T_{t^{m\Lambda_2+(n-1)\Lambda_1}s_2s_1}$\\

$\equiv [v]T_{t^{m\Lambda_2+(n-1)(\Lambda_2-\Lambda_1)}s_1}+[v]T_{t^{m\Lambda_2+(n-1)\Lambda_1}s_2}$\\

$+[v]T_{t^{-m(2\Lambda_1-\Lambda_2)-(n-2)\Lambda_1}s_{212}}+\ T_{t^{(1-m)(2\Lambda_1-\Lambda_2)-(n-1)\Lambda_1}s_1s_2}$\\

$\equiv [v]T_{t^{m\Lambda_2+(n-1)(\Lambda_2-\Lambda_1)}s_1}+[v]T_{t^{m\Lambda_2+(n-1)\Lambda_1}s_2}+[v]T_{t^{m\Lambda_2+(n-2)(\Lambda_2-\Lambda_1)}s_1}$\\

$\ \ \ \ +\ T_{t^{(m-1)(2\Lambda_1-\Lambda_2)+n\Lambda_1}s_1s_2}$\\

$\equiv \cdots\cdots$\\

$\equiv [v]\sum_{i=1}^m(T_{t^{i\Lambda_2+(n-1)(\Lambda_2-\Lambda_1)}s_1}+T_{t^{i\Lambda_2+(n-1)\Lambda_1}s_2}+T_{t^{i\Lambda_2+(n-2)(\Lambda_2-\Lambda_1)}s_1})$\\

$+ T_{t^{n\Lambda_1}s_1s_2}.$

Together with the previous cases, $$T_{t^{i\Lambda_2+j\Lambda_1}s_2}\equiv [v]\sum_{\lambda\in\Lambda^2_{i\Lambda_2+j\Lambda_1}}T_{\mathbb{O}_{\lambda}}+T_{\mathbb{O}_{4(i+j)+1}},$$ and \[T_{t^{i\Lambda_2+j(\Lambda_2-\Lambda_1)}s_1}\equiv\left\{
\begin{array}{ll}
\sum_{\lambda\in\Lambda^1_{i\Lambda_2+j\Lambda_1}}[v]T_{\mathbb{O}_\lambda}+T_{\mathbb{O}'_{6(i+\frac{j-1}{2})+3}}, & j\ \text{is odd}\\
\sum_{\lambda\in\Lambda^1_{i\Lambda_2+j\Lambda_1}}[v]T_{\mathbb{O}_\lambda}+T_{\mathbb{O}'_{6(i+\frac{j}{2})+1}}, & j\ \text{is even},
\end{array}
\right.\]
the results on $\tilde{w}_3$ is proved. And the cases of $\tilde{w}=\tilde{w}_4$ or $\tilde{w}_5$ follow directly from that of $\tilde{w}_3$.

If $\tilde{w}=\tilde{w}_6$, we have\\
$\ \ T_{\tilde{w}}\equiv [v]T_{s_2\cdot t^{m(2\Lambda_1-\Lambda_2)+n\Lambda_1}s_2s_1}+T_{t^{m\Lambda_2+n\Lambda_1}s_1s_2} =[v]T_{t^{m\Lambda_2+n\Lambda_1}s_1}+T_{t^{m\Lambda_2+n\Lambda_1}s_1s_2}$\\

\ $\equiv [v]T_{t^{m\Lambda_2+n\Lambda_1}s_1}+T_{s_1\cdot(t^{m\Lambda_2+n\Lambda_1}s_1s_2)}= [v]T_{t^{m\Lambda_2+n\Lambda_1}s_1}+T_{t^{m\Lambda_2+n(\Lambda_2-\Lambda_1)}s_2s_1}$\\

\ $\equiv [v]T_{t^{m\Lambda_2+n\Lambda_1}s_1}+[v]T_{s_0t^{m\Lambda_2+n(\Lambda_2-\Lambda_1)}s_2s_1}+T_{s_0t^{m\Lambda_2+n(\Lambda_2-\Lambda_1)}s_2s_1s_0}$\\

\ $= [v]T_{t^{m\Lambda_2+n\Lambda_1}s_1}+[v]T_{t^{(1-m)(2\Lambda_1-\Lambda_2)+(n+1)(\Lambda_2-\Lambda_1)}s_2s_1s_2}+ T_{t^{(1-m)(2\Lambda_1-\Lambda_2)+n(\Lambda_2-\Lambda_1)}s_1s_2}$\\

\ $\equiv [v]T_{t^{m\Lambda_2+n\Lambda_1}s_1}+[v]T_{t^{(m-1)\Lambda_2+(n+1)\Lambda_1}s_1}+T_{t^{(1-m)(2\Lambda_1-\Lambda_2)+n(\Lambda_2-\Lambda_1)}s_1s_2}$\\

\ $\equiv [v]T_{t^{m\Lambda_2+n\Lambda_1}s_1}+[v]T_{t^{(m-1)\Lambda_2+(n+1)\Lambda_1}s_1}+[v]T_{t^{(m-1)(2\Lambda_1-\Lambda_2)+n\Lambda_1}s_2}$\\

\ $+\ \  T_{t^{(m-1)(2\Lambda_1-\Lambda_2)+n\Lambda_1}s_2s_1}$\\

\ $\equiv \cdots\cdots$
\ $$\equiv [v]\sum_{i=1}^m(T_{t^{i\Lambda_2+n\Lambda_1}s_1}+T_{t^{(i-1)\Lambda_2+(n+1)\Lambda_1}s_1}+T_{t^{(i-1)(2\Lambda_1-\Lambda_2)+n\Lambda_1}s_2}) +T_{t^{n\Lambda_1}s_2s_1}.$$
Note that $$T_{t^{i(2\Lambda_1-\Lambda_2)+j\Lambda_1}s_2}\equiv [v]T_{t^{i\Lambda_2+j\Lambda_1}}+T_{t^{i\Lambda_2+j\Lambda_1}s_2}$$ and $$T_{t^{i\Lambda_2+j\Lambda_1}s_1}\equiv T_{t^{i\Lambda_2+j(\Lambda_2-\Lambda_1)}s_1}-[v]T_{t^{i\Lambda_2+j\Lambda_1}}.$$
Along with those previous cases, results on $\tilde{w}_6$ is proved. And the cases of $\tilde{w}=\tilde{w}_7$ or $\tilde{w}_8$ follow directly from that of $\tilde{w}_6$.
\end{proof}

\begin{prop}\label{deg1}
Let $\tilde{w}\in\{\tilde{w}_1,\ \tilde{w}_2\}\subset\mathbb{O}_{1}$, where $\ell(\tilde{w}_1)=4m-1$, $\ell(\tilde{w}_2)=4m-3$. Then \[\deg(f_{\tilde{w},\mathbb{O}})=\left\{
                                        \begin{array}{ll}
                                        3, & \emph{if}\ \mathbb{O}\in\{\mathbb{O}_\lambda|\ \lambda\in\Lambda^{<}_{\lambda(\tilde{w})-\Lambda_1}\}\\
                                        2, & \emph{if}\ \mathbb{O}\in\mathbb{O}^<_{t^{\lambda(\tilde{w})}s_1s_2}\\
                                        1, & \emph{if}\ \mathbb{O}\in\{\mathbb{O}_2\}\cup\{\mathbb{O}_{i\Lambda_1}|\ 1\leqslant i\leqslant m-1\}\\
                                        0, & \emph{if}\ \mathbb{O}=\mathbb{O}_1\\
                                        -\infty, & \emph{otherwise},
                                        \end{array}
                                        \right.\]
where \[\lambda(\tilde{w})=\left\{
                                                             \begin{array}{ll}
                                                              m\Lambda_1, & \emph{if}\ \tilde{w}=\tilde{w}_1\\
                                                              (m-1)\Lambda_1, & \emph{if}\ \tilde{w}=\tilde{w}_2.
                                                             \end{array}
                                                           \right.\]
\end{prop}
\begin{proof}
Since $$\tau t^{(1-m)(\Lambda_2-\Lambda_1)}s_2\tau^{-1}=t^{m\Lambda_1}s_1s_2s_1,\ \tau t^{(1-m)\Lambda_1}s_1s_2s_1\tau^{-1}=t^{m(\Lambda_2-\Lambda_1)}s_2,$$ it is sufficient to assume $\tilde{w}_1=t^{m(\Lambda_2-\Lambda_1)}s_2$, $\tilde{w}_2=t^{m\Lambda_1}s_1s_2s_1$. Thus, $$T_{t^{m(\Lambda_2-\Lambda_1)}s_2}\equiv [v]T_{t^{m\Lambda_1}s_1s_2}+T_{t^{m\Lambda_1}s_1s_2s_1}\ \ \ \ \ \ \ \ \ \ \ \ \ \ \ \ \ \ \ \ \ \ \ \ \ \ \ \ \ \ \ \ \ \ \ \ \ \ \ \ $$
$\ \ \ \ \ \ \ \ \ \ \ \ \ \ \ \ \ \ \ \ \ \ \ \ \equiv [v]T_{t^{m\Lambda_1}s_1s_2}+[v]T_{t^{(m-1)\Lambda_1}}+T_{t^{(m-1)(\Lambda_2-\Lambda_1)}s_2}$\\

$\ \ \ \ \ \ \ \ \ \ \ \ \ \ \ \ \ \ \ \ \ \equiv \cdots\cdots$\\

$\ \ \ \ \ \ \ \ \ \ \ \ \ \ \ \ \ \ \ \ \ \equiv [v]\sum_{i=1}^{m-1}T_{t^{i\Lambda_1}} +[v]\sum_{i=1}^{m}T_{t^{i\Lambda_1}s_1s_2}+T_{\mathbb{O}_1}$.\\
By Proposition \ref{1degso2}, the proposition is proved.
\end{proof}

\begin{prop}\label{deg4k+1}
(1) Let $\tilde{w}\in\mathbb{O}_{4k+1}$, and $\ell(\tilde{w})=4k+2m-1$, where $1\leqslant m\leqslant k+1$, then \[\deg(f_{\tilde{w},\mathbb{O}})=\left\{
                                                             \begin{array}{ll}
                                                              1, & \emph{if}\ \mathbb{O}\in\{\mathbb{O}_{k\Lambda_1+i(\Lambda_2-\Lambda_1)}\ |\ 1\leqslant i\leqslant m\}\\
                                                               0, & \emph{if}\ \mathbb{O}=\mathbb{O}_{4k+1}\\
                                                               -\infty, & \emph{otherwise}.
                                                             \end{array}
                                                           \right.\]

(2) Let $\tilde{w}\in\{\tilde{w}_1,\ \tilde{w}_2\}\subset\mathbb{O}_{4k+1}$, and $\ell(\tilde{w}_1)= 6k+4m-1,\ \ell(\tilde{w}_2)=6k+4m+1$, then \[\deg(f_{\tilde{w},\mathbb{O}})=\left\{
                                                             \begin{array}{ll}
                                                              3, & \emph{if}\ \mathbb{O}\in\{\mathbb{O}_\lambda\ |\ \lambda\in\Lambda^{\leqslant}_{k\Lambda_2+(m-1)\Lambda_1}\}\\
                                                              2, & \emph{if}\ \mathbb{O}\in\mathbb{O}^<_{t^{k\Lambda_2+m\Lambda_1}s_1s_2}\\
                                                              1, & \emph{if}\ \mathbb{O}\in\mathbb{O}(\tilde{w})\\
                                                              -\infty, & \emph{otherwise},
                                                             \end{array}
                                                           \right.\]
where \[\mathbb{O}(\tilde{w})=\left\{
\begin{array}{ll}
\{\mathbb{O}_2,\ \mathbb{O}_{k\Lambda_2}\}, & \emph{if}\ \tilde{w}=\tilde{w}_1\\
\{\mathbb{O}_2,\ \mathbb{O}_{k\Lambda_2},\ \mathbb{O}_{k\Lambda_2+m\Lambda_1}\}, & \emph{if}\ \tilde{w}=\tilde{w}_2.
\end{array}
\right.\]
\end{prop}
\begin{proof}
Since $$T_{t^{k\Lambda_1+m(\Lambda_2-\Lambda_1)}s_2}\equiv T_{t^{-k\Lambda_1+m(\Lambda_2-\Lambda_1)}s_2},\ \ \ \ \ \ \ \ \ $$
 $$T_{t^{k(\Lambda_2-\Lambda_1)+m\Lambda_1}s_1s_2s_1}\equiv T_{t^{-k(\Lambda_2-\Lambda_1)+m\Lambda_1}s_1s_2s_1},$$ $$T_{t^{k(\Lambda_2-\Lambda_1)+m\Lambda_1}s_1s_2s_1}\equiv T_{t^{-k\Lambda_1+(1-m)(\Lambda_2-\Lambda_1)}s_2},\ $$ $$T_{t^{-k(\Lambda_2-\Lambda_1)+m\Lambda_1}s_1s_2s_1}\equiv T_{t^{k\Lambda_1+(1-m)(\Lambda_2-\Lambda_1)}s_2},\ $$ it is sufficient to consider $\tilde{w}=t^{k\Lambda_1+m(\Lambda_2-\Lambda_1)}s_2$ or $\tilde{w}=t^{k\Lambda_2+m\Lambda_1}s_1s_2s_1$.

(1) If $\ell(\tilde{w})=4k+2m-1$ with $1\leqslant m\leqslant k$, one can assume that $\tilde{w}=t^{k\Lambda_1+m(\Lambda_2-\Lambda_1)}s_2$. If $\ell(\tilde{w})=6k+1$, we can assume that  $\tilde{w}=t^{k\Lambda_2+\Lambda_1}s_1s_2s_1$. Thus, $$T_{\tilde{w}}\equiv [v]\sum_{i=1}^{m-1}T_{t^{k\Lambda_1+i(\Lambda_2-\Lambda_1)}}+T_{\mathbb{O}_{4k+1}}.$$

(2) If $\ell(\tilde{w})\geqslant 6k+3$, one can take $\tilde{w}_1=t^{k\Lambda_2+m(\Lambda_2-\Lambda_1)}s_2,\ \tilde{w}_2=t^{k\Lambda_2+(m+1)\Lambda_1}s_1s_2s_1$. So

$T_{t^{k\Lambda_2+m(\Lambda_2-\Lambda_1)}s_2}\equiv [v]T_{t^{k\Lambda_2+m\Lambda_1}s_1s_2} +T_{t^{k\Lambda_2+m\Lambda_1}s_1s_2s_1}$\\

$\equiv [v]T_{t^{k\Lambda_2+m\Lambda_1}s_1s_2}+[v]T_{t^{k\Lambda_2+(m-1)\Lambda_1}}+T_{t^{k\Lambda_2+(m-1)(\Lambda_2-\Lambda_1)}s_2}$\\

$\equiv[v]\sum_{i=1}^{m}T_{t^{k\Lambda_2+i\Lambda_1}s_1s_2}+[v]\sum^{m-1}_{i=0}T_{t^{k\Lambda_2+i\Lambda_1}}+T_{t^{k\Lambda_2}s_2}.$\\
At last, we use Proposition \ref{1degso2} to finish the proof.
\end{proof}

\begin{prop}\label{deg1p}
Let $\tilde{w}\in\mathbb{O}'_1$, then
\[\deg(f_{\tilde{w},\mathbb{O}})=\left\{
                                        \begin{array}{ll}
                                        3, & \emph{if}\ \mathbb{O}\in\{\mathbb{O}_\lambda\ |\ \lambda\in\Lambda^{\leqslant}_{\lambda(\tilde{w})}\}\\
                                        2, & \emph{if}\ \mathbb{O}\in\mathbb{O}^<_{x(\tilde{w})}\\
                                        1, & \emph{if}\ \ell(\tilde{w})>1,\ \mathbb{O}\in\{\mathbb{O}_2\}\cup\{\mathbb{O}_{i\Lambda_2}\ |\ 1\leqslant i\leqslant m-1\}\\
                                        0, & \emph{if}\ \mathbb{O}=\mathbb{O}'_1\\
                                        -\infty, & \emph{otherwise},
                                        \end{array}
                                        \right.\]
where \[\lambda(\tilde{w})=\left\{
                                                             \begin{array}{ll}
                                                              m\Lambda_2-\Lambda_1, & \emph{if}\ \ell(\tilde{w})=6m-1\\
                                                              (m-1)\Lambda_2, & \emph{if}\ \ell(\tilde{w})=6m-3\\
                                                              (m-1)\Lambda_2-\Lambda_1, & \emph{if}\ \ell(\tilde{w})=6m-5,
                                                             \end{array}
                                                           \right.\]
\[x(\tilde{w})=\left\{
                                                             \begin{array}{ll}
                                                              t^{m\Lambda_2}s_2s_1, & \emph{if}\ \ell(\tilde{w})=6m-1\\
                                                              t^{(m-1)(2\Lambda_1-\Lambda_2)+\Lambda_1}s_1s_2, & \emph{if}\ \ell(\tilde{w})=6m-3\\
                                                              t^{(m-1)\Lambda_2}s_2s_1, & \emph{if}\ \ell(\tilde{w})=6m-5.
                                                             \end{array}
                                                           \right.\]
\end{prop}
\begin{proof}
Since $$\tau t^{m\Lambda_2}s_2s_1s_2\tau^{-1}=t^{(1-m)\Lambda_2}s_2s_1s_2,\ \ \ \ \ \ \ \ \ \ \ \ \ \ $$ $$\tau t^{m'(2\Lambda_1-\Lambda_2)}s_1\tau^{-1}=t^{m'(2\Lambda_1-\Lambda_2)}s_1,\ (m'\in\mathbb{Z}),$$ it is enough to consider $\tilde{w}\in\{t^{m(2\Lambda_1-\Lambda_2)}s_1,\ t^{(1-m)(2\Lambda_1-\Lambda_2)}s_1,\ t^{m\Lambda_2}s_2s_1s_2\ |\ m\in\mathbb{N}_+\}$. Moreover, we have\\

$T_{t^{m(2\Lambda_1-\Lambda_2)}s_1}\equiv [v]T_{t^{m\Lambda_2}s_2s_1}+T_{t^{m\Lambda_2}s_2s_1s_2},$\\

$T_{t^{m\Lambda_2}s_2s_1s_2}\equiv [v]T_{t^{(1-m)(2\Lambda_1-\Lambda_2)+(\Lambda_2-\Lambda_1)}s_2s_1}+T_{t^{(1-m)(2\Lambda_1-\Lambda_2)}s_1},$\\

$T_{t^{(1-m)(2\Lambda_1-\Lambda_2)}s_1}\equiv [v]T_{t^{(m-1)\Lambda_2}}+T_{t^{(m-1)(2\Lambda_1-\Lambda_2)}s_1}.$

Together with Proposition \ref{1degso2}, the proposition is proved.

\end{proof}

\begin{prop}\label{deg6k+1}
(1) Let $\tilde{w}\in\mathbb{O}'_{6k+1}$, where $k\in\mathbb{N}_+$. If $\ell(\tilde{w})=6k+2m-1,\ 1\leqslant m\leqslant k$, then \[\deg(f_{\tilde{w},\mathbb{O}})=\left\{
                                                             \begin{array}{ll}
                                                              1, & \emph{if}\ \mathbb{O}\in\{\mathbb{O}_{k\Lambda_2+i(2\Lambda_1-\Lambda_2)}\ |\ 1\leqslant i\leqslant m-1\}\\
                                                               0, & \emph{if}\ \mathbb{O}=\mathbb{O}'_{6k+1}\\
                                                               -\infty, & \emph{otherwise}.
                                                             \end{array}
                                                           \right.\]

(2) If $\tilde{w}\in\mathbb{O}'_{6k+1}$, where $k\in\mathbb{N}_+$ and $\ell(\tilde{w})>8k-1$, then \[\deg(f_{\tilde{w},\mathbb{O}})=\left\{
                                        \begin{array}{ll}
                                        3, & \emph{if}\ \mathbb{O}\in\{\mathbb{O}_\lambda,\ \lambda\in\Lambda^{\leqslant}_{\lambda(\tilde{w})}\}\\
                                        2, & \emph{if}\ \mathbb{O}\in\mathbb{O}^<_{x(\tilde{w})}\\
                                        1, & \emph{if}\ \mathbb{O}\in\{\mathbb{O}_2,\ \mathbb{O}_{2k\Lambda_1}\}\ \emph{or}\\
                                        &\ \ \  \mathbb{O}=\mathbb{O}_{2k\Lambda_1+(m-1)\Lambda_2}\ \emph{and}\ \ell(\tilde{w})=8k+6m-5\\
                                        0, & \emph{if}\ \mathbb{O}=\mathbb{O}'_{6k+1}\ \emph{and}\ \ell(\tilde{w})=8k+1\\
                                        -\infty, & \emph{otherwise},
                                        \end{array}
                                        \right.\]
where \[\lambda(\tilde{w})=\left\{
                                                             \begin{array}{ll}
                                                              m\Lambda_2+(2k-1)\Lambda_1, & \emph{if}\ \ell(\tilde{w})=8k+6m-1\\
                                                              (m-1)\Lambda_2+2k\Lambda_1, & \emph{if}\ \ell(\tilde{w})=8k+6m-3\\
                                                              (m-1)\Lambda_2+(2k-1)\Lambda_1, & \emph{if}\ \ell(\tilde{w})=8k+6m-5,
                                                             \end{array}
                                                           \right.\]
\[x(\tilde{w})=\left\{
                                                             \begin{array}{ll}
                                                              t^{m\Lambda_2+2k\Lambda_1}s_2s_1, & \emph{if}\ \ell(\tilde{w})=8k+6m-1\\
                                                              t^{(1-m)(2\Lambda_1-\Lambda_2)+2k(\Lambda_2-\Lambda_1)}s_1s_2, & \emph{if}\ \ell(\tilde{w})=8k+6m-3\\
                                                              t^{(m-1)\Lambda_2+2k\Lambda_1}s_2s_1, & \emph{if}\ \ell(\tilde{w})=8k+6m-5.
                                                             \end{array}
                                                           \right.\]
\end{prop}
\begin{proof}
By the argument of Lemma \ref{6k+1}, it is sufficient to consider $$\tilde{w}=t^{k\Lambda_2+i(2\Lambda_1-\Lambda_2)}s_1,\ (1\leqslant i\leqslant k)\ \ \text{in}\ \ (1),$$ and $$\tilde{w}\in\{t^{2k\Lambda_1+m(2\Lambda_1-\Lambda_2)}s_1,\ t^{2k\Lambda_1+m\Lambda_2}s_2s_1s_2,\ t^{2k(\Lambda_2-\Lambda_1)+(1-m)(2\Lambda_1-\Lambda_2)}s_1\} \ \text{in}\ \ (2).$$ Moreover, $$T_{t^{k\Lambda_2+i(2\Lambda_1-\Lambda_2)}s_1}\equiv [v]\sum_{j=1}^{i-1}T_{\mathbb{O}_{k\Lambda_2+j(2\Lambda_1-\Lambda_2)}}+T_{\mathbb{O}'_{6k+1}},\ \ \ \ \ \ \ \ \ \ \ \ \ \ \ \  \ \ \ \ \ \ \ \ \ \ \ \ \ \ \ \ \ $$
$$T_{t^{2k\Lambda_1+m(2\Lambda_1-\Lambda_2)}s_1}\equiv [v]T_{t^{2k\Lambda_1+m\Lambda_2}s_2s_1}+T_{t^{2k\Lambda_1+m\Lambda_2}s_2s_1s_2}, \ \ \ \ \ \ \ \ \ \ \  \ \ \ \ \ \ \ \ \ \ \ \ \ \ \ $$
$$T_{t^{2k\Lambda_1+m\Lambda_2}s_2s_1s_2}\equiv [v]T_{t^{(1-m)(2\Lambda_1-\Lambda_2)+2k(\Lambda_2-\Lambda_1)}s_1s_2}+T_{t^{2k(\Lambda_2-\Lambda_1)+(1-m)(2\Lambda_1-\Lambda_2)}s_1}$$
$$T_{t^{2k(\Lambda_2-\Lambda_1)+(1-m)(2\Lambda_1-\Lambda_2)}s_1}\equiv [v]T_{\mathbb{O}_{2k\Lambda_1+(m-1)\Lambda_2}}+T_{t^{2k\Lambda_1+(m-1)(2\Lambda_1-\Lambda_2)}s_1}.\ \ \ \ \ \ \ \ \ \ \ $$ Thus the proposition follows from Proposition \ref{1degso2}.
\end{proof}

\begin{prop}\label{deg6k+3}
(1) Let $\tilde{w}\in\mathbb{O}'_{6k+3}$, where $k\in\mathbb{N}$. If $\ell(\tilde{w})=6k+2m+3,\ 0\leqslant m\leqslant k$, then \[\deg(f_{\tilde{w},\mathbb{O}})=\left\{
                                                             \begin{array}{ll}
                                                              1, & \emph{if}\ \mathbb{O}\in\{\mathbb{O}_{\Lambda_1+k\Lambda_2+i(2\Lambda_1-\Lambda_2)}\ |\ 0\leqslant i\leqslant m-1\}\\
                                                               0, & \emph{if}\ \mathbb{O}=\mathbb{O}'_{6k+3}\\
                                                               -\infty, & \emph{if}\ \emph{otherwise}.
                                                             \end{array}
                                                           \right.\]

(2) If $\tilde{w}\in\mathbb{O}'_{6k+3}$, where $k\in\mathbb{N}$ and $\ell(\tilde{w})>8k+3$, then \[\deg(f_{\tilde{w},\mathbb{O}})=\left\{
                                        \begin{array}{ll}
                                        3, & \emph{if}\ \mathbb{O}\in\{\mathbb{O}_\lambda,\ \lambda\in\Lambda^{\leqslant}_{\lambda(\tilde{w})}\}\\
                                        2, & \emph{if}\ \mathbb{O}\in\mathbb{O}^<_{x(\tilde{w})}\\
                                        1, & \emph{if}\ \mathbb{O}\in\{\mathbb{O}_2,\ \mathbb{O}_{(2k+1)\Lambda_1}\}\ \emph{or}\\
                                         &  \ell(\tilde{w})=8k+6m-1\ \emph{and}\ \mathbb{O}=\mathbb{O}_{(2k+1)\Lambda_1+(m-1)\Lambda_2}\ \\
                                        0, & \emph{if}\ \ell(\tilde{w})=8k+5\ \emph{and}\ \mathbb{O}=\mathbb{O}'_{6k+3}\\
                                        -\infty, & \emph{otherwise},
                                        \end{array}
                                        \right.\]
where \[\lambda(\tilde{w})=\left\{
                                                             \begin{array}{ll}
                                                              m\Lambda_2+2k\Lambda_1, & \emph{if}\ \ell(\tilde{w})=8k+6m+3\\
                                                              (m-1)\Lambda_2+(2k+1)\Lambda_1, & \emph{if}\ \ell(\tilde{w})=8k+6m+1\\
                                                              (m-1)\Lambda_2+2k\Lambda_1, & \emph{if}\ \ell(\tilde{w})=8k+6m-1,
                                                             \end{array}
                                                           \right.\]
and \[x(\tilde{w})=\left\{
                                                             \begin{array}{ll}
                                                              t^{m\Lambda_2+(2k+1)\Lambda_1}s_2s_1, & \emph{if}\ \ell(\tilde{w})=8k+6m+3\\
                                                              t^{(1-m)(2\Lambda_1-\Lambda_2)+(2k+1)(\Lambda_2-\Lambda_1)}s_1s_2, & \emph{if}\ \ell(\tilde{w})=8k+6m+1\\
                                                              t^{(m-1)\Lambda_2+(2k+1)\Lambda_1}s_2s_1, & \emph{if}\ \ell(\tilde{w})=8k+6m-1.
                                                             \end{array}
                                                           \right.\]
\end{prop}
\begin{proof}
The proof is exactly the same as that of Proposition \ref{deg6k+1}, we omit it.
\end{proof}

\begin{prop}\label{2degso2}
Let $\tilde{w}\in\mathbb{O}'_2$, and we set $$\tilde{w}_1=t^{(2n+1)\Lambda_1}s_1s_2s_1s_2,\ \tilde{w}_2=t^{\Lambda_2+(2n-1)(\Lambda_2-\Lambda_1)}s_1s_2s_1s_2,$$ $$\tilde{w}_3=t^{\Lambda_2+(2n-1)\Lambda_1}s_1s_2s_1s_2,\ \tilde{w}_4=t^{(2n-1)(\Lambda_2-\Lambda_1)}s_1s_2s_1s_2,$$ $$\tilde{w}_5=t^{m(2\Lambda_1-\Lambda_2)+(2n+1)\Lambda_1}s_1s_2s_1s_2,\ \tilde{w}_6=t^{m\Lambda_2+(2n+1)\Lambda_1}s_1s_2s_1s_2,$$ $$\tilde{w}_7=t^{\Lambda_1-m(2\Lambda_1-\Lambda_2)+2n(\Lambda_2-\Lambda_1)}s_1s_2s_1s_2,\ \tilde{w}_8=t^{m\Lambda_2+(2n+1)(\Lambda_2-\Lambda_1)}s_1s_2s_1s_2.$$
Then
\[\deg(f_{\tilde{w},\mathbb{O}})=\left\{
                             \begin{array}{ll}
                             4, & \emph{if}\ \mathbb{O}\in\{\mathbb{O}_\lambda\ |\ \lambda\in\lambda(\tilde{w})\}\\
                             3, & \emph{if}\ \mathbb{O}\in\mathbb{O}^<_{\tilde{w}}\\
                             2, & \emph{if}\ \mathbb{O}\in\{\mathbb{O}_2\}\cup\mathbb{O}_{(\tilde{w})}\\
                             1, & \emph{if}\ \mathbb{O}\in\mathbb{O}^{(\tilde{w})}\\
                             0, & \emph{if}\ \mathbb{O}=\mathbb{O}'_2\\
                             -\infty, & \emph{otherwise},
                             \end{array}
                           \right.\]
where \[\lambda(\tilde{w})=\left\{
                                                             \begin{array}{ll}
                                                             \Lambda^{\leqslant}_{(2n-1)\Lambda_1}, & \emph{if}\ \tilde{w}\in\{\tilde{w}_1,\ \tilde{w}_2\}\\
                                                             \Lambda^{\leqslant}_{(2n-3)\Lambda_1}, & \emph{if}\ \tilde{w}\in\{\tilde{w}_3,\ \tilde{w}_4\}\\
                                                             \Lambda^{\leqslant}_{m\Lambda_2+(2n-1)\Lambda_1}, & \emph{if}\ \tilde{w}=\tilde{w}_5\\
                                                             \Lambda^{\leqslant}_{(m-1)\Lambda_2+(2n-1)\Lambda_1}, & \emph{if}\ \tilde{w}\in\{\tilde{w}_6,\ \tilde{w}_7\}\\
                                                             \Lambda^{\leqslant}_{(m-1)\Lambda_2+(2n+1)\Lambda_1}, & \emph{if}\ \tilde{w}=\tilde{w}_8,
                                                             \end{array}
                                                           \right.\]
$\mathbb{O}^<_{\tilde{w}}=\{\mathbb{O}_{4i+1}|1\leqslant i\leqslant n(\tilde{w})\}\cup\{\mathbb{O}'_{6i+3}|0\leqslant i\leqslant n'(\tilde{w})\} \cup\{\mathbb{O}'_{6i+1}|1\leqslant i\leqslant n''(\tilde{w})\}$, where \[n(\tilde{w})=\left\{
\begin{array}{ll}
2n-1, & \emph{if}\ \tilde{w}\in\{\tilde{w}_1,\ \tilde{w}_2\}\\
2n-3, & \emph{if}\ \tilde{w}\in\{\tilde{w}_3,\ \tilde{w}_4\}\\
m+2n-1, & \emph{if}\ \tilde{w}=\tilde{w}_5\\
m+2n-2, & \emph{if}\ \tilde{w}\in\{\tilde{w}_6,\ \tilde{w}_7\}\\
m+2n, & \emph{if}\ \tilde{w}=\tilde{w}_8,
\end{array}
\right.\]
\[n'(\tilde{w})=\left\{
\begin{array}{ll}
n-1, & \emph{if}\ \tilde{w}\in\{\tilde{w}_1,\ \tilde{w}_2\}\\
n-2, & \emph{if}\ \tilde{w}\in\{\tilde{w}_3,\ \tilde{w}_4\}\\
m+n-1, & \emph{if}\ \tilde{w}\in\{\tilde{w}_5,\ \tilde{w}_8\}\\
m+n-2, & \emph{if}\ \tilde{w}\in\{\tilde{w}_6,\ \tilde{w}_7\},
\end{array}
\right.\]
\[n''(\tilde{w})=\left\{
\begin{array}{ll}
n-1, & \emph{if}\ \tilde{w}\in\{\tilde{w}_1,\ \tilde{w}_2\}\\
n-2, & \emph{if}\ \tilde{w}\in\{\tilde{w}_3,\ \tilde{w}_4\}\\
m+n-1, & \emph{if}\ \tilde{w}\in\{\tilde{w}_5,\ \tilde{w}_8\}\\
m+n-2, & \emph{if}\ \tilde{w}\in\{\tilde{w}_6,\ \tilde{w}_7\},
\end{array}
\right.\]
and \[\mathbb{O}_{(\tilde{w})}=\left\{
\begin{array}{ll}
\{\mathbb{O}_{(2i-1)\Lambda_1}\ |\ 1\leqslant i\leqslant n\}, & \emph{if}\ \tilde{w}\in\{\tilde{w}_1,\ \tilde{w}_2\}\\
\{\mathbb{O}_{(2i-1)\Lambda_1}\ |\ 1\leqslant i\leqslant n-1\}\cup\{\mathbb{O}_{\lambda}\ |\ \lambda\in\Lambda^1_{(2n-1)\Lambda_1}\}, & \emph{if}\ \tilde{w}\in\{\tilde{w}_3,\ \tilde{w}_4\}\\
\{\mathbb{O}_{i\Lambda_2}\ |\ 1\leqslant i\leqslant m\}\cup\{\mathbb{O}_{(2i-1)\Lambda_1}\ |\ 1\leqslant i\leqslant n\}\\
\cup\{\mathbb{O}_{\lambda}\ |\ \lambda\in\Lambda^2_{m\Lambda_2+2n\Lambda_1}\}, & \emph{if}\ \tilde{w}=\tilde{w}_5\\
\{\mathbb{O}_{i\Lambda_2}\ |\ 1\leqslant i\leqslant m-1\}\cup\{\mathbb{O}_{(2i-1)\Lambda_1}\ |\ 1\leqslant i\leqslant n\}\\
\cup\{\mathbb{O}_{\lambda}\ |\ \lambda\in\Lambda^2_{m\Lambda_2+2n\Lambda_1}\cup\Lambda^1_{m\Lambda_2+(2n-1)\Lambda_1}\}\\
\cup\{\mathbb{O}_{(m-1)\Lambda_2+2n\Lambda_1}\}, & \emph{if}\ \tilde{w}=\tilde{w}_6\\
\{\mathbb{O}_{i\Lambda_2}\ |\ 1\leqslant i\leqslant m-1\}\cup\{\mathbb{O}_{(2i-1)\Lambda_1}\ |\ 1\leqslant i\leqslant n\}\\
\cup\{\mathbb{O}_{\lambda}\ |\ \lambda\in\Lambda^2_{m\Lambda_2+(2n-1)\Lambda_1}\cup\Lambda^1_{m\Lambda_2+(2n-1)\Lambda_1}\}, & \emph{if}\ \tilde{w}=\tilde{w}_7\\
\{\mathbb{O}_{i\Lambda_2}\ |\ 1\leqslant i\leqslant m-1\}\cup\{\mathbb{O}_{(2i-1)\Lambda_1}\ |\ 1\leqslant i\leqslant n+1\}, & \emph{if}\ \tilde{w}=\tilde{w}_8,
\end{array}
\right.\]
\[\mathbb{O}^{(\tilde{w})}=\left\{
\begin{array}{ll}
\{\mathbb{O}_{8n+1}\}, & \emph{if}\ \tilde{w}=\tilde{w}_1\\
\{\mathbb{O}_{4(2n-1)+1},\ \mathbb{O}_{4(2n-2)+1},\ \mathbb{O}'_{6(n-1)+3}\}, & \tilde{w}=\tilde{w}_3\\
\{\mathbb{O}_{4(2n-2)+1},\ \mathbb{O}'_{6(n-1)+3}\}, & \emph{if}\ \tilde{w}=\tilde{w}_4\\
\{\mathbb{O}_{4(m+2n)+1}\}, & \emph{if}\ \tilde{w}=\tilde{w}_5\\
\{\mathbb{O}_{4(m+2n)+1},\ \mathbb{O}_{4(m+2n-1)+1},\ \mathbb{O}'_{6(m+n-1)+3}\}, & \emph{if}\ \tilde{w}=\tilde{w}_6\\
\{\mathbb{O}_{4(m+2n-1)+1},\ \mathbb{O}'_{6(m+n-1)+3}\}, & \emph{if}\ \tilde{w}=\tilde{w}_7.
\end{array}
\right.\]
\end{prop}
\begin{proof}
It is clear that $T_{t^{\Lambda_1}s_1s_2s_1s_2}\equiv T_{\mathbb{O}'_2}$, and for $n\geqslant1$, we have
$$T_{t^{(2n+1)\Lambda_1}s_1s_2s_1s_2}\equiv [v]T_{t^{\Lambda_1+2n(\Lambda_2-\Lambda_1)}s_1s_2s_1}+T_{t^{\Lambda_2+(2n-1)(\Lambda_2-\Lambda_2)}s_1s_2s_1s_2},$$
$$T_{t^{\Lambda_2+(2n-1)(\Lambda_2-\Lambda_2)}s_1s_2s_1s_2}\equiv [v]T_{t^{\Lambda_2+(2n-1)\Lambda_1}s_2s_1s_2}+T_{t^{\Lambda_2+(2n-1)\Lambda_1}s_1s_2s_1s_2},$$
$$T_{t^{\Lambda_2+(2n-1)\Lambda_1}s_1s_2s_1s_2}\equiv [v]T_{t^{(2n-1)(\Lambda_2-\Lambda_1)}s_1s_2s_1}+T_{t^{(2n-1)(\Lambda_2-\Lambda_1)}s_1s_2s_1s_2}$$ and
$$T_{t^{(2n-1)(\Lambda_2-\Lambda_1)}s_1s_2s_1s_2}\equiv [v]T_{t^{(2n-1)\Lambda_1}s_2s_1s_2}+T_{t^{(2n-1)\Lambda_1}s_1s_2s_1s_2}.$$
By propositions \ref{deg4k+1} and \ref{deg6k+3}, the cases for $\tilde{w}\in\{\tilde{w}_1,\ \tilde{w}_2,\ \tilde{w}_3,\ \tilde{w}_4\}$ are proved.

Together with the following relations
$$T_{t^{m(2\Lambda_1-\Lambda_2)+(2n+1)\Lambda_1}s_1s_2s_1s_2}\equiv [v]T_{t^{m\Lambda_2+(2n+1)\Lambda_1}s_1s_2s_1}+T_{t^{m\Lambda_2+(2n+1)\Lambda_1}s_1s_2s_1s_2},$$
$$T_{t^{m\Lambda_2+(2n+1)\Lambda_1}s_1s_2s_1s_2}\equiv [v]T_{t^{\Lambda_1-m(2\Lambda_1-\Lambda_2)+2n(\Lambda_2-\Lambda_1)}s_1s_2s_1}+\ \ \ \ \ \ \ \ \ \ \ \ \ \ \ \ \ \ \ \ \ \ \ \ $$ $$\ \ \ \ \ +T_{t^{\Lambda_1-m(2\Lambda_1-\Lambda_2)+2n(\Lambda_2-\Lambda_1)}s_1s_2s_1s_2},\ \ \ $$
$$T_{t^{\Lambda_1-m(2\Lambda_1-\Lambda_2)+2n(\Lambda_2-\Lambda_1)}s_1s_2s_1s_2}\equiv [v]T_{t^{m(2\Lambda_1-\Lambda_2)+2n\Lambda_1+(\Lambda_2-\Lambda_1)}s_2s_1s_2}+\ \ \ \ \ \ \ \ \ \ \ \ $$ $$\ \ \ \ \ \ \ \ \ \ \ \ \ \ \ \ \ \ \ \ \ \ \ +T_{t^{(m-1)(2\Lambda_1-\Lambda_2)+(2n+1)\Lambda_1}s_1s_2s_1s_2}$$
$$T_{t^{m\Lambda_2+(2n+1)(\Lambda_2-\Lambda_1)}s_1s_2s_1s_2}\equiv [v]T_{t^{m\Lambda_2+(2n+1)\Lambda_1}s_2s_1s_2}+T_{t^{m\Lambda_2+(2n+1)\Lambda_1}s_1s_2s_1s_2},$$we finish the proof.
\end{proof}

Similarly, we have

\begin{prop}\label{2degs1212}
Let $\tilde{w}\in\mathbb{O}_{s_{1212}}$, and we set $$\tilde{w}_1=t^{2n\Lambda_1}s_1s_2s_1s_2,\ \tilde{w}_2=t^{\Lambda_2+2(n-1)(\Lambda_2-\Lambda_1)}s_1s_2s_1s_2,$$ $$\tilde{w}_3=t^{\Lambda_2+2(n-1)\Lambda_1}s_1s_2s_1s_2,\ \tilde{w}_4=t^{2(n-1)(\Lambda_2-\Lambda_1)}s_1s_2s_1s_2,$$ $$\tilde{w}_5=t^{m(2\Lambda_1-\Lambda_2)+2n\Lambda_1}s_1s_2s_1s_2,\ \tilde{w}_6=t^{m\Lambda_2+2n\Lambda_1}s_1s_2s_1s_2,$$ $$\tilde{w}_7=t^{(1-m)(2\Lambda_1-\Lambda_2)+2n(\Lambda_2-\Lambda_1)}s_1s_2s_1s_2,\ \tilde{w}_8=t^{m\Lambda_2+2n(\Lambda_2-\Lambda_1)}s_1s_2s_1s_2.$$
Then
\[\deg(f_{\tilde{w},\mathbb{O}})=\left\{
                             \begin{array}{ll}
                             4, & \emph{if}\ \mathbb{O}\in\{\mathbb{O}_\lambda\ |\ \lambda\in\lambda(\tilde{w})\}\\
                             3, & \emph{if}\ \mathbb{O}\in\mathbb{O}^<_{\tilde{w}}\\
                             2, & \emph{if}\ \mathbb{O}\in\{\mathbb{O}_2\}\cup\mathbb{O}_{(\tilde{w})}\\
                             1, & \emph{if}\ \mathbb{O}\in\mathbb{O}^{(\tilde{w})}\\
                             0, & \emph{if}\ \mathbb{O}=\mathbb{O}_{s_{1212}}\\
                             -\infty, & \emph{otherwise},
                             \end{array}
                           \right.\]
where \[\lambda(\tilde{w})=\left\{
                                                             \begin{array}{ll}
                                                             \Lambda^{\leqslant}_{(2n-2)\Lambda_1}, & \emph{if}\ \tilde{w}\in\{\tilde{w}_1,\ \tilde{w}_2\}\\
                                                             \Lambda^{\leqslant}_{(2n-4)\Lambda_1}, & \emph{if}\ \tilde{w}\in\{\tilde{w}_3,\ \tilde{w}_4\}\\
                                                             \Lambda^{\leqslant}_{m\Lambda_2+(2n-2)\Lambda_1}, & \emph{if}\ \tilde{w}=\tilde{w}_5\\
                                                             \Lambda^{\leqslant}_{(m-1)\Lambda_2+(2n-2)\Lambda_1}, & \emph{if}\ \tilde{w}\in\{\tilde{w}_6,\ \tilde{w}_7\}\\
                                                             \Lambda^{\leqslant}_{(m-1)\Lambda_2+2n\Lambda_1}, & \emph{if}\ \tilde{w}=\tilde{w}_8,
                                                             \end{array}
                                                           \right.\]
$\mathbb{O}^<_{\tilde{w}}=\{\mathbb{O}_{4i+1}|1\leqslant i\leqslant n(\tilde{w})\}\cup\{\mathbb{O}'_{6i+3}|0\leqslant i\leqslant n'(\tilde{w})\} \cup\{\mathbb{O}'_{6i+1}|1\leqslant i\leqslant n''(\tilde{w})\}$, \[n(\tilde{w})=\left\{
\begin{array}{ll}
2n-2, & \emph{if}\ \tilde{w}\in\{\tilde{w}_1,\ \tilde{w}_2\}\\
2n-4, & \emph{if}\ \tilde{w}\in\{\tilde{w}_3,\ \tilde{w}_4\}\\
m+2n-2, & \emph{if}\ \tilde{w}=\tilde{w}_5\\
m+2n-3, & \emph{if}\ \tilde{w}\in\{\tilde{w}_6,\ \tilde{w}_7\}\\
m+2n-1, & \emph{if}\ \tilde{w}=\tilde{w}_8,
\end{array}
\right.\]
\[n'(\tilde{w})=\left\{
\begin{array}{ll}
n-2, & \emph{if}\ \tilde{w}\in\{\tilde{w}_1,\ \tilde{w}_2\}\\
n-3, & \emph{if}\ \tilde{w}\in\{\tilde{w}_3,\ \tilde{w}_4\}\\
m+n-2, & \emph{if}\ \tilde{w}\in\{\tilde{w}_5,\ \tilde{w}_8\}\\
m+n-3, & \emph{if}\ \tilde{w}\in\{\tilde{w}_6,\ \tilde{w}_7\},
\end{array}
\right.\]
\[n''(\tilde{w})=\left\{
\begin{array}{ll}
n-1, & \emph{if}\ \tilde{w}\in\{\tilde{w}_1,\ \tilde{w}_2\}\\
n-2, & \emph{if}\ \tilde{w}\in\{\tilde{w}_3,\ \tilde{w}_4\}\\
m+n-1, & \emph{if}\ \tilde{w}\in\{\tilde{w}_5,\ \tilde{w}_8\}\\
m+n-2, & \emph{if}\ \tilde{w}\in\{\tilde{w}_6,\ \tilde{w}_7\},
\end{array}
\right.\]
and \[\mathbb{O}_{(\tilde{w})}=\left\{
\begin{array}{ll}
\{\mathbb{O}_{2i\Lambda_1}\ |\ 1\leqslant i\leqslant n-1\}, & \emph{if}\ \tilde{w}\in\{\tilde{w}_1,\ \tilde{w}_2\}\\
\{\mathbb{O}_{2i\Lambda_1}\ |\ 1\leqslant i\leqslant n-2\}\cup\{\mathbb{O}_{\lambda}\ |\ \lambda\in\Lambda^1_{(2n-2)\Lambda_1}\}, & \emph{if}\ \tilde{w}\in\{\tilde{w}_3,\ \tilde{w}_4\}\\
\{\mathbb{O}_{i\Lambda_2}\ |\ 1\leqslant i\leqslant m\}\cup\{\mathbb{O}_{2i\Lambda_1}\ |\ 1\leqslant i\leqslant n-1\}\\
\cup\{\mathbb{O}_{\lambda}\ |\ \lambda\in\Lambda^2_{(m+1)\Lambda_2+(2n-2)\Lambda_1}\}, & \emph{if}\ \tilde{w}=\tilde{w}_5\\
\{\mathbb{O}_{i\Lambda_2}\ |\ 1\leqslant i\leqslant m-1\}\cup\{\mathbb{O}_{2i\Lambda_1}\ |\ 1\leqslant i\leqslant n-1\}\\
\cup\{\mathbb{O}_{\lambda}\ |\ \lambda\in\Lambda^2_{(m-1)\Lambda_2+2n\Lambda_1}\cup\Lambda^1_{(m-1)\Lambda_2+2n\Lambda_1}\}, & \emph{if}\ \tilde{w}=\tilde{w}_6\\
\{\mathbb{O}_{i\Lambda_2}\ |\ 1\leqslant i\leqslant m-1\}\cup\{\mathbb{O}_{2i\Lambda_1}\ |\ 1\leqslant i\leqslant n-1\}\\
\cup\{\mathbb{O}_{\lambda}\ |\ \lambda\in\Lambda^2_{m\Lambda_2+(2n-2)\Lambda_1}\cup\Lambda^1_{m\Lambda_2+(2n-2)\Lambda_1}\}, & \emph{if}\ \tilde{w}=\tilde{w}_7\\
\{\mathbb{O}_{i\Lambda_2}\ |\ 1\leqslant i\leqslant m-1\}\cup\{\mathbb{O}_{2i\Lambda_1}\ |\ 1\leqslant i\leqslant n\}, & \emph{if}\ \tilde{w}=\tilde{w}_8,
\end{array}
\right.\]
\[\mathbb{O}^{(\tilde{w})}=\left\{
\begin{array}{ll}
\{\mathbb{O}_{4(2n-1)+1}\}, & \emph{if}\ \tilde{w}=\tilde{w}_1\\
\{\mathbb{O}_{4(2n-2)+1},\ \mathbb{O}_{4(2n-3)+1},\ \mathbb{O}'_{6(n-1)+1}\}, & \emph{if}\ \tilde{w}=\tilde{w}_3\\
\{\mathbb{O}_{4(2n-3)+1},\ \mathbb{O}'_{6(n-1)+1}\}, & \emph{if}\ \tilde{w}=\tilde{w}_4\\
\{\mathbb{O}_{4(m+2n-1)+1}\}, & \emph{if}\ \tilde{w}=\tilde{w}_5\\
\{\mathbb{O}_{4(m+2n-1)+1},\ \mathbb{O}_{4(m+2n-2)+1},\ \mathbb{O}'_{6(m+n-1)+1}\}, & \emph{if}\ \tilde{w}=\tilde{w}_6\\
\{\mathbb{O}_{4(m+2n-2)+1},\ \mathbb{O}'_{6(m+n-1)+1}\}, & \emph{if}\ \tilde{w}=\tilde{w}_7.
\end{array}
\right.\]
\end{prop}

\begin{rmk}
Since $f_{\tilde{w},\mathbb{O}}=f_{\tau\tilde{w}\tau^{-1},\mathbb{O}}$, we obtain all the degrees of the class polynomials $f_{\tilde{w},\mathbb{O}}$ for all $\tilde{w}\in W_a$ and $\mathbb{O}$.
\end{rmk}
\subsubsection{$\tilde{w}\in W_a\tau$}
\begin{prop}\label{degtaulambda}
For any $\lambda\in Q'_+$ and $\tilde{w}\in\mathbb{O}_{\lambda, \tau}$, we have\[\deg(f_{\tilde{w},\mathbb{O}})=\left\{
                                                             \begin{array}{ll}
                                                              0, & \emph{if}\ \mathbb{O}=\mathbb{O}_{\lambda,\tau}\\
                                                              -\infty, & \emph{otherwise}.
                                                             \end{array}
                                                           \right.\]
\end{prop}
\begin{proof}
The proposition follows from Lemma \ref{lambdatau}.
\end{proof}
\begin{prop}\label{2degstau1}
For any $\tilde{w}\in\mathbb{O}_{1,\tau}$, we set $\tilde{w}$ or $\tau\tilde{w}\tau^{-1}=\tilde{w}_i$ $(1\leqslant i\leqslant 8)$, where$$\tilde{w}_1\in\{t^{n\Lambda_1}s_1s_2s_1\tau,\ t^{(n-1)\Lambda_1}s_2\tau,\ t^{(n-1)(\Lambda_2-\Lambda_1)}s_1s_2s_1\tau\ |\ n\geqslant2\}$$ $$\tilde{w}_2=t^{(n-1)(\Lambda_2-\Lambda_1)}s_2\tau,\ \tilde{w}_3=t^{ m(2\Lambda_1-\Lambda_2)+n\Lambda_1}s_1s_2s_1\tau,$$ $$\tilde{w}_4\in\{t^{m\Lambda_2+(n-1)\Lambda_1}s_2\tau,\ t^{m\Lambda_2+(n-1)(\Lambda_2-\Lambda_1)}s_1s_2s_1\tau\},$$ $$\tilde{w}_5=t^{(1-m)(2\Lambda_2-\Lambda_1)+n(\Lambda_2-\Lambda_1)}s_2\tau,$$ $$\tilde{w}_6\in\{t^{m(2\Lambda_1-\Lambda_2)+n\Lambda_1}s_2\tau,\ t^{-m(2\Lambda_1-\Lambda_2)+n(\Lambda_2-\Lambda_1)}s_1s_2s_1\tau,\ t^{m\Lambda_2+(n+1)\Lambda_1}s_1s_2s_1\tau\}$$ $$\tilde{w}_7= t^{m\Lambda_2+n(\Lambda_2-\Lambda_1)}s_2\tau.$$Then
\[\deg(f_{\tilde{w},\mathbb{O}})=\left\{
                                                             \begin{array}{ll}
                                                              2, & \emph{if}\ \mathbb{O}\in\{\mathbb{O}_{\lambda,\tau}\ |\ \lambda\in\lambda(\tilde{w})\}\\
                                                              1, & \emph{if}\ \mathbb{O}\in \mathbb{O}^<_{\tilde{w}}\\
                                                              0, & \emph{if}\ \mathbb{O}=\mathbb{O}_{2}\\
                                                              -\infty, & \emph{otherwise},
                                                             \end{array}
                                                           \right.\]
where,\[\lambda(\tilde{w})=\left\{
                                                             \begin{array}{ll}
                                                             \Lambda^{\leqslant}_{\Lambda_2+(n-1)\Lambda_1}, & \emph{if}\ \tilde{w}=\tilde{w}_1\\
                                                             \Lambda^{\leqslant}_{\Lambda_2+(n-2)\Lambda_1}, & \emph{if}\ \tilde{w}=\tilde{w}_2\\
                                                             \Lambda^{\leqslant}_{(m+1)\Lambda_2+(n-2)\Lambda_1}, & \emph{if}\ \tilde{w}=\tilde{w}_3\\
                                                             \Lambda^{\leqslant}_{m\Lambda_2+(n-1)\Lambda_1}, & \emph{if}\ \tilde{w}=\tilde{w}_4\\
                                                             \Lambda^{\leqslant}_{(m-1)\Lambda_2+n\Lambda_1}, & \emph{if}\ \tilde{w}=\tilde{w}_5\\
                                                             \Lambda^{\leqslant}_{(m+1)\Lambda_2+(n-1)\Lambda_1}, & \emph{if}\ \tilde{w}=\tilde{w}_6\\
                                                             \Lambda^{\leqslant}_{m\Lambda_2+n\Lambda_1}, & \emph{if}\ \tilde{w}=\tilde{w}_7,
                                                             \end{array}
                                                           \right.\]
and $\mathbb{O}^<_{\tilde{w}}=\{\mathbb{O}_{4i+2,\tau}|1\leqslant i\leqslant n(\tilde{w})\}\cup\{\mathbb{O}'_{6i,\tau}|0\leqslant i\leqslant n'(\tilde{w})\} \cup\{\mathbb{O}'_{6i+4,\tau}|1\leqslant i\leqslant n''(\tilde{w})\}$, \[n(\tilde{w})=\left\{
\begin{array}{ll}
n-2, & \emph{if}\ \tilde{w}=\tilde{w}_1\ \emph{or}\ \tilde{w}_2\\
m+n-2, & \emph{if}\ \tilde{w}\in\{\tilde{w}_3,\ \tilde{w}_4,\ \tilde{w}_5\}\\
m+n-1, & \emph{if}\ \tilde{w}\in\{\tilde{w}_6,\ \tilde{w}_7\},
\end{array}
\right.\]
\[n'(\tilde{w})=\left\{
\begin{array}{ll}
\lfloor\frac{n}{2}\rfloor, & \emph{if}\ \tilde{w}=\tilde{w}_1\\
\lfloor\frac{n-1}{2}\rfloor, & \emph{if}\ \tilde{w}=\tilde{w}_2\\
m+\lfloor\frac{n-1}{2}\rfloor, & \emph{if}\ \tilde{w}\in\{\tilde{w}_3,\ \tilde{w}_7\}\\
m+\lfloor\frac{n}{2}\rfloor-1, & \emph{if}\ \tilde{w}=\tilde{w}_4\\
m+\lfloor\frac{n-1}{2}\rfloor-1, & \emph{if}\ \tilde{w}=\tilde{w}_5\\
m+\lfloor\frac{n}{2}\rfloor, & \emph{if}\ \tilde{w}=\tilde{w}_6,
\end{array}
\right.\]
\[n''(\tilde{w})=\left\{
\begin{array}{ll}
\lfloor\frac{n-1}{2}\rfloor, & \emph{if}\ \tilde{w}=\tilde{w}_1\\
\lfloor\frac{n-2}{2}\rfloor, & \emph{if}\ \tilde{w}=\tilde{w}_2\\
m+\lfloor\frac{n}{2}\rfloor-1, & \emph{if}\ \tilde{w}\in\{\tilde{w}_3,\ \tilde{w}_7\}\\
m+\lfloor\frac{n-1}{2}\rfloor-1, & \emph{if}\ \tilde{w}=\tilde{w}_4\\
m+\lfloor\frac{n}{2}\rfloor-2, & \emph{if}\ \tilde{w}=\tilde{w}_5\\
m+\lfloor\frac{n-1}{2}\rfloor, & \emph{if}\ \tilde{w}=\tilde{w}_6.
\end{array}
\right.\]
\end{prop}
\begin{proof}
By the following relations:\\
$$T_{t^{n\Lambda_1}s_1s_2s_1\tau}\equiv T_{t^{(n-1)\Lambda_1}s_2\tau}\equiv [v]T_{t^{\Lambda_1+(n-1)(\Lambda_2-\Lambda_1)}s_1s_2s_1s_2\tau}+T_{t^{(n-1)(\Lambda_2-\Lambda_1)}s_2\tau}$$

$$\equiv [v]T_{t^{\Lambda_1+(n-1)(\Lambda_2-\Lambda_1)}s_1s_2s_1s_2\tau}+[v]T_{t^{(n-1)\Lambda_1}s_1s_2\tau}+T_{t^{(n-1)\Lambda_1}s_1s_2s_1\tau}$$and

\[T_{t^{\Lambda_1+i(\Lambda_2-\Lambda_1)}s_1s_2s_1s_2\tau}\equiv \left\{
\begin{array}{ll}
\sum_{\lambda\in\Lambda'_{\Lambda_2+(i-1)\Lambda_1}}[v]T_{\mathbb{O}_{\lambda,\tau}}+T_{\mathbb{O}'_{6\frac{i}{2},\tau}}, & i\ \text{is even}\\
\sum_{\lambda\in\Lambda'_{\Lambda_2+(i-1)\Lambda_1}}[v]T_{\mathbb{O}_{\lambda,\tau}}+T_{\mathbb{O}'_{6\frac{i-1}{2}+4,\tau}}, & i\ \text{is odd},
\end{array}
\right.\]
$T_{t^{i\Lambda_1}s_1s_2\tau}\equiv T_{\mathbb{O}_{4(i-1)+2,\tau}}$, the case for $\tilde{w}$ or $\tau\tilde{w}\tau^{-1}=\tilde{w}_1$ is proved.

The case for $\tilde{w}$ or $\tau\tilde{w}\tau^{-1}=\tilde{w}_2$ follows from the $\tilde{w}_1$ case directly.

By the $\tilde{w}_1$ case and  $$T_{t^{m(2\Lambda_1-\Lambda_2)+n\Lambda_1}s_1s_2s_1\tau}\equiv [v]T_{t^{m\Lambda_2+n\Lambda_1}s_1s_2s_1s_2\tau}+T_{t^{m\Lambda_2+(n-1)\Lambda_1}s_2\tau}$$

$\equiv [v]T_{t^{m\Lambda_2+n\Lambda_1}s_1s_2s_1s_2\tau}+[v]T_{t^{(m-1)\Lambda_2+(n+1)\Lambda_1}s_1s_2s_1s_2\tau}+T_{t^{(1-m)(2\Lambda_1-\Lambda_2)+n(\Lambda_2-\Lambda_1)}s_2\tau}$

$\equiv [v](T_{t^{m\Lambda_2+n\Lambda_1}s_1s_2s_1s_2\tau}+T_{t^{(m-1)\Lambda_2+(n+1)\Lambda_1}s_1s_2s_1s_2\tau}+T_{t^{(m-1)(2\Lambda_1-\Lambda_2)+n\Lambda_1}s_1s_2\tau})$

$ +T_{t^{(m-1)(2\Lambda_1-\Lambda_2)+n\Lambda_1}s_1s_2s_1\tau}$, the cases for $\tilde{w}_3,\ \tilde{w}_4,\ \tilde{w}_5$ are proved.

Similar argument for the cases of $\tilde{w}_6,\ \tilde{w}_7$.
\end{proof}

\begin{prop}\label{deg4k+2tau}
(1) If $\tilde{w}\in\mathbb{O}_{4k+2,\tau}$ and $\ell(\tilde{w})=4k+2+2i$, where $0\leqslant i\leqslant k$. Then \[\deg(f_{\tilde{w},\mathbb{O}})=\left\{
                                                             \begin{array}{ll}
                                                              1, & \emph{if}\ \mathbb{O}\in\{\mathbb{O}_{(k+1)\Lambda_1+j(\Lambda_2-\Lambda_1),\tau}\ |\ 1\leqslant j\leqslant i\}\\
                                                              0, & \emph{if}\ \mathbb{O}=\mathbb{O}_{4k+2,\tau}\\
                                                              -\infty, & \emph{otherwise}.
                                                             \end{array}
                                                           \right.\]

(2) Let $\tilde{w}\in\{\tilde{w}_1,\ \tilde{w}_2\}\subset\mathbb{O}_{4k+2,\tau}$, and $\ell(\tilde{w}_1)= 6(k+1)+4m,\ \ell(\tilde{w}_2)=6(k+1)+4m-2$ ($m\in\mathbb{N}$), then \[\deg(f_{\tilde{w},\mathbb{O}})=\left\{
                                                             \begin{array}{ll}
                                                              3, & \emph{if}\ \mathbb{O}=\mathbb{O}_{\lambda,\tau},\ \lambda\in\Lambda^{\leqslant}_{\lambda(\tilde{w})}\\
                                                              2, & \emph{if}\ \mathbb{O}\in\mathbb{O}^<_{\tilde{w}}\\
                                                              1, & \emph{if}\ \mathbb{O}\in\mathbb{O}(\tilde{w})\\
                                                              -\infty, & \emph{otherwise},
                                                             \end{array}
                                                           \right.\]
where\[\lambda(\tilde{w})=\left\{
\begin{array}{ll}
(k+1)\Lambda_2+m\Lambda_1, & \emph{if}\ \tilde{w}=\tilde{w}_1\\
(k+1)\Lambda_2+(m-1)\Lambda_1, & \emph{if}\ \tilde{w}=\tilde{w}_2,
\end{array}
\right.\]
$\mathbb{O}^<_{\tilde{w}}=\{\mathbb{O}_{4i+2,\tau}|1\leqslant i\leqslant n(\tilde{w})\}\cup\{\mathbb{O}'_{6i,\tau}|0\leqslant i\leqslant n'(\tilde{w})\} \cup\{\mathbb{O}'_{6i+4,\tau}|1\leqslant i\leqslant n''(\tilde{w})\}$, \[n(\tilde{w})=\left\{
\begin{array}{ll}
k+m, & \emph{if}\ \tilde{w}=\tilde{w}_1\\
k+m-1, & \emph{if}\ \tilde{w}=\tilde{w}_2,
\end{array}
\right.\]
\[n'(\tilde{w})=\left\{
\begin{array}{ll}
k+\lfloor\frac{m+1}{2}\rfloor, & \emph{if}\ \tilde{w}=\tilde{w}_1\\
k+\lfloor\frac{m}{2}\rfloor, & \emph{if}\ \tilde{w}=\tilde{w}_2
\end{array}
\right.\]
\[n''(\tilde{w})=\left\{
\begin{array}{ll}
k+\lfloor\frac{m}{2}\rfloor, & \emph{if}\ \tilde{w}=\tilde{w}_1\\
k+\lfloor\frac{m-1}{2}\rfloor, & \emph{if}\ \tilde{w}=\tilde{w}_2,
\end{array}
\right.\]
and
\[\mathbb{O}(\tilde{w})=\left\{
\begin{array}{ll}
\{\mathbb{O}_{1,\tau},\ \mathbb{O}_{(k+1)\Lambda_2,\tau}\}, & \emph{if}\ \tilde{w}=\tilde{w}_1\\
\{\mathbb{O}_{1,\tau},\ \mathbb{O}_{(k+1)\Lambda_2,\tau},\ \mathbb{O}_{(k+1)\Lambda_2+m\Lambda_1,\tau}\}, & \emph{if}\ \tilde{w}=\tilde{w}_2.
\end{array}
\right.\]
\end{prop}
\begin{proof}
(1) If $\tilde{w},\ \tilde{w}'\in\mathbb{O}_{4k+2,\tau}$ and $\ell(\tilde{w})=\ell(\tilde{w}')=4k+2+2i$, where $0\leqslant i\leqslant k$, then $f_{\tilde{w},\mathbb{O}}=f_{\tilde{w}',\mathbb{O}}$. It is sufficient to consider $\tilde{w}=t^{(k+1)\Lambda_1+i(\Lambda_2-\Lambda_1)}s_1s_2\tau$. If $i=0$, then $T_{\tilde{w}}\equiv T_{\mathbb{O}_{4k+2,\tau}}$. If $i\geqslant1$, then $T_{\tilde{w}}\equiv \sum_{j=1}^{i}[v]T_{t^{(k+1)\Lambda_1+j(\Lambda_2-\Lambda_1)}s_2s_1s_2\tau}+T_{\mathbb{O}_{4k+2,\tau}}$.

(2) By Lemma \ref{4k+2tau}, it is sufficient to consider $\tilde{w}_1=t^{(k+1)\Lambda_2+m(\Lambda_2-\Lambda_1)}s_1s_2\tau,\ \tilde{w}_2=t^{(k+1)\Lambda_2+m\Lambda_1}s_2s_1\tau$. Furthermore, $$T_{t^{(k+1)\Lambda_2+m(\Lambda_2-\Lambda_1)}s_1s_2\tau}\equiv [v]T_{t^{(k+1)\Lambda_2+m\Lambda_1}s_2\tau}+T_{t^{(k+1)\Lambda_2+m\Lambda_1}s_2s_1\tau}$$ and $$T_{t^{(k+1)\Lambda_2+m\Lambda_1}s_2s_1\tau}\equiv [v]T_{t^{(k+1)\Lambda_2+m\Lambda_1}s_2s_1s_2\tau}+T_{t^{(k+1)\Lambda_2+(m-1)(\Lambda_2-\Lambda_1)}s_1s_2\tau}.$$ Thus, together with Proposition \ref{2degstau1}, the proposition is proved.
\end{proof}

\begin{prop}\label{deg0tau}
Let $\tilde{w}\in\mathbb{O}'_{0,\tau}$, then
\[\deg(f_{\tilde{w},\mathbb{O}})=\left\{
                                        \begin{array}{ll}
                                        3, & \emph{if}\ \mathbb{O}=\mathbb{O}_{\lambda,\tau},\ \lambda\in\Lambda^{\leqslant}_{\lambda(\tilde{w})}\\
                                        2, & \emph{if}\ \mathbb{O}\in\mathbb{O}^<_{\tilde{w}}\\
                                        1, & \emph{if}\ \mathbb{O}\in\mathbb{O}(\tilde{w})\ (\ell(\tilde{w})>1)\\
                                        0, & \emph{if}\ \mathbb{O}=\mathbb{O}'_{0,\tau}\\
                                        -\infty, & \emph{otherwise},
                                        \end{array}
                                        \right.\]
where \[\lambda(\tilde{w})=\left\{
                                                             \begin{array}{ll}
                                                              m\Lambda_2-(\Lambda_2-\Lambda_1), & \emph{if}\ \ell(\tilde{w})=6m\\
                                                              (m-1)\Lambda_2-(\Lambda_2-\Lambda_1), & \emph{if}\ \ell(\tilde{w})=6m-2\ \emph{or}\ 6m-4,
                                                             \end{array}
                                                           \right.\]
$\mathbb{O}^<_{\tilde{w}}=\{\mathbb{O}_{4i+2,\tau}|1\leqslant i\leqslant n(\tilde{w})\}\cup\{\mathbb{O}'_{6i,\tau}|0\leqslant i\leqslant m-1\} \cup\{\mathbb{O}'_{6i+4,\tau}|1\leqslant i\leqslant n''(\tilde{w})\}$, \[n(\tilde{w})=n''(\tilde{w})=\left\{
\begin{array}{ll}
m-1, & \emph{if}\ \ell(\tilde{w})=6m\\
m-2, & \emph{if}\ \ell(\tilde{w})=6m-2\ \emph{or}\ 6m-4,
\end{array}
\right.\]

\[\mathbb{O}(\tilde{w})=\left\{
                                  \begin{array}{ll}
                                  \{\mathbb{O}_{1,\tau},\ \mathbb{O}_{i\Lambda_2,\tau}\ (1\leqslant i\leqslant m)\}, & \emph{if}\ \ell(\tilde{w})=6m\ \emph{or}\ 6m-2\\
                                  \{\mathbb{O}_{1,\tau},\ \mathbb{O}_{i\Lambda_2,\tau}\ (1\leqslant i\leqslant m-1)\}, & \emph{if}\ \ell(\tilde{w})=6m-4.
                                  \end{array}
                            \right.\]
\end{prop}
\begin{proof}
Since $\tau t^{m\Lambda_2}\tau\tau^{-1}=t^{-m\Lambda_2}\tau$, it is sufficient to consider\[\tilde{w}=\left\{
\begin{array}{ll}
t^{m\Lambda_2}\tau, & \text{if}\ \ell(\tilde{w})=6m\\
t^{-m(2\Lambda_1-\Lambda_2)+\Lambda_1}s_1s_2s_1s_2\tau, & \text{if}\ \ell(\tilde{w})=6m-2\\
t^{(m-1)(2\Lambda_1-\Lambda_2)+\Lambda_1}s_1s_2s_1s_2\tau, & \text{if}\ \ell(\tilde{w})=6m-4.
\end{array}
\right.\]
Moreover, $$T_{t^{m\Lambda_2}\tau}\equiv [v]T_{t^{m\Lambda_2}s_2\tau}+T_{t^{-m(2\Lambda_1-\Lambda_2)+\Lambda_1}s_1s_2s_1s_2\tau}\ \ \ \ \ \ \ \ \ \ \ \ \ \ \ \ \ \ \ \ \ \ \ \ \ \ \ \ \ \ \ \ \ \ \ \ \ \ \ \ \ \ \ \ \ \ $$
$$\equiv [v]T_{t^{m\Lambda_2}s_2\tau}+[v]T_{t^{m\Lambda_2}s_2s_1s_2\tau}+T_{t^{(m-1)(2\Lambda_1-\Lambda_2)+\Lambda_1}s_1s_2s_1s_2\tau}\ \ \ \ \ \ \ \ \ \ \  $$
$$\equiv [v]T_{t^{m\Lambda_2}s_2\tau}+[v]T_{t^{m\Lambda_2}s_2s_1s_2\tau}+[v]T_{t^{\Lambda_1+(m-1)\Lambda_2}s_1s_2s_1\tau}+T_{t^{(m-1)\Lambda_2}\tau}$$ Thus, the proposition follows from Proposition \ref{2degstau1}.
\end{proof}

\begin{prop}\label{deg6ktau}
(1) Let $\tilde{w}\in\mathbb{O}'_{6k,\tau}$, where $k\in\mathbb{N}_+$. If $\ell(\tilde{w})=6k+2m,\ 0\leqslant m\leqslant k$, then \[\deg(f_{\tilde{w},\mathbb{O}})=\left\{
                                                             \begin{array}{ll}
                                                              1, & \emph{if}\ \mathbb{O}\in\{\mathbb{O}_{\Lambda_1+k\Lambda_2+i(2\Lambda_1-\Lambda_2),\tau}\ |\ 0\leqslant i\leqslant m-1\}\\
                                                              0, & \emph{if}\ \mathbb{O}=\mathbb{O}'_{6k,\tau}\\
                                                               -\infty, & \emph{otherwise}.
                                                             \end{array}
                                                           \right.\]

(2) If $\tilde{w}\in\mathbb{O}'_{6k,\tau}$, and let $\tilde{w}_1,\ \tilde{w}_2,\ \tilde{w}_3$ with length $\ell(\tilde{w}_1)=8k+6m,\ \ell(\tilde{w}_2)=8k+6m-2,\ \ell(\tilde{w}_3)=8k+6m-4$, where $k,\ m\in\mathbb{N}_+$. Then \[\deg(f_{\tilde{w},\mathbb{O}})=\left\{
                                        \begin{array}{ll}
                                        3, & \emph{if}\ \mathbb{O}\in\{\mathbb{O}_{\lambda,\tau}\ |\ \lambda\in\Lambda^{\leqslant}_{\lambda(\tilde{w})}\}\\
                                        2, & \emph{if}\ \mathbb{O}\in\mathbb{O}^<_{\tilde{w}}\\
                                        1, & \emph{if}\ \mathbb{O}=\mathbb{O}_{1,\tau}\ \emph{or}\ \tilde{w}=\tilde{w}_2\ \emph{and}\ \mathbb{O}=\mathbb{O}_{2k\Lambda_1+m\Lambda_2}\\
                                        -\infty, & \emph{otherwise},
                                        \end{array}
                                        \right.\]
where \[\lambda(\tilde{w})=\left\{
                                                             \begin{array}{ll}
                                                              m\Lambda_2+2k\Lambda_1, & \emph{if}\ \tilde{w}=\tilde{w}_1\\
                                                              m\Lambda_2+(2k-1)\Lambda_1, & \emph{if}\ \tilde{w}\in\{\tilde{w}_2,\ \tilde{w}_3\},
                                                             \end{array}
                                                           \right.\]
$\mathbb{O}^<_{\tilde{w}}=\{\mathbb{O}_{4i+2,\tau}|1\leqslant i\leqslant n(\tilde{w})\}\cup\{\mathbb{O}'_{6i,\tau}|0\leqslant i\leqslant n'(\tilde{w})\} \cup\{\mathbb{O}'_{6i+4,\tau}|1\leqslant i\leqslant n''(\tilde{w})\}$, \[n(\tilde{w})=\left\{
\begin{array}{ll}
m+2k-1, & \emph{if}\ \tilde{w}=\tilde{w}_1\\
m+2k-4, & \emph{if}\ \tilde{w}\in\{\tilde{w}_2,\ \tilde{w}_3\},
\end{array}
\right.\]
\[n'(\tilde{w})=\left\{
\begin{array}{ll}
m+k-1, & \emph{if}\ \tilde{w}=\tilde{w}_1\\
m+k-2, & \emph{if}\ \tilde{w}\in\{\tilde{w}_2,\ \tilde{w}_3\},
\end{array}
\right.\]
\[n''(\tilde{w})=\left\{
\begin{array}{ll}
m+k-1, & \emph{if}\ \tilde{w}=\tilde{w}_1\\
m+k-3, & \emph{if}\ \tilde{w}\in\{\tilde{w}_2,\ \tilde{w}_3\}.
\end{array}
\right.\]
\end{prop}
\begin{proof}
Since $\tau(t^{\pm k(2\Lambda_1-\Lambda_2)-m\Lambda_2}\tau)\tau^{-1}=t^{\pm k(2\Lambda_1-\Lambda_2)+m\Lambda_2}\tau$, $$s_1(t^{-k(2\Lambda_1-\Lambda_2)+m\Lambda_2}\tau)s_1=t^{k(2\Lambda_1-\Lambda_2)+m\Lambda_2}\tau,$$ $$\tau(t^{-k\Lambda_2+\Lambda_1+m(2\Lambda_1-\Lambda_2)}s_1s_2s_1s_2\tau)\tau^{-1}=t^{k\Lambda_2+\Lambda_1+m(2\Lambda_1-\Lambda_2)}s_1s_2s_1s_2\tau,$$it is sufficient to consider $t^{k(2\Lambda_1-\Lambda_2)+m\Lambda_2}\tau$ and $t^{k\Lambda_2+\Lambda_1+m(2\Lambda_1-\Lambda_2)}s_1s_2s_1s_2\tau$.\\
(1) If $\ell(\tilde{w})=6k+2m,\ 0\leqslant m\leqslant k$, then we can set $\tilde{w}=t^{k(2\Lambda_1-\Lambda_2)+m\Lambda_2}\tau$. Thus, $$T_{\tilde{w}}=[v]\sum_{i=0}^{m-1}T_{\mathbb{O}_{\Lambda_1+k\Lambda_2+i(2\Lambda_1-\Lambda_2),\tau}}+T_{\mathbb{O}'_{6k,\tau}}.$$\\
(2) We may assume that $\tilde{w}_1=t^{2k\Lambda_1+m\Lambda_2}\tau,$ $$\tilde{w}_2=t^{(2k+1)(\Lambda_2-\Lambda_1)-(m-1)(2\Lambda_1-\Lambda_2)}s_1s_2s_1s_2\tau,$$ $$\tilde{w}_3=t^{(2k+1)\Lambda_1+(m-1)(2\Lambda_1-\Lambda_2)}s_1s_2s_1s_2\tau.$$ Then $$T_{\tilde{w}_1}\equiv [v]T_{t^{-m(2\Lambda_1-\Lambda_2)+2k(\Lambda_2-\Lambda_1)}s_2\tau}+T_{\tilde{w}_2}$$ $$T_{\tilde{w}_2}\equiv [v]T_{t^{(2k+1)\Lambda_1+(m-1)(2\Lambda_1-\Lambda_2)}s_2s_1s_2\tau}+T_{\tilde{w}_3}$$ $$T_{\tilde{w}_3}\equiv[v]T_{t^{(2k+1)\Lambda_1+(m-1)\Lambda_2}s_1s_2s_1\tau}+T_{t^{2k\Lambda_1+(m-1)\Lambda_2}\tau}.$$Together with (1) and Proposition \ref{2degstau1}, we finish the proof.
\end{proof}

\begin{prop}\label{deg6k+4tau}
(1) Let $\tilde{w}\in\mathbb{O}'_{6k+4,\tau}$, where $k\in\mathbb{N}$. If $\ell(\tilde{w})=6k+4+2m,\ 0\leqslant m\leqslant k$, then \[\deg(f_{\tilde{w},\mathbb{O}})=\left\{
                                                             \begin{array}{ll}
                                                              1, & \emph{if}\ \mathbb{O}\in\{\mathbb{O}_{(k+1)\Lambda_2+i(2\Lambda_1-\Lambda_2),\tau}\ |\ 1\leqslant i\leqslant m\}\\
                                                              0, & \emph{if}\ \mathbb{O}=\mathbb{O}'_{6k+4,\tau}\\
                                                               -\infty, & \emph{otherwise}.
                                                             \end{array}
                                                           \right.\]

(2) If $\tilde{w}\in\mathbb{O}'_{6k+4,\tau}$, where $k\in\mathbb{N}$. And let $\tilde{w}_1,\ \tilde{w}_2,\ \tilde{w}_3$ with length $\ell(\tilde{w}_1)=8k+6m+4,\ \ell(\tilde{w}_2)=8k+6m+2,\ \ell(\tilde{w}_3)=8k+6m$, where $m\in\mathbb{N}_+$. Then \[\deg(f_{\tilde{w},\mathbb{O}})=\left\{
                                        \begin{array}{ll}
                                        3, & \emph{if}\ \mathbb{O}\in\{\mathbb{O}_{\lambda,\tau}\ |\ \lambda\in\Lambda^{\leqslant}_{\lambda(\tilde{w})}\}\\
                                        2, & \emph{if}\ \mathbb{O}\in\mathbb{O}^<_{\tilde{w}}\\
                                        1, & \emph{if}\ \mathbb{O}=\mathbb{O}_{1,\tau}\ \emph{or}\ \mathbb{O}=\mathbb{O}_{(2k+1)\Lambda_1+m\Lambda_2,\tau}\ \emph{when}\ \tilde{w}=\tilde{w}_2\\
                                        -\infty, & \emph{otherwise}.
                                        \end{array}
                                        \right.\]
Where \[\lambda(\tilde{w})=\left\{
                                                             \begin{array}{ll}
                                                              m\Lambda_2+(2k+1)\Lambda_1, & \emph{if}\ \tilde{w}=\tilde{w}_1\\
                                                              m\Lambda_2+2k\Lambda_1, & \emph{if}\ \tilde{w}\in\{\tilde{w}_1,\ \tilde{w}_2\},
                                                             \end{array}
                                                           \right.\]
$\mathbb{O}^<_{\tilde{w}}=\{\mathbb{O}_{4i+2,\tau}|1\leqslant i\leqslant n(\tilde{w})\}\cup\{\mathbb{O}'_{6i,\tau}|0\leqslant i\leqslant m+k\} \cup\{\mathbb{O}'_{6i+4,\tau}|1\leqslant i\leqslant m+k-1\}$, where \[n(\tilde{w})=\left\{
\begin{array}{ll}
m+2k, & \emph{if}\ \tilde{w}=\tilde{w}_1\\
m+2k-4, & \emph{if}\ \tilde{w}\in\{\tilde{w}_1,\ \tilde{w}_2\}.
\end{array}
\right.\]
\end{prop}
\begin{proof}
Since $\tau(t^{k(2\Lambda_1-\Lambda_2)+\Lambda_1-(m+1)\Lambda_2}\tau)\tau^{-1}=t^{k(2\Lambda_1-\Lambda_2)+\Lambda_1+m\Lambda_2}\tau$, $$s_1(t^{-k(2\Lambda_1-\Lambda_2)-\Lambda_1+m\Lambda_2}\tau)s_1=t^{(k+1)(2\Lambda_1-\Lambda_2)-\Lambda_1+m\Lambda_2}\tau,$$ $$\tau(t^{-k\Lambda_2+m(2\Lambda_1-\Lambda_2)}s_1s_2s_1s_2\tau)\tau^{-1}=t^{(k+1)\Lambda_2+m(2\Lambda_1-\Lambda_2)}s_1s_2s_1s_2\tau,$$it is sufficient to consider $t^{k(2\Lambda_1-\Lambda_2)+\Lambda_1+m\Lambda_2}\tau$ and $t^{(k+1)\Lambda_2+m(2\Lambda_1-\Lambda_2)}s_1s_2s_1s_2\tau$.\\
(1) If $\ell(\tilde{w})=6k+4+2m,\ 0\leqslant m\leqslant k$, then we can set $\tilde{w}=t^{k(2\Lambda_1-\Lambda_2)+\Lambda_1+m\Lambda_2}\tau$. Thus, $$T_{\tilde{w}}=[v]\sum_{i=0}^{m-1}T_{\mathbb{O}_{(k+1)\Lambda_2+i(2\Lambda_1-\Lambda_2),\tau}}+T_{\mathbb{O}'_{6k+4,\tau}}.$$\\
(2) We may assume that $\tilde{w}_1=t^{(2k+1)\Lambda_1+m\Lambda_2}\tau,$ $$\tilde{w}_2=t^{(2k+2)(\Lambda_2-\Lambda_1)-(m-1)(2\Lambda_1-\Lambda_2)}s_1s_2s_1s_2\tau,$$ $$\tilde{w}_3=t^{(2k+2)\Lambda_1+(m-1)(2\Lambda_1-\Lambda_2)}s_1s_2s_1s_2\tau.$$ Then $$T_{\tilde{w}_1}\equiv [v]T_{t^{-2k\Lambda_1-m(2\Lambda_1-\Lambda_2)}s_1s_2s_1\tau}+T_{\tilde{w}_2}$$ $$T_{\tilde{w}_2}\equiv [v]T_{t^{(2k+2)\Lambda_1+(m-1)(2\Lambda_1-\Lambda_2)}s_2s_1s_2\tau}+T_{\tilde{w}_3}$$ $$T_{\tilde{w}_3}\equiv[v]T_{t^{(2k+2)\Lambda_1+(m-1)\Lambda_2}s_1s_2s_1\tau}+T_{t^{(2k+1)\Lambda_1+(m-1)\Lambda_2}\tau}.$$Together with (1) and Proposition \ref{2degstau1}, we finish the proof.
\end{proof}

\begin{prop}\label{2deg1tau11}
Let $\tilde{w}\in\mathbb{O}'_{1,\tau}$.\\
(1) If $\tilde{w}$ or $\tau\tilde{w}\tau^{-1}=\tilde{w}_i$ ($i=1,\ 2$), where $\tilde{w}_1=t^{n\Lambda_1}s_1\tau,\ \tilde{w}_2=t^{n(\Lambda_2-\Lambda_1)}s_1\tau$ with $n\in\mathbb{N}$. Then
\[\deg(f_{\tilde{w},\mathbb{O}})=\left\{
\begin{array}{ll}
2, & \emph{if}\ \mathbb{O}\in\{\mathbb{O}_{\lambda,\tau}\ |\ \lambda\in\Lambda^<_{\lambda(\tilde{w})}\}\\
1, & \emph{if}\ \mathbb{O}\in\mathbb{O}^<_{\tilde{w}}\\
0, & \emph{if}\ \mathbb{O}=\mathbb{O}'_{1,\tau}\\
-\infty, & \emph{otherwise},
\end{array}
\right.\]
where \[\lambda(\tilde{w})=\left\{
                                                             \begin{array}{ll}
                                                              n\Lambda_1, & \emph{if}\ \tilde{w}=\tilde{w}_1\\
                                                              (n+1)\Lambda_1, & \emph{if}\ \tilde{w}=\tilde{w}_2,
                                                             \end{array}
                                                           \right.\]
$\mathbb{O}^<_{\tilde{w}}=\{\mathbb{O}_{4i+2,\tau}|1\leqslant i\leqslant n-1\}\cup\{\mathbb{O}'_{6i,\tau}|0\leqslant i\leqslant n'(\tilde{w})\} \cup\{\mathbb{O}'_{6i+4,\tau}|1\leqslant i\leqslant n''(\tilde{w})\}$, where
\[n'(\tilde{w})=\left\{
\begin{array}{ll}
\lfloor\frac{n-1}{2}\rfloor, & \emph{if}\ \tilde{w}=\tilde{w}_1\\
\lfloor\frac{n}{2}\rfloor, & \emph{if}\ \tilde{w}=\tilde{w}_2,
\end{array}
\right.\]
\[n''(\tilde{w})=\left\{
\begin{array}{ll}
\lfloor\frac{n}{2}\rfloor-1, & \emph{if}\ \tilde{w}=\tilde{w}_1\\
\lfloor\frac{n-1}{2}\rfloor, & \emph{if}\ \tilde{w}=\tilde{w}_2.
\end{array}
\right.\]
(2) If $\tilde{w}$ or $\tau\tilde{w}\tau^{-1}=\tilde{w}_i$ ($i=3,\ 4,\ 5,\ 6$), where $\tilde{w}_3=t^{(2\Lambda_1-\Lambda_2)+n\Lambda_1}s_1\tau,\ \tilde{w}_4=t^{m(2\Lambda_1-\Lambda_2)+n\Lambda_1}s_1\tau,\ \tilde{w}_5=t^{(m-1)\Lambda_2+(n+1)\Lambda_1}s_1\tau,\ \tilde{w}_6=t^{(1-m)(2\Lambda_1-\Lambda_2)+n(\Lambda_2-\Lambda_1)}s_1\tau$ where $m,\ n\in\mathbb{N}$ and $m\geqslant2$. Then
\[\deg(f_{\tilde{w},\mathbb{O}})=\left\{
\begin{array}{ll}
4, & \emph{if}\ \mathbb{O}=\mathbb{O}_{\lambda,\tau},\ \lambda\in\Lambda^\leqslant_{\lambda(\tilde{w})}\\
3, & \emph{if}\ \mathbb{O}\in\mathbb{O}^<_{\tilde{w}}\\
2, & \emph{if}\ \mathbb{O}\in\{\mathbb{O}_{1,\tau}\}\cup\mathbb{O}^{(\tilde{w})}\\
1, & \emph{if}\ \mathbb{O}\in\mathbb{O}_{(\tilde{w})}\\
0, & \emph{if}\ \mathbb{O}=\mathbb{O}'_{1,\tau}\\
-\infty, & \emph{otherwise}.
\end{array}
\right.\]
Where \[\lambda(\tilde{w})=\left\{
                                                             \begin{array}{ll}
                                                              \Lambda_2+(n-1)\Lambda_1, & \emph{if}\ \tilde{w}=\tilde{w}_3\\
                                                              m\Lambda_2+(n-1)\Lambda_1, & \emph{if}\ \tilde{w}=\tilde{w}_4\\
                                                              (m-1)\Lambda_2+(n-1)\Lambda_1, & \emph{if}\ \tilde{w}\in\{\tilde{w}_5,\ \tilde{w}_6\},
                                                             \end{array}
                                                           \right.\]
$\mathbb{O}^<_{\tilde{w}}=\{\mathbb{O}_{4i+2,\tau}|1\leqslant i\leqslant n(\tilde{w})\}\cup\{\mathbb{O}'_{6i,\tau}|0\leqslant i\leqslant n'(\tilde{w})\} \cup\{\mathbb{O}'_{6i+4,\tau}|1\leqslant i\leqslant n''(\tilde{w})\}$, where
\[n(\tilde{w})=\left\{
\begin{array}{ll}
n-1, & \emph{if}\ \tilde{w}=\tilde{w}_3\\
m+n-2, & \emph{if}\ \tilde{w}=\tilde{w}_4\\
m+n-3, & \emph{if}\ \tilde{w}\in\{\tilde{w}_5,\ \tilde{w}_6\},
\end{array}
\right.\]
\[n'(\tilde{w})=\left\{
\begin{array}{ll}
\lfloor\frac{n}{2}\rfloor-1, & \emph{if}\ \tilde{w}=\tilde{w}_3\\
m+\lfloor\frac{n}{2}\rfloor-1, & \emph{if}\ \tilde{w}=\tilde{w}_4\\
m+\lfloor\frac{n}{2}\rfloor-2, & \emph{if}\ \tilde{w}\in\{\tilde{w}_5,\ \tilde{w}_6\},
\end{array}
\right.\]
\[n''(\tilde{w})=\left\{
\begin{array}{ll}
\lfloor\frac{n-1}{2}\rfloor, & \emph{if}\ \tilde{w}=\tilde{w}_3\\
m+\lfloor\frac{n-1}{2}\rfloor-1, & \emph{if}\ \tilde{w}=\tilde{w}_4\\
m+\lfloor\frac{n-1}{2}\rfloor-2, & \emph{if}\ \tilde{w}\in\{\tilde{w}_5,\ \tilde{w}_6\},
\end{array}
\right.\]
\[\mathbb{O}^{(\tilde{w})}=\left\{
\begin{array}{ll}
\{\mathbb{O}_{\Lambda_2,\tau},\ \mathbb{O}_{\Lambda_2+n\Lambda_1,\tau}\}, & \emph{if}\ \tilde{w}=\tilde{w}_3\\
\{\mathbb{O}_{i\Lambda_2,\tau}|1\leqslant i\leqslant m\}\cup\{\mathbb{O}_{\lambda,\tau}|\lambda\in\Lambda^{2,\leqslant}_{m\Lambda_2+n\Lambda_1}\}, & \emph{if}\ \tilde{w}=\tilde{w}_4\\
\{\mathbb{O}_{i\Lambda_2,\tau}|1\leqslant i\leqslant m-1\}\cup\{\mathbb{O}_{\lambda,\tau}|\lambda\in\Lambda^{2,\leqslant}_{(m-1)\Lambda_2+n\Lambda_1}\cup\Lambda^{1,2}_{(m-1)\Lambda_2+(n+1)\Lambda_1}\}, & \emph{if}\ \tilde{w}=\tilde{w}_5\\
\{\mathbb{O}_{i\Lambda_2,\tau}|1\leqslant i\leqslant m-1\}\cup\{\mathbb{O}_{\lambda,\tau}|\lambda\in\Lambda^{2,\leqslant}_{(m-1)\Lambda_2+n\Lambda_1}\cup\Lambda^{1,<}_{(m-1)\Lambda_2+(n+1)\Lambda_1}\}, & \emph{if}\ \tilde{w}=\tilde{w}_6,
\end{array}
\right.\]
\[\mathbb{O}_{(\tilde{w})}=\left\{
\begin{array}{ll}
\{\mathbb{O}_{4n+2,\tau},\ \mathbb{O}'_{6\lfloor\frac{n}{2}\rfloor,\tau}\}, & \emph{if}\ \tilde{w}=\tilde{w}_3\\
\{\mathbb{O}_{4(m+n-1)+2,\tau},\ \mathbb{O}'_{6(m+\frac{n}{2}),\tau} (n\ \emph{is}\ \emph{even}),\ \mathbb{O}'_{6(m+\frac{n-1}{2})+4,\tau} (n\ \emph{is}\ \emph{odd})\}, & \emph{if}\ \tilde{w}=\tilde{w}_4\\
\{\mathbb{O}_{4(m+n-2)+2,\tau}\}\cup\mathbb{O}_{(\tilde{w}_4)}, & \emph{if}\ \tilde{w}=\tilde{w}_5\\
\{\mathbb{O}_{4(m+n-2)+2,\tau},\ \mathbb{O}'_{6(m+\frac{n}{2}),\tau} (n\ \emph{is}\ \emph{even}),\ \mathbb{O}'_{6(m+\frac{n-1}{2})+4,\tau} (n\ \emph{is}\ \emph{odd})\}, & \emph{if}\ \tilde{w}=\tilde{w}_6.
\end{array}
\right.\]
(3) If $\tilde{w}$ or $\tau\tilde{w}\tau^{-1}=t^{m\Lambda_2+n(\Lambda_2-\Lambda_1)}s_1\tau$ where $m,\ n\in\mathbb{N}_+$, then\[\deg(f_{\tilde{w},\mathbb{O}})=\left\{
\begin{array}{ll}
4, & \emph{if}\ \mathbb{O}=\mathbb{O}_{\lambda,\tau},\ \lambda\in\Lambda^\leqslant_{m\Lambda_2+n\Lambda_1}\\
3, & \emph{if}\ \mathbb{O}\in\mathbb{O}^<_{\tilde{w}}\\
2 & \emph{if}\ \mathbb{O}\in\{\mathbb{O}_{1,\tau},\ \mathbb{O}_{i\Lambda_2,\tau}\ |\ 1\leqslant i\leqslant m\}\\
0, & \emph{if}\ \mathbb{O}=\mathbb{O}'_{1,\tau}\\
-\infty, & \emph{otherwise},
\end{array}
\right.\]
where $\mathbb{O}^<_{\tilde{w}}=\{\mathbb{O}_{4i+2,\tau}|0\leqslant i\leqslant m+n-1\}\cup
 \{\mathbb{O}'_{6i,\tau}|1\leqslant i\leqslant m+\lfloor\frac{n-1}{2}\rfloor\}\cup\{\mathbb{O}'_{6i+4,\tau}|0\leqslant i\leqslant m+\lfloor\frac{n}{2}\rfloor-1\}.$\\
(4) If $\tilde{w}$ or $\tau\tilde{w}\tau^{-1}=t^{m\Lambda_2}s_1\tau$ where $m\in\mathbb{N}_+$, then\[\deg(f_{\tilde{w},\mathbb{O}})=\left\{
\begin{array}{ll}
4, & \emph{if}\ \mathbb{O}=\mathbb{O}_{\lambda,\tau},\ \lambda\in\Lambda^\leqslant_{m\Lambda_2}\\
3, & \emph{if}\ \mathbb{O}\in\mathbb{O}^<_{\tilde{w}}\\
2 & \emph{if}\ \mathbb{O}\in\{\mathbb{O}_{1,\tau},\ \mathbb{O}_{i\Lambda_2,\tau}\ |\ 1\leqslant i\leqslant m\}\cup\\
& \{\mathbb{O}_{\lambda,\tau}|\lambda\in\Lambda^{1,<}_{(m-1)\Lambda_2+2\Lambda_1}\cup\Lambda^{2,\leqslant}_{(m-1)\Lambda_2+\Lambda_1}\}\\
1 & \emph{if}\ \mathbb{O}=\mathbb{O}'_{6m+4,\tau}\\
0, & \emph{if}\ \mathbb{O}=\mathbb{O}'_{1,\tau}\\
-\infty, & \emph{otherwise},
\end{array}
\right.\]
where $\mathbb{O}^<_{\tilde{w}}=\{\mathbb{O}_{4i+2,\tau}|0\leqslant i\leqslant m-1\}\cup\{\mathbb{O}'_{6i,\tau}|1\leqslant i\leqslant m-2\}\cup\{\mathbb{O}'_{6i+4,\tau}|0\leqslant i\leqslant m-1\}$
\end{prop}
\begin{proof}
Since $\tau t^{m(2\Lambda_1-\Lambda_2)-n\Lambda_2}s_1=t^{m(2\Lambda_1-\Lambda_2)+n\Lambda_2}s_1\tau$ and $$\tau t^{(\Lambda_2-\Lambda_1)+m(2\Lambda_1-\Lambda_2)-(n+1)\Lambda_2}s_1=t^{(\Lambda_2-\Lambda_1)+m(2\Lambda_1-\Lambda_2)+n\Lambda_2}s_1\tau,$$it is sufficient to consider $\tilde{w}=t^{m(2\Lambda_1-\Lambda_2)+n(\Lambda_2-\Lambda_1)}s_1\tau$ with $m\in\mathbb{Z},\ n\in\mathbb{N}$.

(1) We have $T_{t^{n\Lambda_1}s_1\tau}\equiv [v]T_{\mathbb{O}_{4(n-1)+2,\tau}}+[v]T_{t^{(n-1)\Lambda_1}\tau}+T_{t^{(n-1)\Lambda_1}s_1\tau}$ and $$T_{t^{n(\Lambda_2-\Lambda_1)}s_1\tau}\equiv [v]T_{t^{n\Lambda_1}\tau}+T_{t^{n\Lambda_1}s_1\tau}.$$By Propositions \ref{deg6ktau}, \ref{deg6k+4tau}, this case is proved.

(2) Since $T_{t^{(2\Lambda_1-\Lambda_2)+n\Lambda_1}s_1\tau}\equiv [v]T_{t^{\Lambda_2+n\Lambda_1}s_2s_1\tau}+T_{t^{(n+1)\Lambda_1}s_1\tau}$, $$T_{\tilde{w}_4}\equiv [v]T_{t^{m\Lambda_2+n\Lambda_1}s_2s_1\tau}+T_{\tilde{w}_5},\ T_{\tilde{w}_5}\equiv [v]T_{t^{(m-1)\Lambda_2+(n+1)\Lambda_1}s_1s_2\tau}+T_{\tilde{w}_6},$$ together with previous cases, (2) is proved.

(3) By previous arguments, it follows from Propositions \ref{deg4k+2tau}, \ref{deg6ktau},\ \ref{deg6k+4tau} and $T_{t^{m\Lambda_2+n(\Lambda_2-\Lambda_1)}s_1\tau}\equiv [v]T_{t^{m\Lambda_2+n\Lambda_1}\tau}+T_{t^{m\Lambda_2+n\Lambda_1}s_1\tau}$.

(4) In this case, we just have to note the relation: $T_{t^{m\Lambda_2}s_1\tau}\equiv [v]T_{t^{m\Lambda_2}s_1s_2\tau}+T_{t^{(1-m)(2\Lambda_1-\Lambda_2)+(\Lambda_2-\Lambda_1)}s_1\tau}$.
\end{proof}

\section{Applications}
\subsection{The emptiness/nonemptiness and dimension formulas}\label{App1}
We will keep the notations as before. Let $b\in G=Sp_{4}(L)$ and $\tilde{w}\in\widetilde{W}$. By results obtained in Section \ref{Classcomput} and the ``Dimension $=$ Degree'' theorem, we will give explicit descriptions on emptiness/nonemptiness and dimension formula of the affine Deligne-Lusztig varieties $X_{\tilde{w}}(b)$.
\begin{thm}\label{2pattern1}
Let $b\in G$ and $\tilde{w}\in\widetilde{W}$. If $b=1$, or $b$ corresponds to $\mathbb{O}$, where $\mathbb{O}\in\{\mathbb{O}_2,\ \mathbb{O}'_2,\ \mathbb{O}_{s_{1212}},\ \mathbb{O}'_1,\ \mathbb{O}_1,\ \mathbb{O}_0\}$ then
\[X_{\tilde{w}}(b)\neq\emptyset\Longleftrightarrow\tilde{w}\in\left\{
\begin{array}{ll}
\mathbb{O}_0\cup\mathbb{O}_1\cup\mathbb{O}'_1\cup\mathbb{O}_2\cup\mathbb{O}'_2\cup\mathbb{O}_{s_{1212}} \\
\mathbb{O}_{4k+1}\ \emph{and}\ \ell(\tilde{w})\geqslant6k+3\\
\mathbb{O}'_{6k+1}\ \emph{and}\ \ell(\tilde{w})\geqslant8k+3\\
\mathbb{O}'_{6(k-1)+3}\ \emph{and}\ \ell(\tilde{w})\geqslant8(k-1)+7,
\end{array}
\right.\]where $k\in\mathbb{N}_+$. Moreover, if $X_{\tilde{w}}(b)\neq\emptyset$ we have
\[\dim X_{\tilde{w}}(b)=\left\{
\begin{array}{ll}
0, & \emph{if}\ \tilde{w}\in\mathbb{O}_0\\
1, & \emph{if}\ \tilde{w}\in\mathbb{O}_1\cup\mathbb{O}'_1\ \emph{and}\ \ell(\tilde{w})=1\\
\frac{\ell(\tilde{w})+2}{2}, & \emph{if}\ \tilde{w}\in\mathbb{O}_2\ \emph{or}\ \tilde{w}\in\Psi(\mathbb{O}'_2,b)\cup\Psi(\mathbb{O}_{s_{1212}},b)\\
\frac{\ell(\tilde{w})+3}{2}, & \emph{if}\ \tilde{w}\in\mathbb{O}_1\cup\mathbb{O}'_1\ \emph{and}\ \ell(\tilde{w})\geqslant3\ \emph{or}\\
                             & \tilde{w}\in\mathbb{O}_{4k+1}\ \emph{and}\ \ell(\tilde{w})\geqslant6k+3\ \emph{or}\\
                             & \tilde{w}\in\mathbb{O}'_{6k+1}\ \emph{and}\ \ell(\tilde{w})\geqslant8k+3\ \emph{or}\\
                             & \tilde{w}\in\mathbb{O}'_{6(k-1)+3}\ \emph{and}\ \ell(\tilde{w})\geqslant8(k-1)+7\\
\frac{\ell(\tilde{w})+4}{2}, & \emph{if}\ \tilde{w}\in(\mathbb{O}'_2-\Psi(\mathbb{O}'_2,b))\cup(\mathbb{O}_{s_{1212}}-\Psi(\mathbb{O}_{s_{1212}},b))
\end{array}
\right.\]
Here, we set $w=s_1s_2s_1s_2$, $$\Psi_0(\mathbb{O}'_2,b)=\{t^{\Lambda_2-\Lambda_1}w,\ t^{\pm i(2\Lambda_1-\Lambda_2)+\Lambda_1}w,\ t^{i\Lambda_2+\Lambda_1}w\ |\ i\in\mathbb{N}\}$$ $\Psi(\mathbb{O}'_2,b)=\Psi_0(\mathbb{O}'_2,b)\cup\tau\cdot(\Psi_0(\mathbb{O}'_2,b))$ and $$\Psi_0(\mathbb{O}_{s_{1212}},b)=\{t^{2\Lambda_1}w,\ t^{\Lambda_2}w,\ t^{2(\Lambda_2-\Lambda_1)}w,\ t^{\Lambda_2+2\Lambda_1}w\}$$
$\Psi(\mathbb{O}_{s_{1212}},b)=\Psi_0(\mathbb{O}_{s_{1212}},b)\cup\tau\cdot(\Psi_0(\mathbb{O}_{s_{1212}},b))$.
\end{thm}
\begin{proof}
We check the degrees of the class polynomials in Section \ref{Classcomput}, and the theorem follows directly from the ``Dimension $=$ Degree" Theorem.
\end{proof}
In the rest of this section, we give the emptiness/nonemptiness of $X_{\tilde{w}}(b)$ and the corresponding dimension formula.
\begin{enumerate}
\item \label{2pattern2} If $b$ corresponds to $\mathbb{O}_{\lambda_0}$, where $\lambda_0=m_0\Lambda_2+n_0\Lambda_1\in Q_{++}\ (i.e.\ m_0,\ n_0\in\mathbb{N}_+)$, then $X_{\tilde{w}}(b)\neq\emptyset$ if and only if:
\begin{enumerate}
\item $\tilde{w}\in\mathbb{O}_{m_0\Lambda_2+n_0\Lambda_1};$
\item $\tilde{w}$ or $\tau\tilde{w}\tau^{-1}$ lies in Proposition \ref{1degso2} such that $\lambda_0\in\lambda(\tilde{w});$
\item $\tilde{w}$ or $\tau\tilde{w}\tau^{-1}$ lies in Proposition \ref{deg1} such that $\lambda_0\in\Lambda^<_{\tilde{w}-\Lambda_1};$
\item $\tilde{w}$ lies in Proposition \ref{deg4k+1} (1) such that $\lambda_0\in\{k\Lambda_1+i(\Lambda_2-\Lambda_1)\ |\ 1\leqslant i\leqslant m-1\};$
\item $\tilde{w}$ lies in Proposition \ref{deg4k+1} (2) such that $\lambda_0\in\Lambda^{\leqslant}_{k\Lambda_2+(m-1)\Lambda_1};$
\item $\tilde{w}$ lies in Proposition \ref{deg1p} such that $\lambda_0\in\Lambda^{\leqslant}_{\lambda(\tilde{w})};$
\item $\tilde{w}$ lies in Proposition \ref{deg6k+1} (1) such that $\lambda_0\in\{k\Lambda_2+i(2\Lambda_1-\Lambda_2)\ |\ 1\leqslant i\leqslant m-1\};$
\item $\tilde{w}$ lies in Proposition \ref{deg6k+1} (2) such that $\lambda_0\in\Lambda^{\leqslant}_{\lambda(\tilde{w})}$, or $\lambda_0=2k\Lambda_1+(m-1)\Lambda_2;$
\item $\tilde{w}$ lies in Proposition \ref{deg6k+3} (1) such that $\lambda_0\in\{\Lambda_1+k\Lambda_2+i(2\Lambda_1-\Lambda_2)\ |\ 1\leqslant i\leqslant m-1\};$
\item $\tilde{w}$ lies in Proposition \ref{deg6k+3} (2) such that $\lambda_0\in\Lambda^{\leqslant}_{\lambda(\tilde{w})}$, or $\lambda_0=(2k+1)\Lambda_1+(m-1)\Lambda_2;$
\item $\tilde{w}$ lies in Propositions \ref{2degso2}, \ref{2degs1212} such that $\lambda_0\in\lambda(\tilde{w})$, or $\mathbb{O}_{(\tilde{w})}.$
\end{enumerate}
Moreover, if $X_{\tilde{w}}(b)\neq\emptyset$ we have $$\dim X_{\tilde{w}}(b)=\frac{1}{2}(\ell(\tilde{w})+\ell(\mathbb{O}_{\lambda_0})+\deg (f_{\tilde{w},\mathbb{O}_{\lambda_0}}))-\langle\bar{\nu}_b,\ 2\rho\rangle.$$

\item \label{2pattern3} If $b$ corresponds to $\mathbb{O}_{\lambda_0}$, where $\lambda_0=m_0\Lambda_2\ m_0\in\mathbb{N}_+$, then $X_{\tilde{w}}(b)\neq\emptyset$ if and only if:
\begin{enumerate}
\item $\tilde{w}\in\mathbb{O}_{m_0\Lambda_2};$
\item $\tilde{w}\in\mathbb{O}_{4m_0+1}$ with $\ell(\tilde{w})\geqslant6m_0+1;$
\item $\tilde{w}\in\mathbb{O}'_1$ with $\ell(\tilde{w})\geqslant6m-5$ and $m\geqslant m_0+1;$
\item $\tilde{w}$ or $\tau\tilde{w}\tau^{-1}$ equals to $\tilde{w}_5$ in Propositions \ref{2degso2}, \ref{2degs1212} and $m\geqslant m_0;$
\item $\tilde{w}$ or $\tau\tilde{w}\tau^{-1}$ is contained in $\{\tilde{w}_i\ |\ i=6,\ 7,\ 8\}$ in Proposition \ref{2degso2}, \ref{2degs1212} and $m\geqslant m_0+1.$
\end{enumerate}
If $X_{\tilde{w}}(b)\neq\emptyset$, then $$\dim X_{\tilde{w}}(b)=\frac{1}{2}(\ell(\tilde{w})+\ell(\mathbb{O}_{\lambda_0})+\deg (f_{\tilde{w},\mathbb{O}_{\lambda_0}}))-\langle\bar{\nu}_b,\ 2\rho\rangle.$$

\item \label{2pattern4} If $b$ corresponds to $\mathbb{O}_{\lambda_0}$, where $\lambda_0=n_0\Lambda_1\ n_0\in\mathbb{N}_+$, then $X_{\tilde{w}}(b)\neq\emptyset$ if and only if:
\begin{enumerate}
\item $\tilde{w}\in\mathbb{O}_{n_0\Lambda_1};$
\item $\tilde{w}\in\mathbb{O}_{1}$ and $\ell(\tilde{w})=4n-3$ or $4n-1$ with $n\geqslant n_0+1;$
\item $n_0=2k$ and $\tilde{w}\in\mathbb{O}'_{6k+1}$ with $\ell(\tilde{w})\geqslant8k+1;$
\item $n_0=2k+1$ and $\tilde{w}\in\mathbb{O}'_{6k+3}$ with $\ell(\tilde{w})\geqslant8k+5;$
\item $\tilde{w}$ or $\tau\tilde{w}\tau^{-1}$ is contained in $\{\tilde{w}_i\ |\ i=1,\ 2,\ 5,\ 6,\ 7\}$ in Proposition \ref{2degso2} and $n_0=2i-1,\ 1\leqslant i\leqslant n;$
\item $\tilde{w}$ or $\tau\tilde{w}\tau^{-1}$ is contained in $\{\tilde{w}_i\ |\ i=3,\ 4\}$ in Proposition \ref{2degso2} and $n_0=2i-1,\ 1\leqslant i\leqslant n-1;$
\item $\tilde{w}$ or $\tau\tilde{w}\tau^{-1}=\tilde{w}_8$ in Proposition \ref{2degso2} and $n_0=2i-1,\ 1\leqslant i\leqslant n+1;$
\item $\tilde{w}$ or $\tau\tilde{w}\tau^{-1}$ is contained in $\{\tilde{w}_i\ |\ i=1,\ 2,\ 5,\ 6,\ 7\}$ in Proposition \ref{2degs1212} and $n_0=2i,\ 1\leqslant i\leqslant n-1;$
\item $\tilde{w}$ or $\tau\tilde{w}\tau^{-1}$ is contained in $\{\tilde{w}_i\ |\ i=3,\ 4\}$ in Proposition \ref{2degs1212} and $n_0=2i,\ 1\leqslant i\leqslant n-2;$
\item $\tilde{w}$ or $\tau\tilde{w}\tau^{-1}=\tilde{w}_8$ in Proposition \ref{2degs1212} and $n_0=2i,\ 1\leqslant i\leqslant n.$
\end{enumerate}
If $X_{\tilde{w}}(b)\neq\emptyset$, then $$\dim X_{\tilde{w}}(b)=\frac{1}{2}(\ell(\tilde{w})+\ell(\mathbb{O}_{\lambda_0})+\deg (f_{\tilde{w},\mathbb{O}_{\lambda_0}}))-\langle\bar{\nu}_b,\ 2\rho\rangle.$$

\item \label{2pattern5} If $b$ corresponds to $\mathbb{O}_{4k_0+1}$, where $k_0\in\mathbb{N}_+$, then $X_{\tilde{w}}(b)\neq\emptyset$ if and only if
\begin{enumerate}
\item $\tilde{w}\in\mathbb{O}_{k_0\Lambda_1}$ or $\mathbb{O}_{4k_0+1};$
\item $\tilde{w}$ or $\tau\tilde{w}\tau^{-1}$ lies in Proposition \ref{1degso2} such that $k_0\leqslant n(\tilde{w});$
\item $\tilde{w}\in\mathbb{O}_1$ with $\ell(\tilde{w})\geqslant4k_0+1;$
\item $\tilde{w}\in\mathbb{O}_{4k+1}$, such that $k\neq k_0$ and $m\in\mathbb{N}_+$ with $\ell(\tilde{w})\geqslant6k+4m-1$ and $m+k\geqslant k_0+1;$
\item $\tilde{w}\in\mathbb{O}'_1$ and $\ell(\tilde{w})\geqslant6k_0+3;$
\item $\tilde{w}\in\mathbb{O}'_{6k+1}$, where $2k=k_0$ with $\ell(\tilde{w})\geqslant4n_0+1;$
\item $\tilde{w}\in\mathbb{O}'_{6k+1}$, where $2k\neq k_0$ such that $\ell(\tilde{w})\geqslant8k+6m-3$ with $m+2k-1\geqslant k_0;$
\item $\tilde{w}\in\mathbb{O}'_{6k+3}$, where $2k+1=k_0$ with $\ell(\tilde{w})\geqslant8k+5;$
\item $\tilde{w}\in\mathbb{O}'_{6k+3}$, where $2k+1\neq k_0$ such that $\ell(\tilde{w})\geqslant8k+6m+1$ with $m+2k\geqslant k_0;$
\item $\tilde{w}$ or $\tau\tilde{w}\tau^{-1}$ lies in Proposition \ref{2degso2} or Proposition \ref{2degs1212} such that $k_0\leqslant n(\tilde{w})$ or $\mathbb{O}_{k_0\Lambda_1}\in\mathbb{O}_{(\tilde{w})}$, or $\mathbb{O}_{4k_0+1}\in\mathbb{O}^{(\tilde{w})}.$
\end{enumerate}
If $X_{\tilde{w}}(b)\neq\emptyset$, then $$\dim X_{\tilde{w}}(b)=\frac{1}{2}\max_{\mathbb{O}}(\ell(\tilde{w})+\ell(\mathbb{O})+\deg (f_{\tilde{w},\mathbb{O}}))-\langle\bar{\nu}_b,\ 2\rho\rangle,$$where $\mathbb{O}\in\{\mathbb{O}_{4k_0+1},\ \mathbb{O}_{k_0\Lambda_1}\}$.

\item \label{2pattern6} If $b$ corresponds to $\mathbb{O}'_{6k_0+1}$, where $k_0\in\mathbb{N}_+$, then $X_{\tilde{w}}(b)\neq\emptyset$ if and only if
\begin{enumerate}
\item $\tilde{w}\in\mathbb{O}_{k_0\Lambda_2}$ or $\mathbb{O}'_{6k_0+1};$
\item $\tilde{w}$ or $\tau\tilde{w}\tau^{-1}$ lies in Proposition \ref{1degso2} such that $k_0\leqslant n''(\tilde{w});$
\item $\tilde{w}\in\mathbb{O}_1$ and $\ell(\tilde{w})=4m-1$ with $\lfloor\frac{m-1}{2}\rfloor\geqslant k_0$ or $\ell(\tilde{w})=4m-3$ with $\lfloor\frac{m-2}{2}\rfloor\geqslant k_0;$
\item $\tilde{w}\in\mathbb{O}_{4k+1}$ and $\ell(\tilde{w})=6k+4m-1$ or $6k+4m+1$ with $k+\lfloor\frac{m-1}{2}\rfloor\geqslant k_0;$
\item $\tilde{w}\in\mathbb{O}'_1$ and $\ell(\tilde{w})\geqslant6k_0+1;$
\item $\tilde{w}\in\mathbb{O}'_{6k+1}$ with $\ell(\tilde{w})=8k+6m-1$ or $8k+6m-3$ such that $m+k-1\geqslant k_0,$ or $\ell(\tilde{w})=8k+6m-5$ such that $m+k-2\geqslant k_0;$
\item $\tilde{w}\in\mathbb{O}'_{6k+3}$ with $\ell(\tilde{w})=8k+6m+3$ or $8k+6m+1$ such that with $m+k\geqslant k_0,$ or $\ell(\tilde{w})=8k+6m-1$ with $m+k-1\geqslant k_0;$
\item $\tilde{w}$ or $\tau\tilde{w}\tau^{-1}$ lies in Proposition \ref{2degso2} or Proposition \ref{2degs1212} such that $k_0\leqslant n''(\tilde{w})$ or $\mathbb{O}_{k_0\Lambda_2}\in\mathbb{O}_{(\tilde{w})}$, or $\mathbb{O}'_{6k_0+1}\in\mathbb{O}^{(\tilde{w})}.$
\end{enumerate}
If $X_{\tilde{w}}(b)\neq\emptyset$, then $$\dim X_{\tilde{w}}(b)=\frac{1}{2}\max_{\mathbb{O}}(\ell(\tilde{w})+\ell(\mathbb{O})+\deg (f_{\tilde{w},\mathbb{O}}))-\langle\bar{\nu}_b,\ 2\rho\rangle,$$where $\mathbb{O}\in\{\mathbb{O}'_{6k_0+1},\ \mathbb{O}_{k_0\Lambda_2}\}$.

\item \label{2pattern7} If $b$ corresponds to $\mathbb{O}'_{6k_0+3}$, where $k_0\in\mathbb{N}$, then $X_{\tilde{w}}(b)\neq\emptyset$ if and only if
\begin{enumerate}
\item $\tilde{w}$ or $\tau\tilde{w}\tau^{-1}$ lies in Proposition \ref{1degso2} such that $k_0\leqslant n'(\tilde{w});$
\item $\tilde{w}\in\mathbb{O}_1$ and $\ell(\tilde{w})=4m-1$ with $\lfloor\frac{m-2}{2}\rfloor\geqslant k_0$ or $\ell(\tilde{w})=4m-3$ with $\lfloor\frac{m-3}{2}\rfloor\geqslant k_0;$
\item $\tilde{w}\in\mathbb{O}_{4k+1}$ and $\ell(\tilde{w})=6k+4m-1$ or $6k+4m+1$ with $k+\lfloor\frac{m}{2}-1\rfloor\geqslant k_0;$
\item $\tilde{w}\in\mathbb{O}'_1$ with $\ell(\tilde{w})=6m-1$ such that $m-1\geqslant k_0$ or $\ell(\tilde{w})=6m-3$ or $6m-5$ such that $m-2\geqslant k_0;$
\item $\tilde{w}\in\mathbb{O}'_{6k+1}$ with $\ell(\tilde{w})=8k+6m-1$ such that $m+k-1\geqslant k_0,$ or $\ell(\tilde{w})=8k+6m-3$ or $8k+6m-5$ such that $m+k-2\geqslant k_0;$
\item $\tilde{w}\in\mathbb{O}'_{6k_0+3};$
\item $\tilde{w}\in\mathbb{O}'_{6k+3}$ where $k\neq k_0$ with $\ell(\tilde{w})=8k+6m+3$ or $8k+6m\pm1$ such that with $m+k-1\geqslant k_0;$
\item $\tilde{w}$ or $\tau\tilde{w}\tau^{-1}$ lies in Proposition \ref{2degso2} or Proposition \ref{2degs1212} such that $k_0\leqslant n'(\tilde{w})$ or $\mathbb{O}'_{6k_0+3}\in\mathbb{O}^{(\tilde{w})}.$
\end{enumerate}
If $X_{\tilde{w}}(b)\neq\emptyset$, then $$\dim X_{\tilde{w}}(b)=\frac{1}{2}(\ell(\tilde{w})+\ell(\mathbb{O}'_{6k_0+3})+\deg (f_{\tilde{w},\mathbb{O}'_{6k_0+3}}))-\langle\bar{\nu}_b,\ 2\rho\rangle.$$

\item \label{2pattern8} If $b$ corresponds to $\mathbb{O}$, where $\mathbb{O}\in\{\mathbb{O}_{1,\tau},\ \mathbb{O}'_{1,\tau},\ \mathbb{O}'_{0,\tau}\}$ then
\[X_{\tilde{w}}(b)\neq\emptyset\Longleftrightarrow\tilde{w}\in\left\{
\begin{array}{ll}
\mathbb{O}_{1,\tau}\cup\mathbb{O}'_{1,\tau}\cup\mathbb{O}'_{0,\tau}\\
\mathbb{O}_{4k+2,\tau}\ (k\in\mathbb{N})\ \emph{and}\ \ell(\tilde{w})\geqslant6k+4\\
\mathbb{O}'_{6k+4,\tau}\ (k\in\mathbb{N})\ \emph{and}\ \ell(\tilde{w})\geqslant8k\\
\mathbb{O}'_{6k}\ (k\in\mathbb{N}_+)\ \emph{and}\ \ell(\tilde{w})\geqslant8k+2.
\end{array}
\right.\]Moreover, if $X_{\tilde{w}}(b)\neq\emptyset$ we have
\[\dim X_{\tilde{w}}(b)=\left\{
\begin{array}{ll}
0, & \tilde{w}\in\mathbb{O}'_{0,\tau}\ \emph{and}\ \ell(\tilde{w})=0\\
\frac{\ell(\tilde{w})+1}{2}, & \tilde{w}\in\mathbb{O}_{1,\tau}\cup\{t^{n\Lambda_1}s_1\tau,\ t^{n(\Lambda_2-\Lambda_1)}s_1\tau\ |\ n\in\mathbb{Z}\}\\
\frac{\ell(\tilde{w})+2}{2}, & \tilde{w}\in\{\tilde{w}\in\mathbb{O}_{4k+2,\tau}|\ell(\tilde{w})\geqslant6k+4\}\cup\{\tilde{w}\in\mathbb{O}'_{0,\tau}|\ell(\tilde{w})>0\}\\
                             & \cup\{\tilde{w}\in\mathbb{O}'_{6k,\tau}|\ell(\tilde{w})\geqslant8k+2\}\cup\{\tilde{w}\in\mathbb{O}'_{6k+4,\tau}|\ell(\tilde{w})\geqslant8k\}\\
\frac{\ell(\tilde{w})+3}{2}, & \tilde{w}\in\mathbb{O}'_{1,\tau}\setminus\{t^{n\Lambda_1}s_1\tau,\ t^{n(\Lambda_2-\Lambda_1)}s_1\tau\ |\ n\in\mathbb{Z}\}.
\end{array}
\right.\]

\item \label{2pattern9} If $b$ corresponds to $\mathbb{O}_{\lambda,\tau}$, where $\lambda=(m'+1)\Lambda_2,\ m'\in\mathbb{N}$. Then $X_{\tilde{w}}(b)\neq\emptyset$ if and only if
\begin{enumerate}
\item $\tilde{w}\in\mathbb{O}_{(m'+1)\Lambda_2,\tau}$ or $\mathbb{O}'_{6m'+4,\tau},$ or $\tilde{w}\in\mathbb{O}_{4m'+2,\tau}$ such that $\ell(\tilde{w})\geqslant6m'+4;$
\item $\tilde{w}$ or $\tau\tilde{w}\tau^{-1}$ lies in Proposition \ref{2degstau1} with $m'\leqslant n''(\tilde{w}),$ or $\tilde{w}$ lies in Proposition \ref{deg4k+2tau} with $k\neq m'$ such that $m'\leqslant n''(\tilde{w});$
\item $\tilde{w}$ is in Proposition \ref{deg0tau} with $m'\leqslant n''(\tilde{w}$ or $\mathbb{O}_{(m'+1)\Lambda_2,\tau}\in\mathbb{O}(\tilde{w});$
\item $\tilde{w}$ is in Propositions \ref{deg6ktau} (any $k$), \ref{deg6k+4tau} $(k\neq m')$ with $m'\leqslant n''(\tilde{w});$
\item $\tilde{w}$ or $\tau\tilde{w}\tau^{-1}$ is in Proposition \ref{2deg1tau11} (1) with $m'\leqslant n''(\tilde{w});$
\item $\tilde{w}$ or $\tau\tilde{w}\tau^{-1}$ is in Proposition \ref{2deg1tau11} (2) with $m'\leqslant n''(\tilde{w}),$ or $\mathbb{O}_{(m'+1)\Lambda_2,\tau}\in\mathbb{O}^{(\tilde{w})},$ or $\mathbb{O}'_{6m'+4}\in\mathbb{O}_{(\tilde{w})};$
\item $\tilde{w}$ or $\tau\tilde{w}\tau=t^{m\Lambda_2+n(\Lambda_2-\Lambda_1)}s_1\tau$ and $m+\lfloor\frac{n}{2}\rfloor\geqslant m'+1$;
\item $\tilde{w}$ or $\tau\tilde{w}\tau=t^{m\Lambda_2}s_1\tau$ and $m\geqslant m'$.
\end{enumerate}
If $X_{\tilde{w}}(b)\neq\emptyset$, then $$\dim X_{\tilde{w}}(b)=\frac{1}{2}\max_{\mathbb{O}}(\ell(\tilde{w})+\ell(\mathbb{O})+\deg (f_{\tilde{w},\mathbb{O}}))-\langle\bar{\nu}_b,\ 2\rho\rangle,$$where $\mathbb{O}\in\{\mathbb{O}'_{6m'+4,\tau},\ \mathbb{O}_{(m'+1)\Lambda_2,\tau}\}$.

\item \label{2pattern10} If $b$ corresponds to $\mathbb{O}_{\lambda,\tau}$, where $\lambda=(m'+1)\Lambda_2+n'\Lambda_2,\ m'\in\mathbb{N},\ n'\in\mathbb{N}_+$. Then $X_{\tilde{w}}(b)\neq\emptyset$ if and only if
\begin{enumerate}
\item $\tilde{w}\in\mathbb{O}_{\lambda,\tau}$;
\item $\tilde{w}$ or $\tau\tilde{w}\tau^{-1}$ is in Proposition \ref{2degstau1} such that $\lambda\in\lambda(\tilde{w});$
\item $\tilde{w}$ is in Proposition \ref{deg4k+2tau} (1) such that $\lambda\in\{(k+1)\Lambda_1+j(\Lambda_2-\Lambda_1)\ |\ 0\leqslant j\leqslant i\};$
\item $\tilde{w}$ is in Proposition \ref{deg4k+2tau} (2) such that $\lambda\in\Lambda^\leqslant_{\lambda(\tilde{w})}$ or $\lambda=(k+1)\Lambda_2+m\Lambda_1;$
\item $\tilde{w}$ is in Proposition \ref{deg0tau} such that $\lambda\in\Lambda^\leqslant_{\lambda(\tilde{w})};$
\item $\tilde{w}$ is in Proposition \ref{deg6ktau} (1) such that $\lambda\in\{\Lambda_1+k\Lambda_1+i(2\Lambda_1-\Lambda_2)\ |\ 0\leqslant i\leqslant m-1\};$
\item $\tilde{w}$ is in Proposition \ref{deg6ktau} (2) such that $\lambda\in\Lambda^\leqslant_{\lambda(\tilde{w})}$ or $\lambda=2k\Lambda_1+m\Lambda_2;$
\item $\tilde{w}$ is in Proposition \ref{deg6k+4tau} (1) such that $\lambda\in\{(k+1)\Lambda_2+i(2\Lambda_1-\Lambda_2)\ |\ 0\leqslant i\leqslant m\};$
\item $\tilde{w}$ is in Proposition \ref{deg6k+4tau} (2) such that $\lambda\in\Lambda^\leqslant_{\lambda(\tilde{w})}$ or $\lambda=(2k+1)\Lambda_1+m\Lambda_2;$
\item $\tilde{w}$ or $\tau\tilde{w}\tau^{-1}$ is in Proposition \ref{2deg1tau11} (1) with $\lambda\in\Lambda^<_{\lambda(\tilde{w})};$
\item $\tilde{w}$ or $\tau\tilde{w}\tau^{-1}$ is in Proposition \ref{2deg1tau11} (2) with $\lambda\in\Lambda^\leqslant_{\lambda(\tilde{w})}\cup\mathbb{O}^{(\tilde{w})};$
\item $\tilde{w}$ or $\tau\tilde{w}\tau=t^{m\Lambda_2+n(\Lambda_2-\Lambda_1)}s_1\tau$ where $m,\ n\in\mathbb{N}_+$ and $\lambda\in\Lambda^\leqslant_{m\Lambda_2+n\Lambda_1}$;
\item $\tilde{w}$ or $\tau\tilde{w}\tau=t^{m\Lambda_2}s_1\tau$ where $m\in\mathbb{N}_+$ and $\lambda\in\Lambda^\leqslant_{m\Lambda_2}\cup\Lambda^{<,1}_{(m-1)\Lambda_2+2\Lambda_1}\cup\Lambda^{\leqslant,2}_{(m-1)\Lambda_2+\Lambda_1}$.
\end{enumerate}
If $X_{\tilde{w}}(b)\neq\emptyset$, then $$\dim X_{\tilde{w}}(b)=\frac{1}{2}(\ell(\tilde{w})+\ell(\mathbb{O}_{\lambda,\tau})+\deg (f_{\tilde{w},\mathbb{O}_{\lambda,\tau}}))-\langle\bar{\nu}_b,\ 2\rho\rangle.$$

\item\label{2pattern11} If $b$ corresponds to $\mathbb{O}_{\lambda,\tau}$, where $\lambda=(m'+1)\Lambda_2+n'\Lambda_2,\ m'\in\mathbb{N},\ n'\in\mathbb{N}_+$. Then $X_{\tilde{w}}(b)\neq\emptyset$ if and only if
\begin{enumerate}
\item $\tilde{w}\in\mathbb{O}_{\lambda,\tau}$;
\item $\tilde{w}$ or $\tau\tilde{w}\tau^{-1}$ is in Proposition \ref{2degstau1} such that $\lambda\in\lambda(\tilde{w});$
\item $\tilde{w}$ is in Proposition \ref{deg4k+2tau} (1) such that $\lambda\in\{(k+1)\Lambda_1+j(\Lambda_2-\Lambda_1)\ |\ 0\leqslant j\leqslant i\};$
\item $\tilde{w}$ is in Proposition \ref{deg4k+2tau} (2) such that $\lambda\in\Lambda^\leqslant_{\lambda(\tilde{w})}$ or $\lambda=(k+1)\Lambda_2+m\Lambda_1;$
\item $\tilde{w}$ is in Proposition \ref{deg0tau} such that $\lambda\in\Lambda^\leqslant_{\lambda(\tilde{w})};$
\item $\tilde{w}$ is in Proposition \ref{deg6ktau} (1) such that $\lambda\in\{\Lambda_1+k\Lambda_1+i(2\Lambda_1-\Lambda_2)\ |\ 0\leqslant i\leqslant m-1\};$
\item $\tilde{w}$ is in Proposition \ref{deg6ktau} (2) such that $\lambda\in\Lambda^\leqslant_{\lambda(\tilde{w})}$ or $\lambda=2k\Lambda_1+m\Lambda_2;$
\item $\tilde{w}$ is in Proposition \ref{deg6k+4tau} (1) such that $\lambda\in\{(k+1)\Lambda_2+i(2\Lambda_1-\Lambda_2)\ |\ 0\leqslant i\leqslant m\};$
\item $\tilde{w}$ is in Proposition \ref{deg6k+4tau} (2) such that $\lambda\in\Lambda^\leqslant_{\lambda(\tilde{w})}$ or $\lambda=(2k+1)\Lambda_1+m\Lambda_2;$
\item $\tilde{w}$ or $\tau\tilde{w}\tau^{-1}$ is in Proposition \ref{2deg1tau11} (1) with $\lambda\in\Lambda^<_{\lambda(\tilde{w})};$
\item $\tilde{w}$ or $\tau\tilde{w}\tau^{-1}$ is in Proposition \ref{2deg1tau11} (2) with $\lambda\in\Lambda^\leqslant_{\lambda(\tilde{w})}\cup\mathbb{O}^{(\tilde{w})};$
\item $\tilde{w}$ or $\tau\tilde{w}\tau=t^{m\Lambda_2+n(\Lambda_2-\Lambda_1)}s_1\tau$ where $m,\ n\in\mathbb{N}_+$ and $\lambda\in\Lambda^\leqslant_{m\Lambda_2+n\Lambda_1}$;
\item $\tilde{w}$ or $\tau\tilde{w}\tau=t^{m\Lambda_2}s_1\tau$ where $m\in\mathbb{N}_+$ and $\lambda\in\Lambda^\leqslant_{m\Lambda_2}\cup\Lambda^{<,1}_{(m-1)\Lambda_2+2\Lambda_1}\cup\Lambda^{\leqslant,2}_{(m-1)\Lambda_2+\Lambda_1}$.
\end{enumerate}
If $X_{\tilde{w}}(b)\neq\emptyset$, then $$\dim X_{\tilde{w}}(b)=\frac{1}{2}(\ell(\tilde{w})+\ell(\mathbb{O}_{\lambda,\tau})+\deg (f_{\tilde{w},\mathbb{O}_{\lambda,\tau}}))-\langle\bar{\nu}_b,\ 2\rho\rangle.$$

\item \label{2pattern12} If $b$ corresponds to $\mathbb{O}_{4k_0+2,\tau}$, where $k_0\in\mathbb{N}$. Then $X_{\tilde{w}}(b)\neq\emptyset$ if and only if
\begin{enumerate}
\item $\tilde{w}$ or $\tau\tilde{w}\tau^{-1}$ is in Proposition \ref{2degstau1} such that $k_0\leqslant n(\tilde{w});$
\item $\tilde{w}\in\mathbb{O}_{4k_0+2,\tau};$
\item $\tilde{w}$ is in Proposition \ref{deg4k+2tau} with $k\neq k_0$ and $k_0\leqslant n(\tilde{w});$
\item $\tilde{w}$ is in Proposition \ref{deg0tau}, Proposition \ref{deg6ktau} or Proposition \ref{deg6k+4tau} with $k_0\leqslant n(\tilde{w});$
\item $\tilde{w}$ or $\tau\tilde{w}\tau^{-1}$ is in Proposition \ref{2deg1tau11} (1) with $k_0\leqslant n-1;$
\item $\tilde{w}$ or $\tau\tilde{w}\tau^{-1}$ is in Proposition \ref{2deg1tau11} (2) with $k_0\leqslant n(\tilde{w})$ or $\mathbb{O}_{4k_0+2,\tau}\in\mathbb{O}_{(\tilde{w})};$
\item $\tilde{w}$ or $\tau\tilde{w}\tau=t^{m\Lambda_2+n(\Lambda_2-\Lambda_1)}s_1\tau$ where $m,\ n\in\mathbb{N}_+$ and $k_0\leqslant m+n-1$;
\item $\tilde{w}$ or $\tau\tilde{w}\tau=t^{m\Lambda_2}s_1\tau$ where $m\in\mathbb{N}_+$ and $k_0\leqslant m-1$.
\end{enumerate}
If $X_{\tilde{w}}(b)\neq\emptyset$, then $$\dim X_{\tilde{w}}(b)=\frac{1}{2}(\ell(\tilde{w})+\ell(\mathbb{O}_{4k_0+2,\tau})+\deg (f_{\tilde{w},\mathbb{O}_{4k_0+2,\tau}}))-\langle\bar{\nu}_b,\ 2\rho\rangle.$$

\item \label{2pattern13} If $b$ corresponds to $\mathbb{O}'_{6k_0,\tau}$, where $k_0\in\mathbb{N}$. Then $X_{\tilde{w}}(b)\neq\emptyset$ if and only if
\begin{enumerate}
\item $\tilde{w}$ or $\tau\tilde{w}\tau^{-1}$ is in Proposition \ref{2degstau1} or Proposition \ref{deg4k+2tau} such that $k_0\leqslant n'(\tilde{w});$
\item $\tilde{w}$ is in Proposition \ref{deg0tau} with $k_0\leqslant m-1;$
\item $\tilde{w}\in\mathbb{O}'_{6k_0,\tau};$
\item $\tilde{w}$ is in Proposition \ref{deg6ktau} (2) and $k\neq k_0$ with $k_0\leqslant n'(\tilde{w});$
\item $\tilde{w}$ is in Proposition \ref{deg6k+4tau} (2) with $k_0\leqslant m+k;$
\item $\tilde{w}$ or $\tau\tilde{w}\tau^{-1}$ is in Proposition \ref{2deg1tau11} (1) with $k_0\leqslant n'(\tilde{w});$
\item $\tilde{w}$ or $\tau\tilde{w}\tau^{-1}$ is in Proposition \ref{2deg1tau11} (2) with $k_0\leqslant n'(\tilde{w})$ or $\mathbb{O}'_{6k_0,\tau}\in\mathbb{O}_{(\tilde{w})};$
\item $\tilde{w}$ or $\tau\tilde{w}\tau=t^{m\Lambda_2+n(\Lambda_2-\Lambda_1)}s_1\tau$ where $m,\ n\in\mathbb{N}_+$ and $k_0\leqslant m+\lfloor\frac{n-1}{2}\rfloor$;
\item $\tilde{w}$ or $\tau\tilde{w}\tau=t^{m\Lambda_2}s_1\tau$ where $m\in\mathbb{N}_+$ and $k_0\leqslant m-2$.
\end{enumerate}
If $X_{\tilde{w}}(b)\neq\emptyset$, then $$\dim X_{\tilde{w}}(b)=\frac{1}{2}(\ell(\tilde{w}+\ell(\mathbb{O}'_{6k_0,\tau})+\deg (f_{\tilde{w},\mathbb{O}'_{6k_0,\tau}}))-\langle\bar{\nu}_b,\ 2\rho\rangle.$$

\item \label{2pattern14} If $b$ corresponds to $\mathbb{O}'_{6k_0+4,\tau}$, where $k\in\mathbb{N}$. Then $X_{\tilde{w}}(b)\neq\emptyset$ if and only if
\begin{enumerate}
\item $\tilde{w}\in\mathbb{O}_{(k_0+1)\Lambda_2,\tau}$ or $\mathbb{O}'_{6k_0+4,\tau};$
\item $\tilde{w}$ or $\tau\tilde{w}\tau^{-1}$ is in Proposition \ref{2degstau1} such that $k_0\leqslant n''(\tilde{w});$
\item $\tilde{w}$ is in Proposition \ref{deg4k+2tau} with $k_0\neq k$ and $k_0\leqslant n''(\tilde{w});$
\item $\tilde{w}\in\mathbb{O}_{4k_0+2,\tau}$ with $\ell(\tilde{w})\geqslant6k_0+4;$
\item $\tilde{w}$ is in Proposition \ref{deg0tau} with $k_0\leqslant n''(\tilde{w}),$ or $\mathbb{O}_{(k_0+1)\Lambda_2,\tau}\in\mathbb{O}(\tilde{w});$
\item $\tilde{w}$ is in Proposition \ref{deg6ktau} (2) with $k_0\leqslant n''(\tilde{w});$
\item $\tilde{w}$ is in Proposition \ref{deg6k+4tau} (2) with $k_0\neq k$ and $k_0\leqslant m+k-1;$
\item $\tilde{w}$ or $\tau\tilde{w}\tau^{-1}$ is in Proposition \ref{2deg1tau11} (1) with $k_0\leqslant n''(\tilde{w});$
\item $\tilde{w}$ or $\tau\tilde{w}\tau^{-1}$ is in Proposition \ref{2deg1tau11} (2) with $k_0\leqslant n'(\tilde{w})$ or $\mathbb{O}'_{6k_0+4,\tau}\in\mathbb{O}_{(\tilde{w})},$ or $\mathbb{O}_{(k_0+1)\Lambda_2,\tau}\in\mathbb{O}^{(\tilde{w})};$
\item $\tilde{w}$ or $\tau\tilde{w}\tau=t^{m\Lambda_2+n(\Lambda_2-\Lambda_1)}s_1\tau$ where $m,\ n\in\mathbb{N}_+$ and $k_0\leqslant m+\lfloor\frac{n}{2}\rfloor-1$;
\item $\tilde{w}$ or $\tau\tilde{w}\tau=t^{m\Lambda_2}s_1\tau$ where $m\in\mathbb{N}_+$ and $k_0\leqslant m$.
\end{enumerate}
If $X_{\tilde{w}}(b)\neq\emptyset$, then $$\dim X_{\tilde{w}}(b)=\frac{1}{2}\max_{\mathbb{O}}(\ell(\tilde{w})+\ell(\mathbb{O})+\deg (f_{\tilde{w},\mathbb{O}}))-\langle\bar{\nu}_b,\ 2\rho\rangle,$$where $\mathbb{O}\in\{\mathbb{O}'_{6k_0+4,\tau},\ \mathbb{O}_{(k_0+1)\Lambda_2}\}$.

\end{enumerate}

\subsection{A conjecture of G\"{o}rtz-Haines-Kottwitz-Reuman}
When $G=Sp_4(L)$, we can check that a conjecture of G\"{o}rtz-Haines-Kottwitz-Reuman is true.

Let $G(L)$ be as in $\S\ \ref{Sect2}$, for any $b\in G(L)$, we set $J_b$ as the $\sigma$-centralizer of $b$, that is $J_b=\{g\in G(L)|gb\sigma(g)^{-1}=b\}$. The defect of $b$, denoted by $\text{def}(b)$, is the $F$-rank of $G$ minus the $F$-rank of $J_b$. It was conjectured in \cite{GHKR10} that
\begin{conj}\label{conj1}
Let $b\in G(L)$ and $b'$ be a basic element in $G(L)$ such that $\kappa(b)=\kappa(b')$. Then for $\tilde{w}\in\widetilde{W}$ with sufficiently large length, $X_{\tilde{w}}(b)\neq\emptyset$ if and only if $X_{\tilde{w}}(b')\neq\emptyset$. In this situation, $$\dim X_{\tilde{w}}(b)=\dim X_{\tilde{w}}(b')-\langle\nu_b,\rho\rangle+\frac{1}{2}(\emph{def}(b')-\emph{def}(b)).$$
\end{conj}
\begin{thm}
The conjecture \emph{\ref{conj1}} is true for the group $\mathrm{Sp}_4(L)$.
\end{thm}
\begin{proof}
It follows from $\S\ \ref{App1}$ directly.
\end{proof}

\section{Acknowledgement}
The author would like to thank Professor Xuhua He for his constant encouragement and many helpful discussions. The author was partially supported by the National Natural Science Foundation of China (Grant No. 11701473 and 12171397), the Fundamental Research Funds for the Central Universities (Grant No. 2682019CX52), and  The ``1000 Talent Plan'' of Sichuan province.

\end{document}